\documentclass{amsart}
\usepackage[toc,page]{appendix}
\usepackage[headheight=12.1pt]{geometry}
\usepackage[all,cmtip]{xy}
\usepackage{amsmath}
\usepackage{amsfonts}
\usepackage{amssymb}
\usepackage{mathrsfs}   
\usepackage{amsthm}
\usepackage{xcolor}
\usepackage{cite}
\usepackage{stmaryrd}
\usepackage{graphicx}
\usepackage{pgf}
\usepackage{eepic}
\usepackage{pstricks}
\usepackage{epsfig}
\usepackage{fancyhdr}
\usepackage{tikz,caption,float}
\usepackage[backref=page]{hyperref}

\usepackage{amsbsy}

\numberwithin{equation}{subsection}

\let\realequation\equation
\def\equation{\setcounter{equation}{\arabic{subsubsection}}%
   \refstepcounter{subsubsection}%
   \realequation}

\usepackage{xr,mathtools}

\hypersetup{
     colorlinks   = true,
     citecolor    = red,
     linkcolor = blue,
     urlcolor = purple
}
    
  \newcommand\imCMsym[4][\mathord]{%
  \DeclareFontFamily{U} {#2}{}
  \DeclareFontShape{U}{#2}{m}{n}{
    <-6> #25
    <6-7> #26
    <7-8> #27
    <8-9> #28
    <9-10> #29
    <10-12> #210
    <12-> #212}{}
  \DeclareSymbolFont{CM#2} {U} {#2}{m}{n}
  \DeclareMathSymbol{#4}{#1}{CM#2}{#3}
}
\newcommand\alsoimCMsym[4][\mathord]{\DeclareMathSymbol{#4}{#1}{CM#2}{#3}}

\imCMsym{cmmi}{124}{\CMjmath}
\imCMsym[\mathop]{cmsy}{113}{\CMamalg}
\imCMsym[\mathop]{cmex}{96}{\CMcoprod}
\alsoimCMsym[\mathop]{cmex}{97}{\CMbigcoprod}

\swapnumbers
\theoremstyle{plain}
\newtheorem*{theoremu}{Theorem}
\newtheorem{theorem}[subsubsection]{Theorem}

\newtheorem{proposition}[subsubsection]{Proposition}

\newtheorem{corollary}[subsubsection]{Corollary}

\newtheorem*{corollaryu}{Corollary}
\newtheorem{lemma}[subsubsection]{Lemma}

\newtheorem{conjecture}[subsubsection]{Conjecture}

\theoremstyle{definition}

\newtheorem{definition}[subsubsection]{Definition}

\newtheorem{convention}[subsubsection]{Convention}
\newtheorem{notation}[subsubsection]{Notation}

\newtheorem{setup}[subsubsection]{Setup}

\theoremstyle{remark}
\newtheorem{remark}[subsubsection]{Remark}

\newtheorem{example}[subsubsection]{Example}

\newtheorem*{claimu}{Claim}
\newtheorem{warning}[subsubsection]{Warning}

\newcommand{\N}{{\mathbb N}}
\newcommand{\Z}{{\mathbb Z}}
\newcommand{\Q}{{\mathbb Q}}
\newcommand{\R}{{\mathbb R}}
\newcommand{\C}{{\mathbb C}}

\newcommand{\B}{{\mathbb B}}
\newcommand{\G}{{\mathbb G}}
\renewcommand{\P}{{\mathbb P}}
\newcommand{\A}{{\mathbb A}}
\newcommand{\D}{{\mathbb D}}

\newcommand{\sD}{\mathscr{D}}

\newcommand{\cO}{\mathcal{O}}

\newcommand{\cM}{\mathcal{M}}

\newcommand{\cV}{\mathcal{V}}

\newcommand{\fP}{\mathfrak{P}}

\newcommand{\fQ}{\mathfrak{Q}}
\newcommand{\fr}[1]{\mathfrak{#1}}

\newcommand{\bD}{{\bf D}}

\renewcommand{\lim}[1]{\underset{#1}{\mathrm{lim}}\,}

\newcommand{\isomto}{\overset{\cong}{\longrightarrow}}

\newcommand{\from}{\colon}
\renewcommand{\to}{\rightarrow}

\newcommand{\hol}{\mathrm{hol}}

\newcommand{\coh}{\mathrm{coh}}

\newcommand{\weak}[1]{\langle #1\rangle^\dagger}
\newcommand{\tate}[1]{\langle #1 \rangle}

\newcommand{\spec}[1]{\mathrm{Spec}\left(#1\right)}
\newcommand{\spa}[1]{\mathrm{Spa}\left(#1\right)}
\newcommand{\spf}[1]{\mathrm{Spf}\left(#1\right)}

\newcommand{\Norm}[1]{\left\Vert #1\right\Vert}
\newcommand{\norm}[1]{\left\vert#1\right\vert}
\newcommand{\tube}[1]{\left]#1\right[}

\newcommand{\isoc}[1]{\mathrm{Isoc}^\dagger(#1)}

\newcommand{\rig}{\mathrm{rig}}
\newcommand{\dR}{\mathrm{dR}}

\newcommand{\et}{\mathrm{\acute{e}t}}

\newcommand{\Tr}{\mathrm{Tr}}

\usepackage{euscript}
\newcommand{\EU}[1]{\EuScript{#1}}

\title{A comparison between compactly supported rigid and $\pmb{\mathscr{D}}$-module cohomology}

\usepackage{amsbsy}
\newcounter{saveenum}

\SetSymbolFont{stmry}{bold}{U}{stmry}{m}{n}
\usepackage{bm}

\author{Tomoyuki Abe}
       \address{Kavli Institute for the Physics and Mathematics of the Universe (WPI) \\ University of Tokyo
 \\ 5-1-5 Kashiwanoha \\ Kashiwa \\ Chiba \\ 277-8583 \\ Japan}
       \email{tomoyuki.abe@ipmu.jp}

\author{Christopher Lazda}
       \address{Department of Mathematics\\ Harrison Building \\ Streatham Campus
 \\ University of Exeter \\ North Park Road \\ Exeter  \\ EX4 4QF \\ United Kingdom }
       \email{c.d.lazda@exeter.ac.uk}
            
\begin{document}

\begin{abstract} The goal of this article is to prove a comparison theorem between rigid cohomology and cohomology computed using the theory of arithmetic $\mathscr{D}$-modules. To do this, we construct a specialisation functor from Le Stum's category of constructible isocrystals to the derived category of arithmetic $\mathscr{D}$-modules. For objects `of Frobenius type', we show that the essential image of this functor consists of overholonomic $\mathscr{D}^\dagger$-modules, and lies inside the heart of the dual constructible t-structure. We use this to give a more global construction of Caro's specialisation functor $\mathrm{sp}_+$ for overconvergent isocrystals, which enables us to prove the comparison theorem for compactly supported cohomology.
\end{abstract}

\maketitle 

\tableofcontents

\section*{Introduction}

Berthelot's theory of rigid cohomology \cite{Ber96b} was introduced to provide a $p$-adic Weil cohomology theory for varieties in characteristic $p$, enjoying the same formal properties as $\ell$-adic \'etale cohomology for $\ell\neq p$. It is constructed as a version of de\thinspace Rham cohomology, and its local systems of coefficients, overconvergent isocrystals, are therefore analogues of vector bundles with integrable connection in characteristic $0$. 

To obtain a good formalism of cohomological operations, it is necessary to go beyond categories of local systems, and work more generally with constructible coefficient objects. For algebraic de\thinspace Rham cohomology (in characteristic $0$) this leads to the study of regular holonomic $\mathscr{D}$-modules, and, inspired by this, Berthelot introduced in \cite{Ber02} a theory of \emph{arithmetic} $\mathscr{D}$-modules, which should play the same role in rigid cohomology as that played by constructible sheaves in $\ell$-adic \'etale cohomology. At least for objects admitting a Frobenius structure, it was proved by Caro--Tsuzuki in \cite{CT12} that the category of overholonomic $\mathscr{D}^\dagger$-modules does indeed support a formalism of Grothendieck's 6 operations.

Of course, in order to know that the theory of arithmetic $\mathscr{D}$-modules and their cohomological operations really does generalise the theory of overcovergent isocrystals and rigid cohomology, it is important to know that an overconvergent isocrystal can be viewed as a particular kind of arithmetic $\mathscr{D}$-module. For smooth varieties, Caro constructed in \cite{Car09a} a functor $\mathrm{sp}_+$ from overconvergent isocrystals to arithmetic $\mathscr{D}$-modules, and proved almost of all its expected properties, in particular compatibility with pullback and duality. It still remains open, however, whether or not the rigid cohomology of some overconvergent isocrystal $E$ can be computed in terms of the associated $\mathscr{D}^\dagger$-module $\mathrm{sp}_+E$. One of the main reasons that this question is so hard is the local nature of the construction of $\mathrm{sp}_+$, which makes \emph{global} comparisons (such as are necessary to relate the two cohomologies) difficult. Our main result in this article is to achieve this comparison for compactly supported cohomology.

A different framework in which to try to understand the relationship between overconvergent isocrystals and arithmetic $\mathscr{D}$-modules is provided by Le Stum's theory of constructible isocrystals \cite{LS14,LS16}. Let $\mathcal{V}$ be a complete discrete valuation ring, of mixed characteristic, with fraction field $K$ and perfect residue field $k$. Fix a smooth formal $\mathcal{V}$-scheme $\mathfrak{P}$, with generic fibre $\mathfrak{P}_K$. Then Le Stum's category of constructible isocrystals on $\mathfrak{P}$ is a different kind of rigid cohomological analogue of the category of constructible $\ell$-adic sheaves, and consists of convergent $\nabla$-modules on $\mathfrak{P}_K$, which become locally free on the tubes of each stratum in some stratification of $\fr{P}$. The expectation is then that there should be a functor
\[ \mathbf{R}\mathrm{sp}_*\colon \mathrm{Isoc}_\mathrm{cons}(\mathfrak{P}) \rightarrow {\bf D}^b_\mathrm{coh}(\mathscr{D}^\dagger_{\mathfrak{P}\Q}) \]
which should recover Caro's functor $\mathrm{sp}_+$ on those constructible isocrystals which are overconvergent isocrystals supported on locally closed subschemes of the special fibre $P$. In fact, as we shall see, it will recover $\mathrm{sp}_+$ up to shifts and duality. One of the advantages of this point of view is that it makes comparisons between de\thinspace Rham cohomology on the rigid and formal sides extremely simple, since it is more or less immediate that
\[ {\rm H}^i(\fr{P}_K, \mathscr{F} \otimes \Omega^\bullet_{\fr{P}_K}) \cong {\rm H}^i(\fr{P} , \mathbf{R}\mathrm{sp}_*\mathscr{F} \otimes \Omega^\bullet_{\fr{P}}) .\]
for any $\nabla$-module $\mathscr{F}$ on $\fr{P}_K$. However, the construction of such a functor $\mathbf{R}\mathrm{sp}_*$ is rather more subtle than it may first appear. One can easily define a functor
\[ \mathbf{R}\mathrm{sp}_*\colon \mathrm{Isoc}_\mathrm{cons}(\mathfrak{P}) \rightarrow {\bf D}^b(\mathscr{D}_{\mathfrak{P}\Q}) \]
simply as the derived functor of $\mathrm{sp}_*$, but lifting it to ${\bf D}^b(\mathscr{D}^\dagger_{\mathfrak{P}\Q})$ is quite a delicate problem.

Our first goal in this article is to explain how to construct such a lifting. The strategy is to localise $\mathbf{R}\mathrm{sp}_*$ along a stratification of $\fr{P}$ over which a given constructible isocrystal $\mathscr{F}$ becomes locally free. This allows us to define certain explicit $\mathrm{sp}_*$-acyclic resolutions of $\mathscr{F}$ via which we can compute $\mathbf{R}\mathrm{sp}_*\mathscr{F}$. We then show that these resolutions admit what we call a `pre-$\mathscr{D}^\dagger$-module structure', which enables us to put a $\mathscr{D}_{\mathfrak{P}\Q}^\dagger$-module structure on $\mathbf{R}\mathrm{sp}_*\mathscr{F}$ extending the natural $\mathscr{D}_{\mathfrak{P}\Q}$-module structure. This then gives rise to a functor
\[ \mathrm{sp}_{!} := \mathbf{R}\mathrm{sp}_* [\dim \fr{P}] \colon \mathrm{Isoc}_\mathrm{cons}(\mathfrak{P})\rightarrow {\bf D}^b(\mathscr{D}^\dagger_{\fr{P}\Q}) \]
lifting the $\mathscr{D}_{\fr{P}\Q}$-linear version. The shift by the dimension of $\fr{P}$ is necessary in order ensure that the functor has the correct essential image. Interestingly enough, this image turns out \emph{not} to be contained in the direct $\mathscr{D}^\dagger$-module analogue of the category of constructible sheaves, but instead in the \emph{dual} of this category. We refer to \S\ref{sec: finite} for a detailed discussion of what we call dual constructible $\mathscr{D}^\dagger$-modules, but roughly speaking these should be thought of as a $\mathscr{D}^\dagger$-module analogue of complexes of $\ell$-adic sheaves whose Verdier dual is a constructible sheaf concentrated in degree zero. 

To state our first main result on $\mathrm{sp}_!$, we say that a constructible isocrystal is of Frobenius type if it is an iterated extension of objects admitting some Frobenius structure.\footnote{This has previously been called `$F$-able' in the literature.} We also call a $k$-variety or formal $\cV$-scheme `realisable' if it admits a locally closed immersion into a smooth and proper formal $\cV$-scheme. 

\begin{theoremu}[\ref{theo: DCon}] Let $\fr{P}$ be a realisable smooth formal $\mathcal{V}$-scheme, and $\mathscr{F}\in \mathrm{Isoc}_\mathrm{cons}(\fr{P})$ a constructible isocrystal of Frobenius type. Then $\mathrm{sp}_{!}\mathscr{F}$ is a dual constructible complex of overholonomic $\mathscr{D}^\dagger_{\fr{P}\Q}$-modules. If $\mathscr{F}$ is supported on some locally closed subscheme $X\hookrightarrow P$, then so is $\mathrm{sp}_!\mathscr{F}$.  
\end{theoremu}

The strategy here is to reduce to the fundamental result of Caro--Tsuzuki \cite{CT12}. We first reduce to the case of a (partially) overconvergent isocrystal supported on a locally closed subscheme $X\hookrightarrow\fr{P}$, and using de\thinspace Jong's theory of alterations we can assume that there is a diagram of frames
\[ \xymatrix{ X'\ar[r] \ar[d] & Y' \ar[r]\ar[d] & \fr{P}'\ar[d] \\ X\ar[r]& Y \ar[r] & \fr{P}  }\]
with $X'\rightarrow X$ finite \'etale, $Y'$ smooth over $k$ and proper over $Y$, and $\fr{P}'\rightarrow\fr{P}$ smooth and projective. The result of Caro--Tsuzuki implies that the overholonomicity of $\mathrm{sp}_!$ holds for the pullback of $\mathscr{F}$ to $X'$, and the key point is then to compare the constructions of $\mathrm{sp}_!$ on $\fr{P}$ and $\fr{P}'$. This requires a rather delicate verification of the compatibility of the functor $\mathrm{sp}_!$ with the rigid analytic trace map defined in \cite{AL20}. 

We then prove the compatibility of $\mathrm{sp}_!$ with arbitrary pullback, as well as pushforward along a locally closed immersion. This explains the slightly odd-looking exactness and adjunction properties of the various pullback and pushforward functors for constructible isocrystals defined in \cite{LS16}: they are precisely dual to the familiar ones encountered for constructible sheaves. It also seems reasonable to conjecture that $\mathrm{sp}_!$ is actually an equivalence between $\mathrm{Isoc}_\mathrm{cons}(\fr{P})$ and the category of dual constructible $\mathscr{D}^\dagger_{\fr{P}\Q}$-modules (at least for objects of Frobenius type), as was predicted by Le Stum. We intend to return to this question in future work.

By embedding a $k$-variety inside a smooth and proper formal $\mathcal{V}$-scheme, this gives rise to a functor
\[ \mathrm{sp}_!\colon \mathrm{Isoc}^\dagger(X) \rightarrow {\bf D}^b_\mathrm{hol}(X), \]
from the category of overconvergent isocrystals on $X$ (of Frobenius type), to the category of overholonomic complexes on $X$ in the sense of \cite{AC18a}. We can then show that this recovers the dual of Caro's functor $\mathrm{sp}_+$.

\begin{theoremu}[\ref{theo: sp_+ comp}] Let $X$ be a realisable smooth $k$-variety, and $E$ an overconvergent isocrystal of Frobenius type on $X$. Then there is a natural isomorphism
\[  \mathrm{sp}_+E \cong \mathbf{D}_X \mathrm{sp}_!E^\vee \,[\dim X] \]
in ${\bf D}^b_\mathrm{hol}(X)$.
\end{theoremu}

Here $\mathbf{D}_X$ is the duality functor on ${\bf D}^b_\mathrm{hol}(X)$. In fact, thanks to a descent construction of the first author \cite{Abe19}, there is no need to restrict to smooth varieties here, since the functor $\mathrm{sp}_+[-\dim X]$ extends to a fully faithful functor
\[ \rho\colon \mathrm{Isoc}^\dagger(X)\rightarrow {\bf D}^b_\mathrm{hol}(X)   \]
defined for all $k$-varieties, whose essential image can be explicitly described. Then Theorem \ref{theo: sp_+ comp} still holds with $\rho[\dim X]$ in place of $\mathrm{sp}_+$.\footnote{In fact, in \S\ref{sec: comp} we redefine $\mathrm{sp}_+$ to be the shifted version of Caro's functor, which then itself extends to non-smooth varieties; we do not use the notation $\rho$ at all outside this introduction.} Given this descent result, and given compatibility with pullbacks, the proof of Theorem \ref{theo: sp_+ comp} is actually relatively straightforward. The point is that the objects involved lie in an abelian category which satisfies h-descent. Thus we may construct such an isomorphism locally for the h-topology on $X$. We may therefore again use alterations to reduce to the case when $X$ has a smooth compactification, which then locally lifts to a smooth formal $\mathcal{V}$-scheme. It is then possible to compare $\mathrm{sp}_!$ with $\mathrm{sp}_+$ more or less directly. 

Finally, we can use the fact that our functor $\mathrm{sp}_!$ has a genuinely global definition to prove a comparison theorem for compactly supported cohomology. For an overholonomic complex $E\in {\bf D}^b_\mathrm{hol}(X)$, we define its compactly supported cohomology to be
\[ {\rm H}^*_{c,\mathscr{D}}(X,E) := {\rm H}^*(f_!E)\]
where $f\colon X \rightarrow \spec{k}$ is the structure morphism. Similarly, we define its (ordinary) cohomology to be
\[ {\rm H}^*_{\mathscr{D}}(X,E) := {\rm H}^*(f_+E).\]
We say that a $k$-variety is realisable if it embeds inside a smooth and proper formal $\mathcal{V}$-scheme.

\begin{theoremu}[\ref{theo: cohom comp 1}] Let $X/k$ be a realisable $k$-variety, and $E$ an overconvergent isocrystal on $X$ of Frobenius type. Then there is a canonical isomorphism
\[  {\rm H}^{*}_{c,\rig} (X,E) \cong {\rm H}^*_{c,\mathscr{D}}(X,\rho(E))   \]
of finite dimensional, graded $K$-vector spaces.
\end{theoremu}

Given an embedding of $i_X:X\hookrightarrow \fr{P}$ into a smooth and proper formal $\mathcal{V}$-scheme, the proof of this ultimately boils down to computing the de\thinspace Rham cohomology of the  constructible $\nabla$-module $i_{X!}E$ on $\fr{P}_K$. A key ingredient in achieving this is again the trace morphism in rigid analytic geometry defined in \cite{AL20}, which enables us to interpret the de\thinspace Rham cohomology of $i_{X!}E$ as the rigid Borel--Moore homology of $X$ with coefficients in $E$. Of course, Theorem \ref{theo: cohom comp 1} together with Poincar\'e duality, implies a similar result in ordinary cohomology. 

\begin{corollaryu}[\ref{cor: cohom comp 2}] Let $X$ be a realisable $k$-variety, which is \emph{either} smooth \emph{or} proper, and $E\in \mathrm{Isoc}^\dagger(X)$ an overconvergent isocrystal of Frobenius type. Then there is a canonical isomorphism
\[  {\rm H}^*_{\rig}(X,E) \cong {\rm H}^*_{\mathscr{D}}(X,\rho(E))  \]
of finite dimensional, graded $K$-vector spaces.
\end{corollaryu}

Let us now give a summary of the contents of the article. In \S\ref{sec: prelim} we recall some of the basic constructions in rigid cohomology, making note of the changes that need to be made when working with adic spaces (as we do throughout). In \S\ref{sec: calc Rsp} we show how to calculate $\mathbf{R}\mathrm{sp}_*\mathscr{F}$ whenever $\mathscr{F}$ is a constructible $\mathcal{O}_{\fr{P}_K}$-module, via explicit complexes built out of restriction and pushforward along a suitable stratification of $\fr{P}$. In \S\ref{sec: pre Dd} we introduce the notion of a pre-$\mathscr{D}^\dagger$-module on the tube of some closed subset of $\fr{P}$, and show that the explicit complexes constructed in \S\ref{sec: calc Rsp} can be canonically upgraded to complexes of pre-$\mathscr{D}^\dagger$-modules. This enables us to define a $\mathscr{D}^\dagger$-linear version of $\mathbf{R}\mathrm{sp}_*$. In \S\ref{sec: trace} we recall the properties of the trace morphism constructed in \cite{AL20}, and use this to interpret de\thinspace Rham cohomology of certain constructible isocrystals as rigid Borel--Moore homology. In \S\ref{sec: trace Ddagger} 
we show compatibility of $\mathrm{sp}_!$ with the trace morphism, at least in some very special cases. This is a key ingredient of the eventual proof in \S\ref{sec: finite} that $\mathrm{sp}_!\mathscr{F}$ is overholonomic, and has the expected support. We also show that it lies in the heart of the dual constructible t-structure on ${\bf D}^b_\mathrm{hol}(\mathscr{D}^\dagger_{\fr{P}\Q})$, and is compatible with pullback, as well as pushforward along locally closed immersions. Finally in \S\ref{sec: comp} we prove the main comparison theorems described above.

\subsection*{Acknowledgements} To be added after the referee process.

\subsection*{Notation and convetions}

\begin{itemize}
\item A variety over a given locally Noetherian scheme $S$ is a separated $S$-scheme of finite type. 
\item We will denote by $K$ a complete, discretely valued field of characteristic $0$, whose residue field $k$ is perfect of characteristic $p$. While the general formalism of rigid cohomology does not require our non-archimedean field $K$ to be discretely valued, or $k$ to be perfect, the theory of arithmetic $\mathscr{D}^\dagger$-modules is generally developed under these assumptions.\footnote{Although the requirement that $k$ is perfect has recently been lifted by \cite{Car19}.} Since our main point of interest is comparisons between the two, we will impose this hypothesis from the beginning.
\item We will let $\mathcal{V}$ denote the ring of integers of $K$, $\mathfrak{m}$ its maximal ideal, and $\varpi$ a uniformiser. We will assume that $K$ admits a lift of the $q=p^a$ Frobenius on $k$, and `Frobenius' will always mean the $q$-power Frobenius.
\item A formal schemes will always mean a separated and (topologically) finite type formal scheme over $\spf{\mathcal{V}}$. Given a formal scheme $\fr{P}$, we will denote by $\fr{P}_K$ its generic fibre as an adic space. We will generally use fraktur letters to denote formal schemes, and the corresponding roman letters for their special fibres, e.g. $P=\fr{P}_k$.
\item  If $\rho\in\sqrt{\norm{K}}$ we will denote by $\D_K(0;\rho)$ the closed disc of radius $\rho$ over $K$, and by $\D_K(0;\rho^-)$ the open disc of radius $\rho$.
\item If $f:\mathscr{X}\rightarrow \mathscr{Y}$ is a smooth morphism of adic spaces (or schemes, or formal schemes \&c.), and $\mathscr{F}$ is an $\mathcal{O}_{\mathscr{X}}$-module with integrable connection relative to $\mathscr{Y}$, we will denote by $\mathbf{R}^qf_{\dR*}\mathscr{F}$ the higher direct images of the de\thinspace Rham complex $\Omega^\bullet_{\mathscr{X}/\mathscr{Y}} \otimes_{\mathcal{O}_{\mathscr{X}}} \mathscr{F}$.
\item For any additive category $\mathcal{A}$ we will denote by $\mathbf{Ch}^{\#}(\mathcal{A})$ the category of chain complexes in $\mathcal{A}$ with boundedness condition $\#\in \{\emptyset,+,-,b\}$. If $\mathcal{A}$ is abelian we denote by ${\bf D}^{\#}(\mathcal{A})$ the corresponding derived category. If $\mathcal{A}$ is the category of sheaves on a topological space $X$, we will usually write $\mathbf{Ch}^{\#}(X)$ and ${\bf D}^{\#}(X)$ instead, and if $\mathcal{A}$ is the category of $\mathcal{O}_X$-modules on a ringed space $(X,\mathcal{O}_X)$, we will write $\mathbf{Ch}^{\#}(\mathcal{O}_X)$ and ${\bf D}^{\#}(\mathcal{O}_X)$.
\end{itemize}

\section{Preliminaries} \label{sec: prelim}

In this section we will recall some of the basic results and constructions we will need from the various theories of adic spaces, rigid cohomology, and arithmetic $\mathscr{D}$-modules.

\subsection{Adic spaces} \label{sec: adic spaces}

For us, analytic varieties will always be considered as adic spaces. We will therefore let $\mathbf{An}_K$ denote the category of adic spaces separated, locally of finite type, and taut\footnote{that is, the closure of every quasi-compact open is quasi-compact} over $\spa{K,\mathcal{V}}$, and refer to such objects as \emph{analytic varieties} over $K$. By \cite[\S1.1.11]{Hub96} there is an equivalence of categories
\begin{align*}
 (-)_0\colon\mathbf{An}_K &\rightarrow \mathbf{Rig}_K \\
 \mathscr{X} &\mapsto \mathscr{X}_0
\end{align*}
between $\mathbf{An}_K$ and the category of separated and taut rigid analytic spaces over $K$ in the sense of Tate \cite{Tat71}. Denote a quasi-inverse to this functor by $r(-)$. If we let $\mathscr{X}_{\mathrm{an}}$ denote the analytic site of $\mathscr{X}$ (that is, the category of open subsets of $\mathscr{X}$ equipped with its canonical topology), and $\mathscr{X}_{0,G}$ the $G$-site of $\mathscr{X}_0$ (that is, the category of admissible opens equipped with the topology of admissible open coverings), then the functor
\begin{align*}
\mathscr{X}_{\mathrm{an}} &\leftarrow \mathscr{X}_{0,G} \\
r(U) &\mapsfrom U
\end{align*}
induces an equivalence
\begin{align*}
 \mathbf{Sh}(\mathscr{X}) &\isomto \mathbf{Sh}_G(\mathscr{X}_0) \\
 \mathscr{F} &\mapsto \mathscr{F}_0
\end{align*}
of toposes, functorial in $\mathscr{X}$ \cite[\S1.1.11]{Hub96}. Thus we have isomorphisms in cohomology
\[ {\rm H}^q(\mathscr{X},\mathscr{F}) \isomto {\rm H}^q(\mathscr{X}_0,\mathscr{F}_0)\]
for any abelian sheaf $\mathscr{F}$. Since $\left(\mathcal{O}_{\mathscr{X}}\right)_0\isomto \mathcal{O}_{\mathscr{X}_0}$, we also deduce an equivalence of categories
\[ \mathbf{Coh}(\mathscr{X}) \isomto \mathbf{Coh}(\mathscr{X}_0)\]
between coherent sheaves of $\mathscr{X}$ and $\mathscr{X}_0$.

\subsection{Germs of adic spaces} The following notion of a germ of an adic space was introduced in \cite{AL20}.

\begin{definition} \begin{enumerate}
\item A pre-germ is a pair $(S,\mathscr{X})$ where $\mathscr{X}$ is an adic space, and $S\subset \mathscr{X}$ is a closed subset. 
\item A morphism $f\colon(S,\mathscr{X})\rightarrow (T,\mathscr{Y})$ of pre-germs is a morphism $f\colon\mathscr{X}\rightarrow \mathscr{Y}$ of adic spaces such that $f(S)\subset T$.
\item A morphism $f\colon (S,\mathscr{X})\rightarrow (T,\mathscr{Y})$ of pre-germs is called a strict neighbourhood if $f\colon\mathscr{X}\hookrightarrow \mathscr{Y}$ is an open immersion inducing a homeomorphism $f\colon S\isomto T$.
\item The category of germs of adic spaces $\mathbf{Germ}$ is the localisation of the category of pre-germs at the class of strict neighbourhoods.
\end{enumerate} 
\end{definition}

We will denote by $\mathbf{Germ}_K$ the category of germs which are separated, locally of finite type, and taut over $\spa{K,\mathcal{V}}$, in the sense of \cite[\S2.2]{AL20}.

\begin{remark} The `separated' and `taut' conditions in this definition are essentially topological conditions on $S$, whereas the `locally of finite type' condition is one on the ambient adic space $\mathscr{X}$.
\end{remark}

We will often talk about morphisms of germs being smooth or partially proper, or other similar adjectives. This should always be understood in the sense of \cite[\S1.10]{Hub96}. For example, being smooth means that there exists a representative $f\colon(S,\mathscr{X}) \rightarrow (T,\mathscr{Y})$ at the level of pre-germs such that $f\colon \mathscr{X}\rightarrow \mathscr{Y}$ is smooth and $S=f^{-1}(T)$. (Partial properness is more involved, see \cite[Definition 1.10.15]{Hub96}.)

Any germ $S$ can be viewed as a ringed space by equipping it with the restriction
\[ \mathcal{O}_S:= \mathcal{O}_{\mathscr{X}}|_S\]
of the structure sheaf from its ambient adic space. Similarly, if $f\colon S\rightarrow T$ is a morphism in $\mathbf{Germ}_K$, we can define the relative de\thinspace Rham complex $\Omega^\bullet_{S/T}$ via restriction from the ambient adic space. Almost all the germs we consider in this article have the following property. 

\begin{definition} A germ $(S,\mathscr{X})$ is overconvergent if $S$ is closed under generalisation in $\mathscr{X}$.
\end{definition}

\subsection{Norms and valuations}

Write $\norm{\,\cdot\,}\colon K\rightarrow \R_{\geq0}$ for the given (multiplicative) discrete valuation on $K$, normalised so that $\norm{p}=p^{-1}$. This extends uniquely to any algebraic extension of $K$, or any completion thereof. If $x\in \mathscr{X}$ is a point of an analytic variety over $K$, let
\[ v_x\colon \mathcal{O}_{\mathscr{X}\hspace{-.7mm},x} \rightarrow \left\{ 0 \right\} \cup \Gamma_x \]
denote the corresponding valuation on the local ring at $x$, written multiplicatively. Write $[x]$ for the maximal generalisation of $x$ in the sense of \cite[\S1.1.9]{Hub96}, this is a rank $1$ valuation that we will implicitly view as taking values in $\R_{\geq 0}$, again normalising so that $v_{[x]}(p)=p^{-1}$. 

For any $\rho \in \sqrt{\norm{K}} \subset \R_{\geq0}$, say $\rho =\norm{\alpha}^{\frac{1}{n}}$ with $\alpha\in K$, we will abuse notation and write ``$v_x(f) \leq \rho$'' as shorthand for ``$v_x(\alpha^{-1}f^n) \leq 1$''. Since $v_x$ is multiplicative, this does not depend on the choice of $\alpha$ and $n$. We will employ a similar abuse of notation with $\leq$ replaced by $<, \geq$ or $>$. If $f\in \Gamma(\mathscr{X},\mathcal{O}_{\mathscr{X}})$, then
\[ \mathscr{X}(\rho^{-1}f):= \left\{ \left. x\in \mathscr{X} \right\vert v_x(f)\leq \rho  \right\} \]
is an open subset of $\mathscr{X}$, and is affinoid whenever $\mathscr{X}$ itself is affinoid. We will adopt the following convention throughout.

\begin{convention} \label{conv: real num} The radius of any disc considered in this article will be assumed to lie in $\sqrt{\norm{K}}\subset \R_{\geq0}$.
\end{convention}

If $\rho\in \sqrt{\norm{K}}\subset \R_{\geq0}$ is such a real number, and $R$ is a Banach $K$-algebra, we will write
\[ R\tate{\rho^{-1}T} := \left\{ \left. \sum_{i\geq0} r_iT^i \right\vert \Norm{r_i}\rho^i \rightarrow0 \right\}, \]
this is a strictly $K$-affinoid algebra whenever $R$ itself is a strictly $K$-affinoid algebra.

\subsection{Frames and tubes}

The basic objects of rigid cohomology are frames and tubes. Since the theory is generally phrased in the language of rigid analytic spaces, we will briefly discuss the changes that need to be made when using adic spaces instead.

\begin{definition} \begin{enumerate} 
\item A pair $(X,Y)$ over $k$ consists of an open immersion $X\overset{j}{\hookrightarrow} Y$ of $k$-varieties. 
\item A frame $(X,Y,\fr{P})$ over $\mathcal{V}$ consists of a pair $(X,Y)$ together with a closed immersion $Y\overset{i}{\hookrightarrow} \fr{P}$ of formal schemes over $\mathcal{V}$, such that $\fr{P}$ is flat over $\mathcal{V}$.
\end{enumerate}
There is the obvious notion of a morphism of pairs or frames.
\end{definition}

Let $(X,Y,\fr{P})$ be a frame, and consider the continuous specialisation map
\[ \mathrm{sp}\colon\fr{P}_K \rightarrow \fr{P}\cong P. \]
We let $[\fr{P}_K]$ denote the set of maximal points of $\fr{P}_K$, equipped with the quotient topology via the map
\[ \mathrm{sep} \colon \fr{P}_K\rightarrow [\fr{P}_K] \]
taking a point to its maximal generalisation. Then the restriction
\[ [\mathrm{sp}] \colon[\fr{P}_K]\rightarrow \fr{P} \]
of $\mathrm{sp}$ to $[\fr{P}_K]$ is anticontinuous, in the sense that preimages of open sets are closed an vice versa \cite[Chapter II, \S4.2,4.3]{FK18}. We may therefore define an anticontinuous map
\[ \widetilde{\mathrm{sp}}\colon \fr{P}_K\rightarrow \fr{P} \]
as the composition of $ \mathrm{sep}$ and $[\mathrm{sp}]$. For any constructible subset $T\subset P$ we define the tube
\[ \tube{T}_\fr{P}:= \widetilde{\mathrm{sp}}^{-1}(T). \]
Note that this will \emph{not} be an adic space in general (unless $T$ is closed in $P$), If $T$ is locally closed, however, then $\tube{T}_\fr{P}$ will be a germ, since if $\overline{T}$ is the closure of $T$ in $P$, then $\tube{T}_\fr{P}$ will be a closed subset of the adic space $\tube{\overline{T}}_\fr{P}$. If $a \colon T\hookrightarrow T'$ is an inclusion of constructible subsets of $P$, we will generally abuse notation and also write $a \colon \tube{T}_\fr{P}\rightarrow \tube{T'}_\fr{P}$ for the induced morphism of tubes.

Now let $i_{D}\colon D\hookrightarrow Y$ be a closed complement to $X$ in $Y$. Then the open immersion $j\colon X\hookrightarrow Y$ induces a morphism
\[ j \colon \tube{X}_\fr{P} \rightarrow \tube{Y}_\fr{P} \]
of germs. Topologically this is the inclusion of a closed subset, and is complementary to the open immersion $i_{D}\colon \tube{D}_\fr{P}\rightarrow \tube{Y}_\fr{P}$ of adic spaces. As a germ, the structure sheaf
\[ \mathcal{O}_{\tube{X}_\fr{P}} = \mathcal{O}_{\tube{Y}_\fr{P}}|_{\tube{X}_\fr{P}} \]
is the restriction of that on $\tube{Y}_\fr{P}$.

\begin{definition} Define endofunctors
\begin{align*}
j_X^\dagger \colon \mathbf{Sh}(\tube{Y}_\fr{P}) &\rightarrow \mathbf{Sh}(\tube{Y}_\fr{P}) \\
\underline{\Gamma}^\dagger_D \colon \mathbf{Sh}(\tube{Y}_\fr{P}) &\rightarrow \mathbf{Sh}(\tube{Y}_\fr{P})
\end{align*}
by $j_X^\dagger:=j_{*}j_X^{-1}$ and $\underline{\Gamma}^\dagger_D:=i_{D!}i_D^{-1}$.
\end{definition}

These are both exact, and there is an exact sequence
\[ 0\rightarrow \underline{\Gamma}^\dagger_D \rightarrow \mathrm{id} \rightarrow j_X^\dagger  \rightarrow 0 \]
of endofunctors of $\mathbf{Sh}(\tube{Y}_\fr{P})$. The following lemma is elementary.

\begin{lemma} The functors 
\[ j_{*}\colon \mathbf{Coh}(\mathcal{O}_{\tube{X}_\fr{P}}) \leftrightarrows \mathbf{Coh}(j_X^\dagger\mathcal{O}_{\tube{Y}_\fr{P}}) \colon j^{-1} \]
are inverse equivalences of categories, and induce isomorphisms
\[ {\rm H}^*(\tube{Y}_\fr{P},\mathscr{E}) \isomto {\rm H}^*(\tube{X}_\fr{P},j^{-1}\mathscr{E})\]
for any coherent $j_X^\dagger\mathcal{O}_{\tube{Y}_\fr{P}}$-module $\mathscr{E}$.
\end{lemma}

Let $\mathfrak{P}_{K0}$ denote the rigid analytic generic fibre of $\fr{P}$ in the sense of \cite[\S0.2]{Ber96b}, in the notation of \S\ref{sec: adic spaces} above this is the rigid analytic space $(\fr{P}_K)_0$. Let $\tube{Y}_{\fr{P}0}\subset \fr{P}_{K0}$ denote the rigid analytic tube in the sense of Berthelot \cite[\S1.1]{Ber96b}, and $j_X^\dagger$, $\underline{\Gamma}^\dagger_D$ the corresponding endofunctors of $\mathbf{Sh}_G(\tube{Y}_{\fr{P}0})$ as defined in \cite[\S2.1]{Ber96b}.

\begin{proposition} \label{prop: adic rig} There is a canonical isomorphism $\left(\tube{Y}_\fr{P}\right)_0 \isomto  \tube{Y}_{\fr{P}0}$ of rigid analytic spaces over $K$. Moreover, the diagrams
 \[ \xymatrix{ \mathbf{Sh}(\tube{Y}_\fr{P}) \ar[r]^-{(-)_0} \ar[d]_{j_X^\dagger} & \mathbf{Sh}(\tube{Y}_{\fr{P}0}) \ar[d]^{j_X^\dagger} & & \mathbf{Sh}(\tube{Y}_\fr{P}) \ar[r]^-{(-)_0} \ar[d]_{\underline{\Gamma}^\dagger_D} & \mathbf{Sh}(\tube{Y}_{\fr{P}0}) \ar[d]^{\underline{\Gamma}^\dagger_D}  \\ \mathbf{Sh}(\tube{Y}_\fr{P}) \ar[r]^-{(-)_0} & \mathbf{Sh}(\tube{Y}_{\fr{P}0}) & & \mathbf{Sh}(\tube{Y}_\fr{P}) \ar[r]^-{(-)_0} & \mathbf{Sh}(\tube{Y}_{\fr{P}0}) } \]
commute up to natural isomorphism.
\end{proposition}

\begin{proof}
Note that we can identify $\mathfrak{P}_{K0}$ with the set of \emph{rigid} points of $\mathfrak{P}_{K}$, and the functor $U\mapsto U\cap \mathfrak{P}_{K0}$ gives a one-to-one correspondence between admissible open subsets of $\mathfrak{P}_{K}$ (in the sense of \cite[Definition II.B.1.1]{FK18}) and admissible open subsets of $\mathfrak{P}_{K0}$ (in the sense of the $G$-topology). Since tube open subsets of $\fr{P}$ are admissible by \cite[Proposition II.B.1.7]{FK18}, the first claim is reduced to showing that $\tube{Y}_\fr{P} \cap \fr{P}_{K0}=\tube{Y}_{\fr{P}0}$ as subsets of $\fr{P}$. The question is now local on $\fr{P}$, which we may assume to be affine. Let $f_1,\ldots,f_r\in \Gamma(\fr{P},\mathcal{O}_\fr{P})$ be such that $Y=V(\varpi,f_1,\ldots,f_r)$. Then by \cite[Proposition II.4.2.11]{FK18} we can identify
 \[ \tube{Y}_\fr{P} = \left\{\left. x\in \fr{P}_K \right\vert v_{[x]}(f_i)<1 \;\forall i \right\}, \]
 and since rigid points $x$ satisfy $x=[x]$, we can conclude by applying \cite[Proposition 1.1.1]{Ber96b}. For the second claim, there are exact sequences
\[ 0\rightarrow \underline{\Gamma}^\dagger_D \rightarrow \mathrm{id} \rightarrow j_X^\dagger  \rightarrow 0 \]
 of functors on both $\mathbf{Sh}(\tube{Y}_\fr{P})$ and $\mathbf{Sh}(\tube{Y}_{\fr{P}0})$, it therefore suffices to consider $j_X^\dagger$. In this case, since the topos of an analytic variety is equivalent to that of the associated rigid space, the claim fact follows from the alternative definition of $j_X^\dagger$ given in \cite[Proposition 5.3]{LS07} (for the equivalence with Berthelot's definition, see \cite[Proposition 5.1.12]{LS07}).
\end{proof}

From this we may deduce:
\begin{enumerate}
\item there is a canonical isomorphism $\left( j_X^\dagger\mathcal{O}_{\tube{Y}_\fr{P}}\right)_0 \cong j_X^\dagger\mathcal{O}_{\tube{Y}_{\fr{P}0}}$;
\item the functor $\mathscr{F}\mapsto \mathscr{F}_0$ induces an equivalence of categories
\[ \mathbf{Coh}(j_X^\dagger\mathcal{O}_{\tube{Y}_\fr{P}}) \isomto  \mathbf{Coh}(j_X^\dagger\mathcal{O}_{\tube{Y}_{\fr{P}0}}).\] 
\end{enumerate}
We may use this to transport all results proved for overconvergent sheaves in the language of rigid analytic spaces into the adic context, for example the following.

\begin{proposition}[Proposition 2.1.10, \cite{Ber96b}] \label{prop: coherent sheaves colimit neighbourhoods} The inverse image functor induces an equivalence of categories
\[ \varinjlim_V \mathbf{Coh}(\mathcal{O}_V) \isomto \mathbf{Coh}(\mathcal{O}_{\tube{X}_\fr{P}})   \]
where $V$ ranges over all open neighbourhoods of $\tube{X}_\fr{P}$ in $\tube{Y}_\fr{P}$.  
\end{proposition}

\subsubsection{}

As well as the open tube $\tube{Y}_\fr{P}$ of a closed subscheme $Y\hookrightarrow P$ defined above, we will also need the variants $\left[Y\right]_{\mathfrak{P}\eta}$ and $\tube{Y}_{\mathfrak{P}\eta}$, which are defined for $\eta<1$ sufficiently close to $1$. When $\fr{P}$ is affine, and $f_1,\ldots,f_r\in \Gamma(\fr{P},\mathcal{O}_\fr{P})$ are such that $Y=V(\varpi,f_1,\ldots,f_r)$, define 
 \begin{align*} \left[Y\right]_{\mathfrak{P}\eta}&:= \left\{\left. x\in \fr{P}_K \right\vert v_{x}(f_i)\leq \eta \;\forall i \right\} \\  
 \tube{Y}_{\mathfrak{P}\eta}&:= \left\{\left. x\in \fr{P}_K \right\vert v_{[x]}(f_i)< \eta \;\forall i \right\}.
 \end{align*}
When $\norm{\varpi}<\eta<1$ these do not depend on the choice of the $f_i$, and hence glue together over an affine covering of $\fr{P}$, see \cite[\S1.1.8]{Ber96b}.

\subsection{Partitions and stratifications}

We use the notions of partition and stratification as defined in \cite[\href{https://stacks.math.columbia.edu/tag/09XY}{Tag 09XY}]{stacks}. 

\begin{definition} A partition of Noetherian topological space $X$ is a finite decomposition $X=\bigsqcup_{\alpha\in A }X_\alpha$ into locally closed subsets $X_\alpha$. It is called a stratification if $X_\alpha\cap \overline{X}_\beta \neq \emptyset \implies X_\alpha\subset \overline{X}_\beta$. 
\end{definition}

This is in fact called a (finite) good stratification in \cite[\href{https://stacks.math.columbia.edu/tag/09XY}{Tag 09XY}]{stacks}. There is an obvious notion of a refinement of a partition or stratification, and it is easy to see that every partition can be refined to a stratification. 

\subsection{Constructible isocrystals} \label{sec: constructible}

We recall the theory of constructible isocrystals as described in \cite{LS16}. Let $(X,Y,\mathfrak{P})$ be a frame over $\mathcal{V}$.

\begin{definition} An $\mathcal{O}_{\tube{X}_\fr{P}}$-module $\mathscr{F}$ is called \emph{constructible} if there exists a partition $\left\{ X_i\right\}_{i\in I}$ of $X$ such that each $\mathscr{F}|_{\tube{X_i}_{\fr{P}}}$ is a coherent $\mathcal{O}_{\tube{X_i}_{\fr{P}}}$-module.
\end{definition}

Now suppose that $\fr{P}$ is smooth over $\mathcal{V}$ in a neighbourhood of $X$, let $p_1,p_2\colon\tube{X}_{\fr{P}^2}\rightarrow \tube{X}_\fr{P}$ denote the two projection maps, and $\Delta\colon\tube{X}_\fr{P}\rightarrow \tube{X}_{\fr{P}^2}$ the diagonal. Let $\mathscr{F}$ be an $\mathcal{O}_{\tube{X}_\fr{P}}$-module.

\begin{definition} An overconvergent stratification on $\mathscr{F}$ is an isomorphism $\epsilon\colon  p_2^*\mathscr{F}\isomto p_1^*\mathscr{F}$ of $\mathcal{O}_{\tube{X}_{\fr{P}^2}}$-modules, restricting to the identity on $\Delta(\tube{X}_\fr{P})$, and satisfying the cocycle condition on $\tube{X}_{\fr{P}^3}$. 
\end{definition}

Note that `overconvergent' here might be more appropriately termed `partially overconvergent along $Y\setminus X$', since we are not assuming that $Y$ is proper over $k$. For brevity we will generally use the former term, unless there is scope for ambiguity. When $X=Y$, however, we will usually speak of convergent stratifications instead.

\begin{definition} A constructible isocrystal on the pair $(X,Y)$ is a constructible $\mathcal{O}_{\tube{X}_\fr{P}}$-module $\mathscr{F}$ together with an overconvergent stratification $\epsilon$.
\end{definition}

As the language suggests, the category of such objects is independent of the formal embedding $\fr{P}$ up to canonical equivalence (under the assumption that $\fr{P}$ is smooth around $X$), and is also functorial in the pair $(X,Y)$. To emphasise this fact, we will employ the following notation.

\begin{notation} If $(X,Y,\fr{P})$ is a frame, with $\fr{P}$ smooth around $X$, and $E$ is a constructible isocrystal on $(X,Y)$, we will write $E_\fr{P}$ for the realisation of $E$ on $\tube{X}_\fr{P}$.
\end{notation}

It follows from the results of \cite{LS16} that the category of constructible isocrystals on $(X,Y)$ embeds fully faithfully in the category of constructible $\mathcal{O}_{\tube{X}_\fr{P}}$-modules with integrable connection. We will therefore speak of an (integrable) connection on a constructible $\cO_{\tube{X}_\fr{P}}$-module as being (over)convergent, that is, in the essential image of this functor.

\begin{lemma} \label{lemma: flat} Let $\mathscr{F}$ be a constructible isocrystal on $(X,Y)$. Then $\mathscr{F}$ is a flat $\mathcal{O}_{\tube{X}_\fr{P}}$-module. 
\end{lemma} 

\begin{proof}
Since flatness can be checked on stalks, we can reduce to the case when $\mathscr{F}$ is a coherent $\mathcal{O}_{\tube{X}_\fr{P}}$-module. Since it admits an integrable connection, it is locally free, and hence flat.
\end{proof}

\begin{definition} \label{defn: support cons}We say that a constructible isocrystal $E$ is supported on some locally closed subscheme $Z\hookrightarrow X$ if $E_{\mathfrak{P}}|_{\tube{X\setminus Z}_\fr{P}}=0$. Again, this only depends on $(X,Y)$ and not on $\fr{P}$.
\end{definition}

Since most of our main results will ultimately depend on \cite{CT12}, we will need to consider objects with Frobenius structures.

\begin{definition} \label{defn: F type} A Frobenius structure on a constructible isocrystal $E$ is an isomorphism $F^{n*}E\cong E$ for some $n\geq 1$. A constructible isocrystal is said to be \emph{of Frobenius type} if it is an iterated extension of objects admitting Frobenius structures. 
\end{definition}

\begin{remark} Even if $\fr{P}$ does not admit a lift of Frobenius, the fact that constructible isocrystals are functorial in $(X,Y)$ means that there is always a canonically defined Frobenius pullback functor. 
\end{remark}

We will denote by $\mathrm{Isoc}_\mathrm{cons}(X,Y)$ the category of constructible isocrystals \emph{of Frobenius type} on the pair $(X,Y)$. 

\begin{warning} The reader should be warned that this notation is rather non-standard. Although the results of \S\ref{sec: calc Rsp} and \S\ref{sec: pre Dd} don't require Frobenius structures, the key results of \S\ref{sec: finite} and \S\ref{sec: comp} do, so we shall impose the condition of being `of Frobenius type' throughout. For the sake of brevity, we shall, however, suppress this condition from the notation. The reader should also be warned that our condition of being `of Frobenius type' is much stronger than that of the same name occuring in \cite{LS14}.
\end{warning}

When $Y$ is proper over $k$, the category of constructible isocrystals on $(X,Y)$ only depends on $X$ up to canonical equivalence; we will call such objects constructible isocrystals on $X$, and will denote by $\mathrm{Isoc}^\dagger_\mathrm{cons}(X)$ the category of such objects which are of Frobenius type. When $X=Y=P$, we will often write $\mathrm{Isoc}_{\mathrm{cons}}(\mathfrak{P})$, and talk about constructible isocrystals on $\fr{P}$, rather than on $(P,P)$.

The full subcategory of $\mathrm{Isoc}_\mathrm{cons}(X,Y)$ consisting of objects whose underlying $\mathcal{O}_{\tube{X}_\fr{P}}$-module is coherent is precisely the category of (partially) overconvergent isocrystals on $(X,Y)$ of Frobenius type (in the senes of Definition \ref{defn: F type}), and will be denoted by $\mathrm{Isoc}(X,Y)$. When $Y$ is proper, this induced the inclusion 
\[ \mathrm{Isoc}^\dagger(X)\subset \mathrm{Isoc}_\mathrm{cons}^\dagger(X) \]
from the category of overconvergent isocrystals on $X$ of Frobenius type, to the category of constructible isocrystals on $X$ of Frobenius type.

If $(f,g):(X,Y)\rightarrow (X',Y')$ is a morphism of pairs, there is an exact pullback functor
\[ f^*:\mathrm{Isoc}_\mathrm{cons}(X',Y')\rightarrow\mathrm{Isoc}_\mathrm{cons}(X,Y),\]
which only depends on the morphism $X\rightarrow X'$ when $Y$ and $Y'$ are proper. If moreover $Y\hookrightarrow Y'$ is a closed immersion (thus $X\hookrightarrow X'$ is a locally closed immersion), there is an exact functor
\[ f_!: \mathrm{Isoc}_\mathrm{cons}(X,Y)\rightarrow\mathrm{Isoc}_\mathrm{cons}(X',Y')\]
which on realisations is defined to be the extension by zero along the map of tubes $f:\tube{X}_\fr{P}\rightarrow \tube{X'}_\fr{P}$. This induces an equivalence of categories between $\mathrm{Isoc}_\mathrm{cons}(X,Y)$ and the full subcategory of $\mathrm{Isoc}_\mathrm{cons}(X',Y')$ consisting of objects supported on $X$. If $f:X\hookrightarrow X'$ is a closed immersion, then $f_!$ is a \emph{left} adjoint to $f^*$, and if $f:X\hookrightarrow X'$ is an open immersion, then it is a \emph{right} adjoint to $f^*$.

In particular, if $a\colon Z\rightarrow X$ is a locally closed immersion, and $\overline{Z}$ denotes the closure of $Z$ in $Y$, then $a_!$ induces a fully faithful functor
\[ a_!\colon \mathrm{Isoc}(Z,\overline{Z})\rightarrow \mathrm{Isoc}_\mathrm{cons}(X,Y) \]
from overconvergent isocrystals on $(Z,\overline{Z})$ to constructible isocrystals on $(X,Y)$. Moreover, it follows from \cite[Proposition 3.6]{LS16} that every object $E\in \mathrm{Isoc}(X,Y)$ is an iterated extension of objects in the essential image of such functors (for varying $Z$).

\subsection{Arithmetic \texorpdfstring{$\pmb{\mathscr{D}}$}{D}-modules on pairs}

We recall the definition of overholonomic complexes on pairs from \cite{AC18a}. Given a smooth formal scheme $\fr{P}$, we let ${\bf D}^b_{\mathrm{hol}}(\mathscr{D}^\dagger_{\fr{P}\Q})$ denote the category of overholonomic complexes of $\mathscr{D}^\dagger_{\fr{P}\Q}$-modules as defined in \cite{Car09b}, and 
\[{\bf D}^b_{\mathrm{hol}}(\fr{P})\subset {\bf D}^b_{\mathrm{hol}}(\mathscr{D}^\dagger_{\fr{P}\Q}) \]
its full subcategory of objects of Frobenius type, that is, objects whose cohomology sheaves are iterated extensions of objects admitting Frobenius structures.

\begin{definition} \label{defn: lp frame}
\begin{enumerate} 
\item  An l.p.\ frame is a frame $(X,Y,\fr{P})$ over $\mathcal{V}$ such that $\fr{P}$ is smooth and admits a locally closed immersion $\fr{P}\hookrightarrow \fr{Q}$ into a smooth and proper formal scheme over $\mathcal{V}$.
\item  A pair $(X,Y)$ over $k$ is called realisable if there exists an l.p.\ frame $(X,Y,\fr{P})$ over $\mathcal{V}$.
\item A variety $X$ over $k$ is called realisable if the pair $(X,X)$ is realisable. This just says that there exists an immersion $X\hookrightarrow \fr{P}$ with $\fr{P}$ smooth and proper over $\mathcal{V}$. More generally, we say that a formal $\cV$-scheme is realisable if it admits an immersion into a smooth and proper formal $\cV$-scheme.
\item We will say that 
\end{enumerate} 
\end{definition}

Note that while our definition of an l.p. frame is slightly more general than that used in \cite{AC18a}, where the map $\fr{P}\rightarrow \fr{Q}$ is required to be an open immersion, the resulting notion of a pair being realisable is the same. Also note that the above notation of realisability for varieties is stronger than one often appearing in the literature, which requires the existence of a frame $(X,Y,\fr{P})$ such that $Y$ is proper and $\fr{P}$ is smooth around $X$. 

If $(X,Y)$ is a realisable pair, and $(X,Y,\fr{P})$ is an l.p.\ frame, we denote by 
\[ \mathbf{R}\underline{\Gamma}_X^\dagger:= \mathbf{R}\underline{\Gamma}_Y^\dagger \circ (^\dagger Y\setminus X) \colon {\bf D}^b_{\mathrm{hol}}(\mathscr{D}^\dagger_{\fr{P}\Q}) \rightarrow {\bf D}^b_{\mathrm{hol}}(\mathscr{D}^\dagger_{\fr{P}\Q}) \]
the functor of sections with support on $X$. 

\begin{definition} \label{defn: support Dmod} We say that $\mathcal{M}\in {\bf D}^b_{\mathrm{hol}}(\mathscr{D}^\dagger_{\fr{P}\Q})$ is supported on $X$ if there exists an isomorphism $\mathcal{M} \isomto \mathbf{R}\underline{\Gamma}_X^\dagger\mathcal{M}$ in ${\bf D}^b_{\mathrm{hol}}(\mathscr{D}^\dagger_{\fr{P}\Q})$.
\end{definition}

The category ${\bf D}^b_{\mathrm{hol}}(X,Y)$ of overholonomic complexes on $(X,Y)$ of Frobenius type is defined to be the full-subcategory of ${\bf D}^b_{\mathrm{hol}}(\fr{P})$ consisting of objects which are supported on $X$. That this does not depend on the choice of l.p.\ frame, up to canonical equivalence, is shown for the more restricted version of an l.p.\ frame in \cite{AC18a}. This invariance is easily seen to extend to the marginally more general version we use here. If $Y$ is proper over $k$, ${\bf D}^b_{\mathrm{hol}}(X,Y)$ only depends on $X$ up to canonical equivalence, in which case we will denote it by ${\bf D}^{b}_{\mathrm{hol}}(X)$.

The categories ${\bf D}^b_{\mathrm{hol}}(-)$ support a suitable formalism of Grothendieck's six operations. For any morphism of couples $(f,g):(X',Y')\rightarrow (X,Y)$ there are functors
\[ f^!,f^+\colon {\bf D}^b_{\mathrm{hol},F}(X,Y) \rightarrow {\bf D}^b_{\mathrm{hol}}(X',Y'), \]
as well as functors
\[ f_+,f_!\colon {\bf D}^b_{\mathrm{hol}}(X',Y') \rightarrow {\bf D}^b_{\mathrm{hol}}(X,Y) \]
whenever $g$ is proper (note the slight abuse of notation in the pushforward and pullback functors associated to a morphism of couples). We also have the duality functor, written (again, slightly abusively) as
\[ \mathbf{D}_{X} \colon {\bf D}^b_{\mathrm{hol}}(X,Y)^\mathrm{op} \rightarrow {\bf D}^b_{\mathrm{hol}}(X,Y). \]
We refer to \cite[\S1]{AC18a} and the references therein for more details.

\begin{remark} The fact that $\bD^b_{\hol}(X,Y)$ is defined using an l.p. frame is the source of the `realisability' hypotheses in Theorems \ref{theo: DCon} and \ref{theo: sp_+ comp}. We are relatively confident that this hypothesis could be removed from Theorem \ref{theo: DCon} with a little extra effort, however, this would have no effect on Theorem \ref{theo: sp_+ comp}.

A natural question to ask in this direction would be whether it is possible to define $\bD^b_\hol(X,Y)$ for pairs $(X,Y)$ with $Y$ admitting a closed immersion into a smooth formal $\cV$-scheme $\fP$. That us, we do not assume that $\fP$ admits a locally closed immersion into a smooth and \emph{proper} formal scheme over $\cV$. If this could be done, it would likely enable us to slightly weaken the `realisability' assumption in Theorem \ref{theo: sp_+ comp}.

The key step that is missing in order to achieve this is a more general version of \cite[Th\'eor\`eme 3.9]{Car09b}. What is needed is to prove the following: if $u\from \fP\to \fQ$ is a morphism of smooth formal schemes, and is $\cM$ an overholonomic $\sD^\dagger_{\fP\Q}$-module whose support is proper over $\fQ$, then $u_+\cM$ is overholonomic on $\fQ$. Proving this would, in turn, require showing the following: if $\cM$ is a coherent $\sD^\dagger_{\fP\Q}$-module whose support is proper over $\fQ$, then:
\begin{enumerate}
\item $u_+\cM\in \bD^b_\coh(\sD^\dagger_{\fQ\Q})$,
\item there is an isomorphism $u_+(\bD_{\fP}\cM)  \isomto \bD_{\fQ}(u_+\cM)$ in $\bD^b_\coh(\sD^\dagger_{\fQ\Q})$.
\end{enumerate}
\end{remark}

\section{Calculating \texorpdfstring{$\mathbf{R}\mathrm{sp}_*$}{Rsp*}} \label{sec: calc Rsp}

Let $\mathfrak{P}$ be a formal scheme, flat over $\mathcal{V}$, and $\mathscr{F}$ a constructible $\mathcal{O}_{\mathfrak{P}_K}$-module. Our goal in this section is to calculate $\mathbf{R}\mathrm{sp}_*\mathscr{F}$ explicitly. There are two key ingredients in this calculation.

The first is a procedure for `localising' cohomology on $\fr{P}_K$ over a stratification of $P$. This bears certain similarities with the exit-path construction from \cite{Tre09}, used to describe constructible sheaves as representations in a similar fashion to the well-known description for locally constant sheaves.

The second is a method for computing pushforwards along maps of tubes, based upon the Roos complexes used to compute derived inverse limits.

\subsection{Localising cohomology over a stratification} \label{sec: loc strat}

Since constructible modules on $\mathfrak{P}_K$ are controlled over partitions of $P$, we will need to be able to calculate cohomology locally with respect to such a partition. We explain in this section how to do so, at least when our partition is actually a stratification.

Let $i\colon Y\hookrightarrow P$ be a closed subscheme, $\mathscr{F}$ a sheaf on $\tube{Y}_\mathfrak{P}$, and $\left\{Y_\alpha \right\}_{\alpha\in A}$ a finite stratification of $Y$. Even though we will eventually be interested in the case $Y=P$, it will be helpful to allow more general $Y$ in order to make inductive proofs work more easily. For any $\alpha\in A$, let $i_{\alpha}\colon Y_\alpha\rightarrow Y$ denote the given locally closed immersion, as well as the induced map $i_{\alpha}\colon ]Y_\alpha[_\mathfrak{P}\rightarrow ]Y[_\mathfrak{P}$ on tubes.  

The index set $A$ is endowed with a natural partial order such that $\alpha \leq \beta \iff Y_\beta \subset \overline{Y}_\alpha$ (here $\overline{Y}_\alpha$ is the closure of $Y_\alpha$ in $Y$). We may therefore define a new poset $\mathrm{sd}(A)$ consisting of \emph{chains} in $A$ (i.e. totally ordered subsets of $A$), where the order on $\mathrm{sd}(A)$ is simply given by inclusion of chains. Note that chains do not need to be `complete', for example, we can have the chain $\{\alpha<\gamma\}$ contained in the chain $\{\alpha<\beta<\gamma\}$.

We thus obtain a functor
\[ \mathbf{R}i_{Y_\bullet*}i_{Y_\bullet}^{-1}\mathscr{F} : \mathrm{sd}(A) \rightarrow {\bf D}(\tube{Y}_\mathfrak{P}) \]
defined by setting
\[ \left(\mathbf{R}i_{Y_\bullet*}i_{Y_\bullet}^{-1}\mathscr{F} \right)_{\{\alpha_0 < \ldots < \alpha_r \}}= \mathbf{R}i_{\alpha_0*}i_{\alpha_0}^{-1}\ldots \mathbf{R}i_{\alpha_r*}i_{\alpha_r}^{-1}\mathscr{F},  \]
with the map associated to a given inclusion of chains coming from the adjunction between pushforward and pullback. Although $\mathbf{R}i_{Y_\bullet*}i_{Y_\bullet}^{-1}\mathscr{F}$ is defined as a diagram in ${\bf D}(\tube{Y}_\mathfrak{P})$, it is easy to lift it to a diagram in $\mathbf{Ch}^+(\tube{Y}_\mathfrak{P})$ (for example, by using Godement resolutions) and hence consider it canonically as an object of the derived category
\[ {\bf D}\left(\mathbf{Sh}(\tube{Y}_\mathfrak{P})^{\mathrm{sd}(A)}\right)\]
of $\mathrm{sd}(A)$-shaped diagrams in $\mathbf{Sh}(\tube{Y}_\mathfrak{P})$.

Let $\mathrm{sd}(A)^{\vartriangleleft}$ denote the cone category, that is, the category obtained from $\mathrm{sd}(A)$ by freely adjoining an initial object. Then the natural adjunction maps 
\[ \mathscr{F} \rightarrow \mathbf{R}i_{\alpha_0*}i_{\alpha_0}^{-1}\ldots \mathbf{R}i_{\alpha_r*}i_{\alpha_r}^{-1}\mathscr{F} \]
given an extension of $\mathbf{R}i_{Y_\bullet*}i_{Y\bullet}^{-1}\mathscr{F}$ to a functor
\[ 	\mathrm{sd}(A)^{\vartriangleleft} \rightarrow {\bf D}(\tube{Y}_\fr{P}), \]
which, as above, can be canonically lifted to give an object of the derived category
\[ {\bf D}\left(\mathbf{Sh}(\tube{Y}_\mathfrak{P})^{\mathrm{sd}(A)^{\vartriangleleft}}\right)\]
of $\mathrm{sd}(A)^{\vartriangleleft}$-shaped diagrams in $\mathbf{Sh}(\tube{Y}_\mathfrak{P})$. The following proposition then tells us how to localise the cohomology of $\mathscr{F}$ along the given stratification of $Y$.

\begin{proposition} \label{prop: coh loc strat} The natural map
\[\mathscr{F}\rightarrow \mathrm{holim}_{\mathrm{sd}(A)} \mathbf{R}i_{Y_\bullet*}i_{Y_\bullet}^{-1}\mathscr{F} \]
is an isomorphism in ${\bf D}(\tube{Y}_\mathfrak{P})$.
\end{proposition}

\begin{remark} For the general theory of homotopy limits and colimits, see either \cite[Appendix A.2]{Lur09} or \cite[Part I]{Rie14}. Here we are endowing the category $\mathbf{Ch}(\tube{Y}_\fr{P})$ of unbounded chain complexes with the injective model structure constructed in \cite{Hov01}.
\end{remark}

\begin{proof}
We will induct on the number of strata $\norm{A}$, and to do so rigorously we will need to make use of the language of $\infty$-categories. Let $\EU{D}(\tube{Y}_\fr{P})$ denote the bounded below derived $\infty$-category of sheaves on $\fr{P}_K$, that is, the $\infty$-category presented by the model category $\mathbf{Ch}(\fr{P}_K)$. Then $\EU{D}(\tube{Y}_\fr{P})$ is a presentable $\infty$-category, and our diagram
\[ \mathbf{R}i_{Y_\bullet*}i_{Y_\bullet}^{-1}\mathscr{F} : \mathrm{sd}(A)\rightarrow \mathbf{Ch}(\tube{Y}_\mathfrak{P}) \]
constructed using (for example) Godement resolutions gives rise to a diagram
\[ \mathbf{R}i_{Y_\bullet*}i_{Y_\bullet}^{-1}\mathscr{F}: \mathrm{N}(\mathrm{sd}(A))\rightarrow \EU{D}(\tube{Y}_\fr{P}). \]
By \cite[Theorem 4.2.4.1]{Lur09} the homotopy limit of our original diagram $\mathbf{R}i_{Y_\bullet*}i_{Y_\bullet}^{-1}\mathscr{F}$ is equivalent to the $\infty$-categorical limit of $\mathbf{R}i_{Y_\bullet*}i_{Y_\bullet}^{-1}\mathscr{F}$, which in turn is equivalent, essentially by definition, to the right Kan extension of $\mathbf{R}i_{Y_\bullet*}i_{Y_\bullet}^{-1}\mathscr{F}$ along the unique map
\[  \mathrm{N}(\mathrm{sd}(A))\rightarrow \Delta^0.  \]
If $\norm{A}=1$ then we have the trivial stratification $\left\{ Y \right\}$ and there is nothing to prove. If $\norm{A}>1$, then we choose an open stratum $i_{\alpha}:Y_{\alpha} \hookrightarrow Y$, and let $i_{Z} : Z\hookrightarrow Y$ be its closed complement. We consider the map
\[ F:\mathrm{N}(\mathrm{sd}(A)) \rightarrow \Lambda^2_2\] defined by
\[ \begin{cases} \beta_0<\ldots <\beta_r &\mapsto 0 \;\;\;\; (\beta_0\neq \alpha) \\  \alpha &\mapsto 1  \\ \alpha<\beta_1<\ldots <\beta_r &\mapsto 2 \;\;\;\;(r\geq 1) \end{cases}\]
and the factorisation
\[ \mathrm{N}(\mathrm{sd}(A)) \overset{F}{\rightarrow} \Lambda^2_2 \overset{G}{\rightarrow} \Delta^0. \]
Note that $F$ is easily checked to be a Cartesian fibration. If we denote right Kan extension by $(-)_*$, it follows from \cite[Proposition 4.3.2.15, Proposition 4.3.3.7]{Lur09} that
\[ \mathrm{lim}_{\mathrm{sd}(A)} \mathbf{R}i_{Y_\bullet*}i_{Y_\bullet}^{-1}\mathscr{F} \simeq  (G\circ F)_*\mathbf{R}i_{Y_\bullet*}i_{Y_\bullet}^{-1}\mathscr{F} \simeq  G_*F_*\mathbf{R}i_{Y_\bullet*}i_{Y_\bullet}^{-1}\mathscr{F}  \simeq  \mathrm{lim}_{\Lambda^2_2}F_*\mathbf{R}i_{Y_\bullet*}i_{Y_\bullet}^{-1}\mathscr{F}. \]
Now applying \cite[Proposition 4.3.3.10]{Lur09} we see that $F_*\mathbf{R}i_{Y_\bullet*}i_{Y_\bullet}^{-1}\mathscr{F} $ can be calculated pointwise. In other words, for a vertex $i\in \Lambda^2_2$ we have
\[ F_*\mathbf{R}i_{Y_\bullet*}i_{Y_\bullet}^{-1}\mathscr{F}(i) \simeq \mathrm{lim}_{F^{-1}(i)}\mathbf{R}i_{Y_\bullet*}i_{Y_\bullet}^{-1}\mathscr{F}. \]
%
To finish the proof, we know by induction on $\norm{A}$ that
\begin{align*} \mathrm{lim}_{F^{-1}(0)} \mathbf{R}i_{Y_\bullet*}i_{Y_\bullet}^{-1}\mathscr{F} &\simeq  \mathbf{R}i_{Z*}i_Z^{-1}\mathscr{F} \\
\mathrm{lim}_{F^{-1}(2)} \mathbf{R}i_{Y_\bullet*}i_{Y_\bullet}^{-1}\mathscr{F} &\simeq i_{\alpha*}i_{\alpha}^{-1} \mathbf{R}i_{Z*}i_Z^{-1}\mathscr{F}.
 \end{align*}
We therefore deduce that $\mathrm{lim}_{\mathrm{sd}(A)} \mathbf{R}i_{Y_\bullet*}i_{Y_\bullet}^{-1}\mathscr{F}$ is equivalent to the pullback of the diagram
\[ \xymatrix{  &   i_{\alpha*}i_{\alpha}^{-1}\mathscr{F} \ar[d] \\  \mathbf{R}i_{Z*}i_{Z}^{-1}\mathscr{F} \ar[r]  & i_{\alpha*}i_{\alpha}^{-1}\mathbf{R}i_{Z*}i_{Z}^{-1}\mathscr{F}.  }  \]
Thus we need to show that 
 \[ \xymatrix{ \mathscr{F} \ar[r]\ar[d]&   i_{\alpha*}i_{\alpha}^{-1}\mathscr{F} \ar[d] \\  \mathbf{R}i_{Z*}i_{Z}^{-1}\mathscr{F} \ar[r]  & i_{\alpha*}i_{\alpha}^{-1}\mathbf{R}i_{Z*}i_{Z}^{-1}\mathscr{F},  }  \]
is a limit diagram. But this just follows from the fact that both
\[ i_{Z!}i_{Z}^{-1}\mathscr{F} \rightarrow \mathscr{F} \rightarrow i_{\alpha*}i_{\alpha}^{-1}\mathscr{F}  \overset{+1}{\rightarrow}\]
and 
\[ i_{Z!}i_{Z}^{-1}\mathscr{F} \rightarrow \mathbf{R}i_{Z*}i_{Z}^{-1}\mathscr{F} \rightarrow i_{\alpha*}i_{\alpha}^{-1}\mathbf{R}i_{Z*}i_{Z}^{-1}\mathscr{F}  \overset{+1}{\rightarrow}\]
are exact triangles. 
\end{proof}

There is a slightly different way of expressing this construction that will end up being more useful for us. For each $\alpha\in A$, let $\overline{Y}_\alpha$ denote the closure of $Y_\alpha$ in $Y$, with closed immersion $\bar{i}_{\alpha}:\overline{Y}_\alpha\hookrightarrow Y$. If $\alpha\leq \beta$ are comparable elements of $A$, let $\overline{i}_{\beta\alpha}:\overline{Y}_\beta\hookrightarrow \overline{Y}_\alpha$ denote the given closed immersion.

\begin{lemma} For any chain $\{\alpha_0 < \ldots < \alpha_r \}\subset A$, and any sheaf $\mathscr{F}$ on $\tube{Y}_\mathfrak{P}$ the `base change' map
\[ \mathbf{R}i_{\alpha_0*}i_{\alpha_0}^{-1}\ldots \mathbf{R}i_{\alpha_r*}i_{\alpha_r}^{-1}\mathscr{F} \rightarrow \mathbf{R}\bar{i}_{\alpha_0*} j_{Y_{\alpha_0}}^\dagger \mathbf{R}\overline{i}_{\alpha_1\alpha_0*}j_{Y_{\alpha_1}}^\dagger \ldots\mathbf{R}\overline{i}_{\alpha_{r}\alpha_{r-1}*}j_{Y_{\alpha_r}}^\dagger \overline{i}^{-1}_{\alpha_r}\mathscr{F}  \]
is an isomorphism in ${\bf D}^+(\tube{Y}_\mathfrak{P})$.
\end{lemma}

\begin{proof}
This is a straightforward calculation, which ultimately boils down to transitivity of pushforwards and the fact that $\overline{i}_\alpha^{-1}\mathbf{R}\overline{i}_{\alpha*}\cong \mathrm{id}$. 
\end{proof}

We thus define a functor
\[ \mathbf{R}i_{\overline{Y}_\bullet}j_{Y_\bullet}^\dagger\mathscr{F} : \mathrm{sd}(A)\rightarrow {\bf D}(\tube{Y}_\mathfrak{P}) \]
by setting
\[ \left( \mathbf{R}i_{\overline{Y}_\bullet}j_{Y_\bullet}^\dagger\mathscr{F}\right)_{\{\alpha_0 < \ldots < \alpha_r \}} :=\mathbf{R}\bar{i}_{\alpha_0*} j_{Y_{\alpha_0}}^\dagger \mathbf{R}\overline{i}_{\alpha_1\alpha_0*}j_{Y_{\alpha_1}}^\dagger \ldots\mathbf{R}\overline{i}_{\alpha_{r}\alpha_{r-1}*}j_{Y_{\alpha_r}}^\dagger \overline{i}^{-1}_{\alpha_r}\mathscr{F},   \]
again with the natural adjunction maps as the chain $\{\alpha_1 < \ldots < \alpha_r \}$ varies. As before, we can use Godement resolutions to lift this functor, together with the adjunction map from $\mathscr{F}$, canonically to an object of ${\bf D}\left(\mathbf{Sh}(\tube{Y}_\mathfrak{P})^{\mathrm{sd}(A)^{\vartriangleleft}}\right)$.

\begin{corollary}  \label{cor: coh loc strat}The map
\[ \mathscr{F} \rightarrow \mathrm{holim}_{\mathrm{sd}(A)}\mathbf{R}i_{\overline{Y}_\bullet}j_{Y_\bullet}^\dagger\mathscr{F} \]
is an isomorphism in ${\bf D}(\tube{Y}_\mathfrak{P})$.
\end{corollary}

\subsection{Good stratifications and Roos complexes} \label{sec: Roos new}

To be able to make use of Corollary \ref{cor: coh loc strat} to calculate $\mathbf{R}\mathrm{sp}_*$ of a constructible $\mathcal{O}_{\mathfrak{P}_K}$-module, we will need to explicitly calculate $\mathbf{R}i_{\overline{Y}_\bullet}j_{Y_\bullet}^\dagger\mathscr{F}$, at least for suitable stratifications. To do so, we shall need to calculate $\mathbf{R}i_{Z*}$ when $i_Z:Z \hookrightarrow Y$ is an inclusion of closed subschemes of $P$. This can be achieved by a version of the usual Roos complexes computing countable inverse limits. 

So suppose that we have such a closed immersion $i_Z:Z\hookrightarrow Y$. For every $n\geq 1$, we set $\eta_n:=\norm{\varpi}^{\frac{1}{n+1}}$. Thus $\norm{\varpi}<\eta_n<1$, and the closed tubes $\left[Z\right]_n:=\left[ Z \right]_{\mathfrak{P}\eta_n}$ are well-defined. We clearly have $\tube{Z}_\mathfrak{P}=\bigcup_{n\geq 1} \left[Z \right]_n$. We denote by
\[ i_{Z,n}\colon\left[Z\right]_n \hookrightarrow \tube{Y}_\mathfrak{P} \]
the inclusion. For any sheaf $\mathscr{F}$ on $\tube{Z}_\mathfrak{P}$, and any $n\geq 1$, we therefore obtain natural restriction maps
\[ \pi_{n+1,n} \colon i_{Z,n+1*}\left(\mathscr{F}|_{\left[ Z \right]_{n+1}}\right)\rightarrow i_{Z,n*}\left(\mathscr{F}|_{\left[ Z \right]_n}\right) \]
of sheaves on $\tube{Y}_\mathfrak{P}$.

\begin{definition} We define the Roos complex 
\[ \mathcal{R}_{i_Z}(\mathscr{F}):=\left[ \prod_{n\geq 1} i_{Z,n*}\left(\mathscr{F}|_{\left[ Z \right]_n}\right) \overset{\pi-\mathrm{id}}{\longrightarrow}  \prod_{n\geq 1} i_{Z,n*}\left(\mathscr{F}|_{\left[ Z \right]_n}\right) \right] \]
in the usual way, where $\pi:=\prod_{n\geq 2}\pi_{n+1,n}$ is the product of the restriction maps. If $n_0\geq 1$ then we can define the analogous Roos complex
\[ \mathcal{R}_{i_Z,n_0}(\mathscr{F}):= \left[ \prod_{n\geq n_0} i_{Z,n*}\left(\mathscr{F}|_{\left[ Z \right]_n}\right) \overset{\pi-\mathrm{id}}{\longrightarrow}  \prod_{n\geq n_0} i_{Z,n*}\left(\mathscr{F}|_{\left[ Z \right]_n}\right)  \right] \]
starting at $n_0$. By passing to the colimit in $n_0$, we define the infinite level Roos complex
\[ \mathcal{R}_{i_Z,\infty}(\mathscr{F}):= \mathrm{colim}_{n_0} \mathcal{R}_{i_Z,n_0}(\mathscr{F}). \]
\end{definition} 

The reason for considering the latter two variants is to be able to endow these complexes with the structure of pre-$\mathscr{D}^\dagger$-modules in \S\ref{sec: pre Dd} below, since (in the cases of interest to us) no individual $\mathcal{R}_{i_Z,n_0}(\mathscr{F})$ will admit such a structure. In order to describe $\mathbf{R}i_{\overline{Y}_\bullet}j_{Y_\bullet}^\dagger\mathscr{F}$ in terms of these Roos complexes, it will be helpful to introduce the following terminology.

\begin{definition} A sheaf $\mathscr{F}$ on an open subspace $U\subset \fr{P}_K$ is called \emph{affinoid-acyclic} if ${\rm H}^q(W,\mathscr{F})=0$ for every open affinoid $W\subset U$, and every $q>0$.
\end{definition}

The fundamental example of such a sheaf is the following.

\begin{example} \label{exa: aff-acy} Suppose that $Y$ is a closed subscheme of $P$, and $X\hookrightarrow Y$ is a open immersion, locally (on $Y$) complementary to a hypersurface $V(f)\subset Y$. Then any coherent $j_X^\dagger\mathcal{O}_{\tube{Y}_\mathfrak{P}}$-module $\mathscr{F}$ is affinoid-acyclic on $\tube{Y}_\mathfrak{P}$. Indeed, since $Y\setminus X$ is locally a hypersurface, it follows that, for any open affinoid $W\subset \tube{Y}_\mathfrak{P}$, there exists a cofinal system of open neighbourhoods $\{W_\lambda\}$ of $\tube{X}_\mathfrak{P}\cap W$ in $W$, such that each $W_\lambda$ is also affinoid. Moreover, we can assume that $\mathscr{F}$ extends to a compatible family $\{\mathscr{F}_\lambda\}$ of coherent $\mathcal{O}_{W_\lambda}$-modules. If we let $j_{\lambda}:W_\lambda\rightarrow W$ denote the inclusion, then
\[ \mathscr{F} \cong \mathrm{colim}_\lambda j_{\lambda^*}\mathscr{F}_\lambda \cong \mathrm{colim}_\lambda \mathbf{R}j_{\lambda*}\mathscr{F}_\lambda  \]
since $j_\lambda$ is an open immersion of affinoids and $\mathscr{F}_\lambda$ is coherent. Thus
\[ {\rm H}^q(W,\mathscr{F}) \cong {\rm H}^q(W,\mathrm{colim}_\lambda \mathbf{R}j_{\lambda*}\mathscr{F}_\lambda) \cong \mathrm{colim}_\lambda {\rm H}^q(W_\lambda,\mathscr{F}_\lambda) = 0 \]
if $q>0$.
\end{example}

Open immersions  $X\hookrightarrow Y$ of the above form will play an important role in this article, and the following terminology was suggested to us by B. Le Stum. 

\begin{definition} An open immersion $j:X\hookrightarrow Y$ of schemes is called strongly affine if, locally on $Y$, there exists $f\in \mathcal{O}_Y$ such that $X=D(f)$. 
\end{definition}

Note that if $Y$ is regular, then every affine open immersion is strongly affine. We then have the following stabilities of affinoid-acyclicity.

\begin{lemma} \label{lemma: aff stab}\begin{enumerate} 
\item Let $V\subset U \subset \fr{P}_K$ be open subspaces, and $\mathscr{F}$ an affinoid-acyclic sheaf on $U$. Then $\mathscr{F}|_{V}$ is affinoid-acyclic.
\item Let $\left\{\mathscr{F}_i \right\}_{i\in I}$ a family of affinoid-acyclic sheaves on $U\subset \fr{P}_K$. Then $\left\{\mathscr{F}_i \right\}_{i\in I}$ is $\prod_{i\in I}$-acyclic, and moreover $\prod_{i\in I}\mathscr{F}_i$ is also affinoid-acyclic.
\item Let $Y$ be a closed subscheme of $P$, and $X\hookrightarrow Y$ a strongly affine open immersion. If $\mathscr{F}$ is an affinoid-acyclic sheaf on $\tube{Y}_\mathfrak{P}$, then so is $j_X^\dagger\mathscr{F}$. 
\item If $\mathscr{F}$ is an affinoid-acyclic sheaf on $\mathfrak{P}_K$, then $\mathscr{F}$ is $\mathrm{sp}_*$-acyclic in the usual sense.
\end{enumerate}
\end{lemma}

\begin{proof}
\begin{enumerate}
\item Trivial.
\item Suppose that $0\rightarrow \mathscr{F}_i \rightarrow \mathscr{G}_i \rightarrow \mathscr{H}_i\rightarrow 0$ is an arbitrary $I$-indexed family of exact sequences of sheaves on $U$, and $W\subset U$ is an open affinoid. Then the fact that $\mathscr{F}_i$ is affinoid-acyclic implies that the sequence
\[0\rightarrow  \Gamma(W,\mathscr{F}_i) \rightarrow \Gamma(W,\mathscr{G}_i) \rightarrow \Gamma(W,\mathscr{H}_i)\rightarrow 0 \]
of sections is exact, hence so is the sequence
\[ 0\rightarrow  \prod_{i\in I}\Gamma(W,\mathscr{F}_i) \rightarrow \prod_{i\in I}\Gamma(W,\mathscr{G}_n) \rightarrow \prod_{i\in I}\Gamma(W,\mathscr{H}_i)\rightarrow 0 .\]
Thus the sequence 
\[ 0\rightarrow  \prod_{i\in I}\mathscr{F}_i \rightarrow \prod_{i\in I} \mathscr{G}_i \rightarrow \prod_{i\in I}\mathscr{H}_i\rightarrow 0 \]
of sheaves is exact, and since $\left\{ \mathscr{G}_i\right\}$ and $\left\{ \mathscr{H}_i\right\}$ were arbitrary, we deduce that $\mathbf{R}^q\prod_{i\in I}\mathscr{F}_i=0$ for $q>0$. Then, since cohomology commutes with derived products, we see that
\[ {\rm H}^q(W,\prod_{i\in I}\mathscr{F}_i)={\rm H}^q(W,\mathbf{R}\prod_{i\in I}\mathscr{F}_i)=\prod_{i\in I}{\rm H}^q(W,\mathscr{F}_i)=0\]
if $q>0$ and $W\subset \mathscr{X}$ is open affinoid. 
\item Entirely similar to Example \ref{exa: aff-acy} above.
\item This is again trivial, since $\mathfrak{P}$ has a basis of opens $\mathfrak{U}$ such that $\mathrm{sp}^{-1}(\mathfrak{U})$ is affinoid.
\end{enumerate}\end{proof}

\begin{proposition} \label{prop: roos aff acy} Let $i_Z:Z\hookrightarrow Y$ be an inclusion between closed subschemes of $P$, and $\mathscr{F}$ an affinoid-acyclic sheaf on $\tube{Z}_{\mathfrak{P}}$. Then, for any $n_0\geq 1$, $\mathcal{R}_{i_Z,n_0}(\mathscr{F})$ is a complex of affinoid-acyclic sheaves on $\tube{Y}_\mathfrak{P}$,  quasi-isomorphic to $\mathbf{R}i_{Z*} \mathscr{F}$. 
\end{proposition}

\begin{proof} We start by showing that the individual terms of $\mathcal{R}_{i_Z,n_0}(\mathscr{F})$ are affinoid-acyclic. The key point is that whenever $V\subset \tube{Y}_\mathfrak{P}$ is open affinoid, so is $V\cap \left[Z\right]_n$. From this it follows easily that each 
\[ i_{Z,n*}\left(\mathscr{F}|_{\left[ Z \right]_n} \right) = \mathbf{R}i_{Z,n*}\left(\mathscr{F}|_{\left[ Z \right]_n} \right) \]
is affinoid-acyclic, and hence by Lemma \ref{lemma: aff stab} that 
\begin{equation} \label{eqn: exact}  \prod_{n\geq n_0} i_{Z,n*}\left(\mathscr{F}|_{\left[ Z \right]_n} \right) = \mathbf{R}\!\!\prod_{n\geq n_0}\mathbf{R}i_{Z,n*}\left(\mathscr{F}|_{\left[ Z \right]_n} \right)
 \end{equation}
is affinoid-acyclic, as claimed. To see that 
\[ \mathbf{R}i_{Y*} \mathscr{F} \simeq  \mathcal{R}_{i_Z,n_0}(\mathscr{F}), \]
we will let $i_n\colon  \left[ Z \right]_n\rightarrow \tube{Z}_\mathfrak{P}$ denote the inclusion from the closed tube into the open one, so that $i_{Z,n}=i_Z\circ i_n$.

\begin{claimu} The sequence
\[ 0 \rightarrow \mathscr{F} \rightarrow\prod_{n\geq n_0} i_{n*}\left(\mathscr{F}|_{\left[ Z \right]_n} \right)\overset{\pi-\mathrm{id}}{\longrightarrow}\prod_{n\geq n_0} i_{n*}\left(\mathscr{F}|_{\left[ Z \right]_n} \right) \rightarrow 0 \]
is exact on $\tube{Z}_\fr{P}$.
\end{claimu}

\begin{proof}
If we take sections on some open affinoid $V\subset \tube{Z}_\fr{P}$ we obtain the sequence
\[ 0 \rightarrow \Gamma(V,\mathscr{F}) \rightarrow \prod_{n\geq n_0} \Gamma( V\cap \left[Z\right]_n,\mathscr{F})\overset{\pi-\mathrm{id}}{\longrightarrow}  \prod_{n\geq n_0} \Gamma( V\cap \left[Z\right]_n,\mathscr{F}) \rightarrow 0,  \]
we shall show that this sequence is exact. First, we observe that the latter two terms form the standard Roos complex explicitly computing the derived inverse limit $\mathbf{R}\!\varprojlim_n \Gamma(V\cap\left[Z\right]_n ,\mathscr{F})$. We can see that $\varprojlim_n \Gamma(V\cap\left[Z\right]_n ,\mathscr{F}) = \Gamma(V,\mathscr{F})$ simply by the sheaf axiom for $\mathscr{F}$, since $V=\bigcup_n V\cap \left[Z \right]_n$. By quasi-compactness of $V$ there exists some $n_1$ such that $V\cap \left[Z\right]_n=V\cap \left[Z\right]_{n_1}$ for all $n\geq n_1$. In particular the inverse system $\left\{\Gamma(V\cap\left[Z\right]_n,\mathscr{F}) \right\}_{n}$ is eventually constant, and so $\mathbf{R}^1\!\varprojlim_n \Gamma(V\cap\left[Z\right]_n ,\mathscr{F}) =0$. The sequence is therefore exact as claimed.
\end{proof}

In particular,
\[ \prod_{n\geq n_0} i_{n*}\left(\mathscr{F}|_{\left[ Z \right]_n} \right)\overset{\pi-\mathrm{id}}{\longrightarrow}\prod_{n\geq n_0} i_{n*}\left(\mathscr{F}|_{\left[ Z \right]_n} \right) \]
is a resolution of $\mathscr{F}$. Now arguing exactly as in the proof of the equality (\ref{eqn: exact}) above, we can show that this resolution of  $\mathscr{F}$ is quasi-isomorphic to
\[ \left[ \mathbf{R}\!\!\prod_{n\geq n_0} \mathbf{R}i_{n*}\left(\mathscr{F}|_{\left[ Z \right]_n} \right)\overset{\pi-\mathrm{id}}{\longrightarrow}\mathbf{R}\!\!\prod_{n\geq n_0} \mathbf{R}i_{n*}\left(\mathscr{F}|_{\left[ Z \right]_n} \right) \right].   \]
Since (derived) pushforwards commute with (derived) products, we therefore find that
\begin{align*} \mathbf{R}i_{Z*}\mathscr{F} &\cong \left[ \mathbf{R}i_{Z*}\mathbf{R}\!\!\prod_{n\geq n_0} \mathbf{R}i_{n*}\left(\mathscr{F}|_{\left[ Z \right]_n} \right)\overset{\pi-\mathrm{id}}{\longrightarrow}\mathbf{R}i_{Z*}\mathbf{R}\!\!\prod_{n\geq n_0} \mathbf{R}i_{n*}\left(\mathscr{F}|_{\left[ Z \right]_n} \right) \right] \\
&\cong   \left[ \mathbf{R}\!\!\prod_{n\geq n_0} \mathbf{R}i_{Z,n*}\left(\mathscr{F}|_{\left[ Z \right]_n} \right)\overset{\pi-\mathrm{id}}{\longrightarrow}\mathbf{R}\!\!\prod_{n\geq n_0} \mathbf{R}i_{Z,n*}\left(\mathscr{F}|_{\left[ Z \right]_n} \right) \right] \\
&\cong   \left[ \prod_{n\geq n_0} i_{Z,n*}\left(\mathscr{F}|_{\left[ Z \right]_n} \right)\overset{\pi-\mathrm{id}}{\longrightarrow}\prod_{n\geq n_0} i_{Z,n*}\left(\mathscr{F}|_{\left[ Z \right]_n} \right) \right] \\
 &= \mathcal{R}_{i_Z,n_0}(\mathscr{F})
 \end{align*}
 as required.
\end{proof}

\subsection{Resolutions of constructible modules}

To tie together the constructions of \S\ref{sec: loc strat} and \S\ref{sec: Roos new}, we make the following definition.

\begin{definition} Let $\mathscr{F}$ be a constructible $\mathcal{O}_{\mathfrak{P}_K}$-module. A finite stratification $\left\{ P_\alpha \right\}_{\alpha\in A}$ of $P$ is said to be good with respect to $\mathscr{F}$ if, for all $\alpha\in A$:
\begin{enumerate}
\item the the open immersion $P_{\alpha}\hookrightarrow \overline{P}_\alpha$ is strongly affine;
\item $\mathscr{F}|_{\tube{P_\alpha}_\mathfrak{P}}$ is a coherent $\mathcal{O}_{\tube{P_\alpha}_\mathfrak{P}}$-module.
\end{enumerate}
\end{definition}

Clearly, any constructible $\mathcal{O}_{\mathfrak{P}_K}$-module admits a good stratification, and the category of good stratifications for a fixed constructible $\mathcal{O}_{\mathfrak{P}_K}$-module $\mathscr{F}$ is filtered. As in \S\ref{sec: loc strat} above, given such a good stratification we let $\overline{i}_\alpha:\overline{P}_\alpha \hookrightarrow P$ and $\overline{i}_{\beta\alpha}:\overline{P}_\beta \hookrightarrow \overline{P}_\alpha$ denote the given closed immersions. Then starting with Example \ref{exa: aff-acy}, and repeatedly applying Lemma \ref{lemma: aff stab} and Proposition \ref{prop: roos aff acy}, we arrive at the following. 

\begin{theorem} Let $\mathscr{F}$ be a constructible $\mathcal{O}_{\mathfrak{P}_K}$-module, and $\left\{P_{\alpha}\right\}_{\alpha\in A}$ a good stratification for $\mathscr{F}$. Then, for every chain $\{\alpha_0<\ldots < \alpha_r\}\subset A$,
\[  \mathrm{Tot}\left( \mathcal{R}_{\overline{i}_{\alpha_0},\infty}j^\dagger_{P_{\alpha_0}}\mathcal{R}_{\overline{i}_{\alpha_1\alpha_0},\infty}j^\dagger_{P_{\alpha_1}}\ldots \mathcal{R}_{\overline{i}_{\alpha_r\alpha_{r-1}},\infty}j^\dagger_{P_{\alpha_r}}\overline{i}^{-1}_{\alpha_r}\mathscr{F}\right)\]
is a complex of $\mathrm{sp}_*$-acyclic sheaves on $\mathfrak{P}_K$, quasi-isomorphic to
\[\mathbf{R}\bar{i}_{\alpha_0*} j_{P_{\alpha_0}}^\dagger \mathbf{R}\overline{i}_{\alpha_1\alpha_0*}j_{P_{\alpha_1}}^\dagger \ldots\mathbf{R}\overline{i}_{\alpha_{r}\alpha_{r-1}*}j_{P_{\alpha_r}}^\dagger \overline{i}^{-1}_{\alpha_r}\mathscr{F}  . \]
\end{theorem}

Thus setting
\[ \left( \mathcal{R}_{\overline{P}_{\bullet},\infty}j^\dagger_{P_\bullet}\mathscr{F}\right)_{\left\{ \alpha_0<\ldots <\alpha_r \right\}} = \mathrm{Tot}\left( \mathcal{R}_{\overline{i}_{\alpha_0},\infty}j^\dagger_{P_{\alpha_0}}\mathcal{R}_{\overline{i}_{\alpha_1\alpha_0},\infty}j^\dagger_{P_{\alpha_1}}\ldots \mathcal{R}_{\overline{i}_{\alpha_r\alpha_{r-1}},\infty}j^\dagger_{P_{\alpha_r}}\overline{i}^{-1}_{\alpha_r}\mathscr{F}\right) \]
gives us a(nother) particular lifting
\[ \mathcal{R}_{\overline{P}_{\bullet},\infty}j^\dagger_{P_\bullet} \mathscr{F}: \mathrm{sd}(A) \rightarrow \mathbf{Ch}(\mathfrak{P}_K)\]
of the diagram $\mathbf{R}i_{\overline{P}_\bullet *}j^\dagger_{P_\bullet}\mathscr{F}$ to a diagram of $\mathrm{sp}_*$-acyclic chain complexes on $\mathfrak{P}$.

\subsubsection{} \label{sec: concrete holim}

We next describe a procedure for constructing an explicit representative of the homotopy limit of an $\mathrm{sd}(A)$-shaped diagram. Suppose that $V$ is a topological space, $A$ is a finite poset, and
\[ \mathscr{K}_\bullet: \mathrm{sd}(A) \rightarrow \mathbf{Ch}(V)\]
is a diagram indexed by the poset of chains in $A$. For each $r\geq 0$ we consider the complex
\[ \prod_{{\left\{ \alpha_0<\ldots <\alpha_r \right\}}\in \mathrm{sd}(A)}  \mathscr{K}_{\left\{ \alpha_0<\ldots <\alpha_r \right\}} \]
and define maps
\[ \prod_{{\left\{ \alpha_0<\ldots <\alpha_r \right\}}\in \mathrm{sd}(A)}  \mathscr{K}_{\left\{ \alpha_0<\ldots <\alpha_r \right\}} \longrightarrow \prod_{{\left\{ \beta_0<\ldots <\beta_{r+1} \right\}}\in \mathrm{sd}(A)}  \mathscr{K}_{\left\{ \beta_0<\ldots <\beta_{r+1}\right\}} \]
whose $({\left\{ \alpha_0<\ldots <\alpha_r \right\}},{\left\{ \beta_0<\ldots <\beta_{r+1}\right\}})$-component is defined to be
$(-1)^s$ times the natural restriction map if $\{\alpha_0<\ldots <\alpha_r\}=\left\{ \beta_0<\ldots <\hat{\beta}_s<\ldots<\beta_{r+1}\right\}$, and zero otherwise. This gives rise to a double complex, and we define $\mathrm{Tot}_{\mathrm{sd}(A)}\mathscr{K}_\bullet$ to be the associated simple complex.

\begin{lemma} The complex $\mathrm{Tot}_{\mathrm{sd}(A)}\mathscr{K}_\bullet$ is a representative in $\mathbf{Ch}(V)$ of the homotopy limit $\mathrm{holim}_{\mathrm{sd}(A)} \mathscr{K}_\bullet\in {\bf D}(V)$.
\end{lemma}

\begin{proof}
We follow through the proof of Proposition \ref{prop: coh loc strat} explicitly. Pick a minimal element $\alpha\in A$, and let $B=A\setminus \{\alpha\}$. We then have the two diagrams
\begin{align*} \mathscr{K}_\bullet: \mathrm{sd}(B) &\rightarrow \mathrm{Ch}(V) \\
\mathscr{K}_{\alpha<\bullet}: \mathrm{sd}(B) &\rightarrow \mathrm{Ch}(V)
 \end{align*}
given by sending $(\beta_0<\ldots<\beta_r)$ to $\mathscr{K}_{\beta_0<\ldots<\beta_r}$ and $\mathscr{K}_{\alpha<\beta_0<\ldots<\beta_r}$ respectively. Moreover, these fit into a diagram of complexes
\[  \xymatrix{ & \mathscr{K}_{\alpha} \ar[d]^{f} \\ \mathrm{Tot}_{\mathrm{sd}(B)} \mathscr{K}_\bullet \ar[r]^-{g} & \mathrm{Tot}_{\mathrm{sd}(B)}\mathscr{K}_{\alpha<\bullet} } \]
in a natural way. By induction on $\norm{A}$, it therefore suffices to exhibit $\mathrm{Tot}_{\mathrm{sd}(A)}\mathscr{K}_{\bullet}$ as a homotopy fibre product of the above diagram of chain complexes. We can check directly from the definitions that $\mathrm{Tot}_{\mathrm{sd}(A)}\mathscr{K}_{\bullet}$ is isomorphic (not just quasi-isomorphic) to the mapping fibre\footnote{that is, the shifted mapping cone} of the difference map
\[ f-g\colon\mathscr{K}_{\alpha} \oplus  \mathrm{Tot}_{\mathrm{sd}(B)} \mathscr{K}_\bullet \rightarrow \mathrm{Tot}_{\mathrm{sd}(B)}\mathscr{K}_{\alpha<\bullet} . \]
It remains to show, then, that for any diagram of (bounded below) complexes
\[  \xymatrix{ & A  \ar[d]^f\\B \ar[r]^-g & C, }\]
the homotopy fibre product $A\times^h_C B$ can be computed as the mapping fibre of
\[ f-g:A\oplus B \rightarrow C.\]
This is surely well know, but let us give a short proof. 

Denote this mapping fibre of $f-g$ by $X$. The problem is unchanged upon taking a fibrant replacement of $f$ and $B$ (for the injective model structure on chain complexes described in \cite{Hov01}), hence we may assume that $f$ is a fibration and $B$ is fibrant. In this case the homotopy fibre product can be calculated simply as the fibre product $A\times_C B$, and we can explicitly check that the chain map
\[ A\times_C B \rightarrow X \]
given in degree $n$ by
\[ A^n\times_{C^n} B^n \rightarrow A^n\oplus B^n \oplus C^{n-1}\;\;\;\;(a,b)\mapsto(a,b,0)  \]
induces an isomorphism in cohomology.
\end{proof}

\begin{corollary}  Let $\mathscr{F}$ be a constructible $\mathcal{O}_{\mathfrak{P}_K}$-module, and $\left\{P	_{\alpha}\right\}_{\alpha\in A}$ a good stratification for $\mathscr{F}$. Then
\[ \mathrm{Tot}_{\mathrm{sd}(A)}\mathcal{R}_{\overline{P}_{\bullet},\infty}j^\dagger_{P_\bullet} \mathscr{F}\]
is an $\mathrm{sp}_*$-acyclic resolution of $\mathscr{F}$.
\end{corollary}

Thus we find
\[ \mathbf{R}\mathrm{sp}_*\mathscr{F}\isomto \mathrm{sp}_*\mathrm{Tot}_{\mathrm{sd}(A)} \mathcal{R}_{\overline{P}_{\bullet},\infty}j^\dagger_{P_\bullet} \mathscr{F}=\mathrm{Tot}_{\mathrm{sd}(A)} \mathrm{sp}_*\mathcal{R}_{\overline{P}_{\bullet},\infty}j^\dagger_{P_\bullet} \mathscr{F} \]
since $\mathrm{sp}_*$ commutes with products. This construction is functorial with respect to refinements of the stratification $\{P_\alpha\}_{\alpha\in A}$, in a way which we leave to the reader to make precise. 

\section{Pre-\texorpdfstring{$\mathscr{D}^\dagger$}{Ddagger}-modules } \label{sec: pre Dd}

Let $\fr{P}$ be a smooth formal scheme. In the previous section we showed how to explicitly compute $\mathbf{R}\mathrm{sp}_*\mathscr{F}$ for any constructible $\mathcal{O}_{\mathfrak{P}_K}$-module $\mathscr{F}$. Our next task is to show, given a convergent connection on $\mathscr{F}$, how to endow $\mathbf{R}\mathrm{sp}_*\mathscr{F}$ with a $\mathscr{D}^\dagger_{\mathfrak{P}\Q}$-module structure lifting the natural $\mathscr{D}_{\mathfrak{P}\Q}$-module structure. 

he key tool for doing so will be the notion, introduced below, of a pre-$\mathscr{D}^\dagger$-module. For $Y\hookrightarrow P$ a closed immersion, and $\mathrm{sp}_Y:\tube{Y}_\mathfrak{P}\rightarrow \fr{P}$ the restriction of the specialisation map, this will be an extra structure on a $\mathscr{D}_{\tube{Y}_\mathfrak{P}}$-module, which will enable us to upgrade $\mathrm{sp}_{Y*}\mathscr{F}$ from a $\mathscr{D}_{\fr{P}\Q}$-module to a $\mathscr{D}^\dagger_{\fr{P}\Q}$-module.

Crucially, this extra structure can be transported along the functors $j^\dagger_{-}$ and $\mathcal{R}_{-,\infty}$ considered in the previous section, which enables us to show that the complexes $\mathrm{Tot}_{\mathrm{sd}(A)}\mathcal{R}_{\overline{P}_{\bullet},\infty}j^\dagger_{P_\bullet}\mathscr{F}$ resolving $\mathscr{F}$ are naturally pre-$\mathscr{D}^\dagger$-modules. We therefore obtain a $\mathscr{D}^\dagger_{\fr{P}\Q}$-module structure on
\[ \mathbf{R}\mathrm{sp}_*\mathscr{F}=\mathrm{sp}_*\mathrm{Tot}_{\mathrm{sd}(A)}\mathcal{R}_{\overline{P}_{\bullet},\infty}j^\dagger_{P_\bullet}\mathscr{F}. \]
The notion of a pre-$\mathscr{D}^\dagger$-module is quite delicate, and dependent on endowing the sheaves under consideration with topologies. We therefore start by recalling how this can be done. 

\subsection{Fr\'echet $\pmb{\mathcal{O}_\mathscr{X}}$- and $\pmb{\mathscr{D}_{\mathscr{X}}}$-modules}

Let $\mathbf{LC}_K$ denote the category of locally convex $K$-vector spaces. An object of $\mathbf{LC}_K$ is therefore a $K$-vector space $V$, equipped with a uniform topology, admitting a neighbourhood basis of $0$ consisting of $\mathcal{V}$-submodules $L \subset V$ spanning $V$. The category $\mathbf{LC}_K$ admits all limits and colimits, both of which commute with the forgetful functor
\[\mathbf{LC}_K \rightarrow \mathbf{Vec}_K \]
to the category of abstract vector spaces. It follows that an $\mathbf{LC}_K$-valued presheaf on a site is a sheaf if and only if its underlying $\mathbf{Vec}_K$-valued presheaf is a sheaf. In an entirely similar manner, we may replace $\mathbf{LC}_K$ by the category $\mathbf{LCA}_K$ of locally convex $K$-algebras (that is, $K$-algebras $A$ with a locally convex topology for which multiplication $A\times A\rightarrow A$ is continuous), and thus speak about sheaves of locally convex $K$-algebras on a site. If $\mathscr{A}$ is such a sheaf, then a locally convex $\mathscr{A}$-module is a sheaf of locally convex $K$-vector spaces $\mathscr{F}$ together with a morphism
\[ \mathscr{A}\times \mathscr{F} \rightarrow \mathscr{F}. \]
of $\mathbf{LC}_K$-valued sheaves, giving rise to an action of $\mathscr{A}$ on $\mathscr{F}$. Be warned that `locally convex' here refers to $\mathscr{F}$ being locally convex \emph{over} $K$. It does \emph{not} mean that we have chosen a suitable integral subring $\mathscr{A}^+\subset \mathscr{A}$ and require each $\Gamma(V,\mathscr{F})$ to have a neighbourhood base consisting $\Gamma(V,\mathscr{A}^+)$-lattices.

\begin{example} If $\mathscr{X}\in\mathbf{An}_K$, then $\mathcal{O}_{\mathscr{X}}$ is naturally a sheaf of locally convex $K$-algebras on $\mathscr{X}$.
\end{example}

\begin{definition} Let $\mathscr{X}\in\mathbf{An}_K$ be an analytic variety, and $\mathscr{A}$ a sheaf of locally convex $K$-algebras on $\mathscr{X}$. A Fr\'echet $\mathscr{A}$-module is a locally convex $\mathscr{A}$-module $\mathscr{F}$ such that, for every open affinoid $V\subset\mathscr{X}$, the locally convex $K$-vector space $\Gamma(V,\mathscr{F})$ is a Fr\'echet space.
\end{definition}

Note that this property only depends on the underling sheaf of locally convex $K$-vector spaces. The reason for restricting to analytic varieties is the existence of the natural base of quasi-compact opens given by the open affinoids in $\mathscr{X}$.

As a first example, we see that $\mathcal{O}_\mathscr{X}$ is a Fr\'echet module over itself, since if $V\subset \cO_{\mathscr{X}}$ is affinoid, then $\Gamma(V,\cO_\mathscr{X})$ is a Banach space over $K$.

\begin{remark} Since finite limits of Fr\'echet spaces are Fr\'echet spaces, the property of a locally convex $\cO_{\mathscr{X}}$-module being a Fr\'echet module can be checked locally on $\mathscr{X}$. 
\end{remark}

We can consider $\mathscr{D}_\mathscr{X}$ as a locally convex $\mathcal{O}_\mathscr{X}$-module with the initial locally convex topology. Concretely, if $x_1,\ldots,x_d$ are \'etale co-ordinates on $\mathscr{X}$, with corresponding derivations $\partial_1,\ldots,\partial_d$, then 
\[ \mathscr{D}_\mathscr{X} = \underset{n}{\mathrm{colim}} \bigoplus_{k_1+\ldots+k_d\leq n} \mathcal{O}_\mathscr{X}\cdot \partial_1^{k_1}\ldots\partial_d^{k_d} \]
is given the direct limit topology coming from the direct sum topology on each piece. If $\mathscr{F}$ is a locally convex $\mathcal{O}_\mathscr{X}$-module together with a compatible (abstract) $\mathscr{D}_{\mathscr{X}}$-module structure, then $\mathscr{F}$ is a locally convex $\mathscr{D}_{\mathscr{X}}$-module iff all local derivations act continuously on sections of $\mathscr{F}$. v

We can similarly talk about sheaves of locally convex $K$-modules or algebras on $\fr{P}$. If $\mathscr{A}$ is a sheaf of locally convex $K$-algebras on $\fr{P}$ we can talk about locally convex or Fr\'echet $\mathscr{A}$-modules, the latter by requiring sections on open affines in $\fr{P}$ to be Fr\'echet spaces. For example, $\cO_{\fr{P}\Q}$ will be a Fr\'echet $\mathscr{D}_{\fr{P}\Q}$-module.

\subsection{The radial filtration and $\bm{\eta}$-admissible neighbourhoods}

The difficulty in upgrading $\mathbf{R}\mathrm{sp}_*$ to take values in ${\bf D}^b(\mathscr{D}^\dagger_{\mathfrak{P}\Q})$ essentially comes from the fact that there is no natural action of $\mathrm{sp}^{-1}\mathscr{D}^\dagger_{\mathfrak{P}\Q}$ on $\mathcal{O}_{\mathfrak{P}_K}$. To get around this problem, we introduce the notion of an $\eta$-admissible neighbourhood, which itself is defined in terms of what we call the radial filtration on the ring $\mathscr{D}^\dagger_{\mathfrak{P}\Q}$ of arithmetic differential operators. While this filtration generally has worse ring-theoretic properties than the usual filtration by the level $m$, it is better adapted to rigid analytic calculations. This is essentially because $\widehat{\mathscr{D}}^{(m)}_{\mathfrak{P}\Q}$ cannot (locally) be expressed as the set of power series in the divided powers $\partial^{[k]}$ of derivations, with a prescribed radius of convergence.

\subsubsection{} Let $p_1,_2:\tube{P}_{\fr{P}^2}\rightarrow \fr{P}_K$ be the projections, and $\mathrm{sp}:\fr{P}_K\rightarrow \fr{P}$ the specialisation map. We first construct a natural pairing
\begin{equation} \label{eqn: pair}  \mathscr{D}^\dagger_{\fr{P}\Q} \times \mathrm{sp}_*p_{1*}\mathcal{O}_{\tube{P}_{\fr{P}^2}} \rightarrow \mathcal{O}_{\fr{P}\Q}
\end{equation}
of $\mathcal{O}_{\fr{P}\Q}$-modules. To do so, we denote by $\mathscr{P}_{i,(m)}^n$ the sheaf of principal parts of level $m$ and order $n$ on $P_i:=\fr{P}\otimes_{\mathcal{V}} \mathcal{V}/\varpi^{i+1}$, as defined in \cite[\S1]{Ber02}. Thus by definition
\begin{align*}
\mathscr{D}^{(m)}_{P_i,n} &:= \mathrm{Hom}_{\mathcal{O}_{P_i}}(\mathscr{P}_{i,(m)}^n,\mathcal{O}_{P_i}) 
 & \widehat{\mathscr{D}}^{(m)}_{\fr{P}} &:=\mathrm{lim}_i \mathscr{D}^{(m)}_{P_i} & \mathscr{D}^\dagger_{\fr{P}\Q} &:= \mathrm{colim}_m \widehat{\mathscr{D}}^{(m)}_{\fr{P}\Q} \\ \mathscr{D}^{(m)}_{P_i} &:= \mathrm{colim}_n \mathscr{D}^{(m)}_{P_i,n} & \widehat{\mathscr{D}}^{(m)}_{\fr{P}\Q} &:=\widehat{\mathscr{D}}^{(m)}_{\fr{P}} \otimes_{\Z} \Q
\end{align*}
where $\mathscr{P}_{i,(m)}^n$ is considered as an $\mathcal{O}_{P_i}$-module via the first projection. On the other hand, we can consider the $m$-PD envelope $\mathscr{P}_{i,(m)}$ of $P_i$ inside $P_i\times_{\mathcal{V}/\varpi^{i+1}} P_i$. This maps in a compatible way to each $\mathscr{P}^n_{i,(m)}$, and so we have a pairing
\[  \mathscr{D}^{(m)}_{P_i} \times \mathscr{P}_{i,(m)} \rightarrow \mathcal{O}_{P_i}. \]
If we define
\begin{align*} 
\widehat{\mathscr{P}}_{(m)}&:=\mathrm{lim}_i\mathscr{P}_{i,(m)} \\ \widehat{\mathscr{P}}_{(m),\Q} &:= \widehat{\mathscr{P}}_{(m)}\otimes_{\Z}\Q  
\end{align*}
then we get natural pairings
\begin{align*}
\widehat{\mathscr{D}}^{(m)}_{\fr{P}} \times \widehat{\mathscr{P}}_{(m)} &\rightarrow \mathcal{O}_\fr{P} \\
\widehat{\mathscr{D}}^{(m)}_{\fr{P}\Q} \times \widehat{\mathscr{P}}_{(m),\Q} &\rightarrow \mathcal{O}_{\fr{P}\Q} \\
\mathscr{D}^\dagger_{\fr{P}\Q} \times \mathrm{lim}_m\widehat{\mathscr{P}}_{(m),\Q} &\rightarrow \mathcal{O}_{\fr{P}\Q}.
\end{align*}
The pairing (\ref{eqn: pair}) we are after then comes from identifying $\mathrm{lim}_m\widehat{\mathscr{P}}_{(m),\Q}$ with $\mathrm{sp}_*p_{1*}\mathcal{O}_{\tube{P}_{\fr{P}^2}}$, which follows, for example, from \cite[(2.3)]{Cre04}.

\subsubsection{} 
If we have any local section $P\in \mathscr{D}^\dagger_{\fr{P}\Q}$, defined on some open subset of $\fr{P}$, this then induces a map
\[  \mathrm{sp}_*p_{1*}\mathcal{O}_{\tube{P}_{\fr{P}^2}} \rightarrow \mathcal{O}_{\fr{P}\Q} \]
defined on that same open subset. 

\begin{definition} For $\norm{\varpi}<\eta<1$ we define the subsheaf $\mathscr{D}^\eta_{\fr{P}\Q}\subset \mathscr{D}^\dagger_{\fr{P}\Q}$ to consist of those sections which extend to a continuous $\mathcal{O}_{\fr{P}\Q}$-linear map
\[ \mathrm{sp}_*p_{1*}\mathcal{O}_{\left[P\right]_{\fr{P}^2\eta}} \rightarrow \mathcal{O}_{\fr{P}\Q}. \]
\end{definition}

Note that $\mathrm{sp}_*p_{1*}\mathcal{O}_{\left[P\right]_{\fr{P}^2\eta}}$ is naturally a sheaf of Fr\'echet $\mathcal{O}_{\fr{P}\Q}$-modules, and so the continuity condition makes sense. Also note that since $\mathrm{sp}_*p_{1*}\mathcal{O}_{\tube{P}_\fr{P}^2}$ is dense in $\mathrm{sp}_*p_{1*}\mathcal{O}_{\left[P\right]_{\fr{P}^2\eta}}$ such an extension, if it exists, is unique.

\subsubsection{}

If we are given local co-ordinates $x_1,\ldots,x_d$ on $\fr{P}$, then $\mathscr{D}^\eta_{\fr{P}\Q}$ has a relatively concrete description. Indeed, if we let $\partial_1,\ldots,\partial_d$ be the corresponding derivations, and, for any $k=(k_1,\ldots,k_d)\in \Z_{\geq0}^d$, set
\begin{align*}
\norm{k} &=k_1+\ldots+k_d \\
\partial^{[k]}&=\frac{1}{k_1!\ldots k_d!}\partial_1^{k_1}\ldots\partial_d^{k_d}
\end{align*}  
in the usual way, and fix an affinoid norm $\Norm{\,\cdot\,}$ on local sections of $\mathcal{O}_{\fr{P}\Q}$, then
\[ 	\mathscr{D}^\eta_{\mathfrak{P}\Q} = \left\{ \sum_{k} a_{k} \partial^{[k]} \left\vert a_{k}\in \mathcal{O}_{\mathfrak{P}\Q},\;\exists c \in \R_{\geq0} \text{ s.t. }\Norm{a_{k}}\leq c\eta^{\norm{k}}  \right. \right\}.\] 
Without necessarily assuming that $\mathfrak{P}$ admits co-ordinates, this implies that $\mathscr{D}^\eta_{\fr{P}\Q}$ is a sheaf of Fr\'echet $\mathcal{O}_{\fr{P}\Q}$-modules, and moreover that
\[ \mathscr{D}^\dagger_{\fr{P}\Q}= \mathrm{colim}_{\eta<1}\mathscr{D}^\eta_{\fr{P}\Q}\]
by \cite[4.1.1.1]{Ber02}. It also follows from this local description of $\mathscr{D}^\eta_{\fr{P}\Q}$ that it is a \emph{subring} of $\mathscr{D}^\dagger_{\fr{P}\Q}$, since
\[ \partial^{[k]}\partial^{[k']} = \frac{(k+k')!}{k!(k')!} \partial^{[k+k']}\]
and $\Norm{\frac{(k+k')!}{k!(k')!}}\leq 1$. 

This radial filtration on $\mathscr{D}^\dagger_{\fr{P}\Q}$ is closely connected with the following notion of an `$\eta$-admissible' subset of $\fr{P}_K$, which, roughly speaking, amounts to saying that their ring of functions is acted on by $\mathscr{D}^\eta_{\fr{P}\Q}$.

\begin{definition} Let $W\subset \mathfrak{P}_K$ be an open subset, and $\norm{\varpi}<\eta<1$. Then $W$ is called $\eta$-admissible if
\[ p_2\left(p_1^{-1}(W)\cap\left[P\right]_{\fr{P}^2\eta}\right)\subset W,\]
where $p_i:\tube{P}_{\fr{P}^2}\rightarrow \fr{P}_K$ are the two projections. 
\end{definition}

\begin{remark}  \label{rem: m adm inter} \begin{enumerate} \item  Suppose that $\mathfrak{P}$ is affine, and that $W\subset\mathfrak{P}_K$ is affinoid. Given the definition of $\mathscr{D}^\eta_{\mathfrak{P}\Q}$, we see that $W$ is $\eta$-admissible if and only if the $\Gamma(\mathfrak{P},\mathscr{D}_{\mathfrak{P}\Q})$-module structure on $\Gamma(W,\mathcal{O}_W)$ extends to a continuous $\Gamma(\mathfrak{P},\mathscr{D}^\eta_{\mathfrak{P}\Q})$-module structure. Note that since $\Gamma(\mathfrak{P},\mathscr{D}^\eta_{\mathfrak{P}\Q})$ and $\Gamma(W,\mathcal{O}_W)$ are both Banach spaces over $K$, and $\Gamma(\mathfrak{P},\mathscr{D}_{\mathfrak{P}\Q})$ is dense in $\Gamma(\mathfrak{P},\mathscr{D}^\eta_{\mathfrak{P}\Q})$, such an extension, if it exists, is necessarily unique. 
\item  \label{rem: m adm inter num} It is clear from the definition that the intersection of two $\eta$-admissible subsets of $\mathfrak{P}_K$ is $\eta$-admissible.  
\end{enumerate}
\end{remark}

\begin{example}\begin{enumerate}
\item  Let $\mathfrak{P}=\widehat{\A}^1_\mathcal{V}=\spf{\mathcal{V}\tate{z}}$ and $U=D(z)=\G_{m,k}\subset P =\A^1_k$. A cofinal system of neighbourhoods of $]U[_\mathfrak{P}$ inside $\mathfrak{P}_K=\D_K(0;1)$ is given by
\[ W_\lambda:= \left\{x \in \fr{P}_K \mid v_x(z) \geq \lambda \right\} \subset \D_K(0;1)  \]
for $0<\lambda<1$, and $W_\lambda$ is $\eta$-admissible if and only if $\lambda\geq \eta$.
\item More generally, suppose that $\mathfrak{P}$ admits \'etale co-ordinates, and $U\subset P$ is open affine, with complement $D$. Set $W_\lambda$ to be the interior of $\mathfrak{P}_K\setminus \tube{D}_{\mathfrak{P}\lambda}$. Thus, if $D$ is defined by $f=0$, then $W_\lambda$ is defined by $v(f)\geq \lambda$. It then follows from \cite[Lemma 4.3.10]{LS07} that $W_\lambda$ is $\eta$-admissible for all $\lambda$ sufficiently close to $1$.
\end{enumerate}
\end{example}

\subsection{Pre-$\pmb{\mathscr{D}^\dagger}$-modules} \label{sec: pre Ddag sub}

We can now introduce the key (and rather delicate) notion of a pre-$\mathscr{D}^\dagger$-module, which will be our main technical tool in upgrading $\mathbf{R}\mathrm{sp}_*$ to take values in ${\bf D}^b(\mathscr{D}^\dagger_{\mathfrak{P}\Q})$. Let $Y\hookrightarrow P$ a closed immersion.

\begin{definition} \label{defn: pre Ddag} A pre-$\mathscr{D}^\dagger$-module on $\tube{Y}_\fr{P}$ is an ind-object $\left\{ \mathscr{F}_i \right\}_{i\in I}$ in the category of Fr\'echet $\mathscr{D}_{\tube{Y}_\mathfrak{P}}$-modules, such that:
\begin{itemize}
\item for all $\eta<1$ sufficiently close to $1$, and all open affine subschemes $\mathfrak{P}'\subset \mathfrak{P}$, there exists $i_\eta\in I$, and $T_\eta\subset \tube{Y}_{\mathfrak{P}'}$ an open affinoid, such that:
\begin{itemize}
\item for all affine open immersions $U\hookrightarrow P'$, all $\eta$-admissible affinoid neighbourhoods $W$ of $\tube{U}_{\mathfrak{P}'}$ inside $\mathfrak{P}'$, all $i\geq i_\eta$, and all affinoids $T_\eta\subset T \subset \tube{Y}_{\mathfrak{P}'} $, the $\Gamma(\mathfrak{P},\mathscr{D}_{\mathfrak{P}'\Q})$-module structure on $\Gamma(T\cap W,\mathscr{F}_{i})$ extends to a continuous $\Gamma(\mathfrak{P},\mathscr{D}^\eta_{\mathfrak{P}'\Q})$-module structure.
\end{itemize}
\end{itemize}

\end{definition}

Again, since $\Gamma(T\cap W,\mathscr{F}_{i})$ is a Fr\'echet space, the $\Gamma(\mathfrak{P},\mathscr{D}^\eta_{\mathfrak{P}'\Q})$-module structure on $\Gamma(T\cap W,\mathscr{F}_{i})$ is unique if it exists.  A morphism of pre-$\mathscr{D}^\dagger$-modules is simply a morphism of ind-objects.

\begin{definition}If $\mathscr{F}$ is a $\mathscr{D}_{\tube{Y}_\mathfrak{P}}$-module, then a pre-$\mathscr{D}^\dagger$-module structure on $\mathscr{F}$ is a pre-$\mathscr{D}^\dagger$-module $\left\{\mathscr{F}_i\right\}_{i\in I}$, together with an isomorphism 
\[ \mathrm{colim}_{i\in I} \mathscr{F}_i \isomto \mathscr{F} \]
of $\mathscr{D}_{\tube{Y}_\mathfrak{P}}$-modules.
\end{definition} 

If $\mathscr{F}, \mathscr{G}$ are endowed with pre-$\mathscr{D}^\dagger$-module structures $\left\{ \mathscr{F}_i\right\}_{i\in I}$ and $\{\mathscr{G}_j\}_{j\in J}$ respectively, and $\alpha:\mathscr{F}\rightarrow \mathscr{G}$ is a morphism, then a pre-$\mathscr{D}^\dagger$-lift of $\alpha$ is morphism of pre-$\mathscr{D}^\dagger$-modules which recovers $\alpha$ upon taking the colimit.

Let $\mathrm{sp}_{Y}\colon\tube{Y}_\mathfrak{P}\rightarrow \mathfrak{P}$ denote the restriction of the specialisation map to $\tube{Y}_\mathfrak{P}$, and, as in \S\ref{sec: Roos new} above, let $\left[Y \right]_n$ for $n\geq 1$ denote the closed tube of radius $\eta_n$. Let $\mathrm{sp}_{Y,n}:\left[Y \right]_n \rightarrow\mathfrak{P}_K$ again denote the restriction of the specialisation map. 

\begin{lemma} \label{lemma: pre Ddag sp} Any pre-$\mathscr{D}^\dagger$-module structure on $\mathscr{F}$ induces a $\mathscr{D}^\dagger_{\mathfrak{P}\Q}$-module structure on $\mathrm{sp}_{Y*}\mathscr{F}$ extending the natural $\mathscr{D}_{\mathfrak{P}\Q}$-module structure. If $\alpha:\mathscr{F}\rightarrow  \mathscr{G}$ is a morphism of $\mathscr{D}_{\tube{Y}_\fr{P}}$-modules, then any pre-$\mathscr{D}^\dagger$-lift of $\alpha$ induces a $\mathscr{D}^\dagger_{\fr{P}\Q}$-linear lift of $\mathrm{sp}_{Y*}\alpha$.
\end{lemma}

\begin{proof} To see that $\mathrm{sp}_{Y*}\mathscr{F}$ is a $\mathscr{D}^\dagger_{\mathfrak{P}\Q}$-module, we simply take $U=\mathfrak{P}'_k$ in the definition and let $n$ be large enough so that $T_\eta \subset \left[Y \right]_n$. Thus for all suitably large $i$ (depending on $\eta$), $\Gamma(\mathfrak{P}'_K\cap \left[Y \right]_n,\mathscr{F}_i)$ is naturally a $\Gamma(\mathfrak{P},\mathscr{D}^\eta_{\mathfrak{P}'\Q})$-module. Continuity implies that these structures are compatible as $\mathfrak{P}'$ varies, hence $ \mathrm{sp}_{Y,n*}\left(\mathscr{F}_i|_{\left[Y\right]_n}\right)$ is naturally a $\mathscr{D}^\eta_{\mathfrak{P}\Q}$-module. Again, by continuity, these structures are compatible as $i$ varies, hence $\mathrm{sp}_{Y,n*}\left(\mathscr{F}|_{\left[Y\right]_n}\right)=\mathrm{colim}_{i\in I} \mathrm{sp}_{Y,n*}\left(\mathscr{F}_i|_{\left[Y\right]_n}\right)$ is again a $\mathscr{D}^\eta_{\mathfrak{P}\Q}$-module (for $n$ sufficiently large). Once more, by continuity, these structures are compatible as $n$ and $\eta$ vary, so
\[ \mathrm{sp}_{Y*}\mathscr{F}=\mathrm{lim}_n \mathrm{sp}_{Y,n*}\left(\mathscr{F}|_{\left[Y\right]_n}\right)\]
is a $\mathscr{D}^\dagger_{\mathfrak{P}\Q}$-module. Again, functoriality in $\mathscr{F}$ follows from continuity.
\end{proof}

Given the proof of the lemma, the definition of a pre-$\mathscr{D}^\dagger$-module may seem like overkill, since we only used the particular case when $U=P'$. However, the full force of Definition \ref{defn: pre Ddag} will be required to prove that pre-$\mathscr{D}^\dagger$-modules can be transported along the functors $j_{-}^\dagger$ and $\mathcal{R}_{-,\infty}$ considered in \S\ref{sec: Roos new} above (see \S\ref{sec: Ddagger stab} below).

\subsection{Pre-$\mathscr{D}^\dagger$-module structures on overconvergent isocrystals}  \label{sec: fund}

The fundamental example of a pre-$\mathscr{D}^\dagger$-module structure is that on an overconvergent isocrystal, which we describe in this section. 

Let $Y\hookrightarrow P$ be a closed immersion, $X\hookrightarrow Y$ be a strongly affine open immersion, and $\mathscr{F}$ a coherent $j_X^\dagger\mathcal{O}_{\tube{Y}_\mathfrak{P}}$-module with (over)convergent integrable connection. As in \S\ref{sec: Roos new} above, set $\left[Y\right]_n:=\left[Y\right]_{\mathfrak{P}\eta_n}$ to be the closed tube of radius $\eta_n=\norm{\varpi}^{\frac{1}{n+1}}$. 

For $\lambda<1$ sufficiently close to $1$, we let $V_\lambda$ be the subspace of $\fr{P}_K$ locally defined by $v(f)\geq \lambda$, where $f\in\cO_\fr{P}$ is such that $X=Y\cap D(f)$. Then a cofinal system of neighbourhoods of $\tube{X}_\mathfrak{P}$ inside $\tube{Y}_\mathfrak{P}$ is given by the collection of open subsets
\[ V_{\underline{\lambda}}:=\bigcup_{n\geq 0} \left[Y\right]_n \cap V_{\lambda_n} \]
indexed by \emph{increasing} sequences $\underline{\lambda}:=\{\lambda_n\}_{n\geq 1}$ of real numbers $\lambda_n < 1$. Let $(\mathcal{S},\leq)$ denote the poset of such sequences, endowed with the usual partial order, that is, the one where $\underline{\mu}\geq \underline{\lambda}$ iff $\mu_n\geq\lambda_n$ for all $n$. 

We will need to consider a slightly different partial order on the set of such sequences. Given a sequence $\underline{\lambda}$ we define
\[ m(\underline{\lambda}):= \inf \left\{n\mid \lambda_n \neq \lambda_1 \right\}.\]
We then say that $\underline{\mu}\gtrsim \underline{\lambda}$ if and only if:
\begin{enumerate}
\item $\mu_n\geq \lambda_n$ for all $n$;
\item $m(\underline{\mu})\geq m(\underline{\lambda})$.
\end{enumerate}
The second condition here says that the `initially constant' segment of $\underline{\mu}$ has to be at least as long as that of $\underline{\lambda}$. There is an obvious map $(\mathcal{S},\lesssim)\rightarrow (\mathcal{S},\leq)$ of posets, which is the identity on underlying sets. 

\begin{lemma} \label{lemma: funny order}\begin{enumerate}
\item The map $(\mathcal{S},\lesssim)\rightarrow (\mathcal{S},\leq)$ is cofinal.
\item For any $n_0\in \N$, there exists $\underline{\lambda}_0\in\mathcal{S}$ such that for all $\underline{\lambda}\gtrsim\underline{\lambda}_0$, and all $n\geq n_0$, there exists an open cover of $[Y]_{n}\cap V_{\underline{\lambda}}$ by open affinoids of the form $[Y]_{n'}\cap V_\lambda$ with $n'\geq n_0$ and $\lambda\geq \lambda_{n'}$. 
\end{enumerate}  
\end{lemma}

\begin{proof}
These are both easily verified.
\end{proof}

Note that the first claim shows that we can take colimits with respect to either $\leq$ or $\lesssim$. The second claim in the lemma is not true for the usual partial order $\leq$, and is the reason for introducing $\lesssim$.

Let us denote by $j_{\underline{\lambda}}\colon V_{\underline{\lambda}} \rightarrow  \tube{Y}_\mathfrak{P}$ the given immersions. After possibly passing to a cofinal subset of $\mathcal{S}$, we can assume that our overconvergent isocrystal $\mathscr{F}$ extends to a coherent sheaf with integrable connection $\mathscr{F}_{\underline{\lambda}}$ on each $V_{\underline{\lambda}}$, thus
\[  \mathrm{colim}_{ \underline{\lambda}\in\mathcal{S}} j_{\underline{\lambda}*}\mathscr{F}_{\underline{\lambda}}\isomto \mathscr{F} .\]
Moreover, if $T\subset \tube{Y}_\mathfrak{P}$ is affinoid, then $T\cap V_{\underline{\lambda}}$ is a finite union of affinoids (because $T$ is quasi-compact), so
\[ \Gamma(T,j_{\underline{\lambda}*}\mathscr{F}_{\underline{\lambda}}) = \Gamma(T\cap V_{\underline{\lambda}},\mathscr{F}_{\underline{\lambda}}  ) \]
is a finite limit of Banach spaces, hence a Banach space. Thus $\left\{ j_{\underline{\lambda}*}\mathscr{F}_{\underline{\lambda}}\right\}_{\underline{\lambda}\in\mathcal{S}}$ is an ind-object in the category of Fr\'echet $\mathscr{D}_{\tube{Y}_\mathfrak{P}}$-modules. 

\begin{theorem} \label{theo: pre Ddag over} The ind-object $\left\{ j_{\underline{\lambda}*}\mathscr{F}_{\underline{\lambda}}\right\}_{\underline{\lambda}\in\mathcal{S}}$ together with the isomorphism
\[ \mathrm{colim}_{ \underline{\lambda}\in\mathcal{S}} j_{\underline{\lambda}*}\mathscr{F}_{\underline{\lambda}} \isomto \mathscr{F} \]
defines a pre-$\mathscr{D}^\dagger$-module structure on $\mathscr{F}$.
 \end{theorem}

\begin{proof}
We need to prove that $\left\{ j_{\underline{\lambda}*}\mathscr{F}_{\underline{\lambda}}\right\}_{\underline{\lambda}\in\mathcal{S}}$ is a pre-$\mathscr{D}^\dagger$-module. The question is local on $\mathfrak{P}$, which we may assume therefore assume to be affine, say $\mathfrak{P}=\spf{R}$. We may also assume that $\mathfrak{P}$ admits \'etale co-ordinates $x_1,\ldots x_d$, with corresponding derivations $\partial_1,\ldots,\partial_d$. We will check that the conditions of $\left\{ j_{\underline{\lambda}*}\mathscr{F}_{\underline{\lambda}}\right\}_{\underline{\lambda}\in\mathcal{S}}$ being a pre-$\mathscr{D}^\dagger$-module are satisfied for $\mathfrak{P}'=\mathfrak{P}$. 

Given $\eta$, we will take $T_\eta$ to be the closed tube $\left[  Y \right]_{\eta}$. We need to choose a sequence $\underline{\lambda}$, such that for all open affines $U\hookrightarrow P$, all $\eta$-admissible affinoid neighbourhoods $\tube{U}_\mathfrak{P}\subset W \subset \mathfrak{P}$, all $\underline{\mu}\gtrsim \underline{\lambda}$, and all open affinoids $\left[Y\right]_{\eta} \subset T\subset \tube{Y}_\mathfrak{P}$,
\[ \Gamma(T \cap W,j_{\underline{\mu}*}\mathscr{F}|_{V_{\underline{\mu}}})=\Gamma(T\cap W \cap V_{\underline{\mu}},\mathscr{F}_{\underline{\mu}})\]
is naturally a $\Gamma(\fr{P},\mathscr{D}^\eta_{\fr{P}\Q})$-module.

A cofinal system of such affinoids $T$ is given by the family $\left[Y\right]_n$ for $n$ such that $\eta_n\geq \eta$. Similarly, if we let $g\in \Gamma(\mathfrak{P},\mathcal{O}_{\mathfrak{P}})$ be a function such that $X=D(g)\cap Y$, then a cofinal family of possible such affinoids $W$ is given by
\[ W_\rho:= \left\{ \left. x\in \mathfrak{P}_K \right\vert v_x(g) \geq \rho \right\} \]
for $\rho<1$ sufficiently close to $1$. Hence in constructing our sequence $\underline{\lambda}$ it suffices to consider $T$ and $W$ of this form.

Take $n_0$ to be minimal such that $\eta_{n_0}\geq \eta$. By Lemma \ref{lemma: funny order}, for all sufficiently large $\underline{\lambda}$,\footnote{that is, all $\underline{\lambda}\gtrsim \underline{\lambda}_0$ for some fixed $\underline{\lambda}_0$} and all $n\geq n_0$, we can cover $[Y]_n\cap V_{\underline{\lambda}}$ with open affinoids of the form $[Y]_{n'}\cap V_\lambda$ for $n'\geq n_0$ and $\lambda\geq \lambda_{n'}$. Hence it suffices to show the following:
\begin{itemize}\item[$(\star)$] for all $n\geq n_0$ there exists $\lambda_n$ such that;
\begin{enumerate}
\item $\mathscr{F}$ extends to a module with integrable connection $\mathscr{F}_{\lambda_n}$ on $[Y]_n\cap V_{\lambda_n}$;
\item for all $\eta$-admissible $W_\rho$, all $\lambda\geq \lambda_n$, and all sections
\begin{align*}
m &\in \Gamma(\left[Y\right]_n \cap W_\rho \cap V_\lambda, \mathscr{F}_{\lambda_n}) \\
P=\sum_k a_k\partial^{[k]} &\in \Gamma(\mathfrak{P},\mathscr{D}^\eta_{\mathfrak{P}\Q}),
\end{align*} 
the sequence $a_k\partial^{[k]}(m)$ converges to zero in $\Gamma(\left[Y\right]_n \cap W_\rho \cap V_\lambda, \mathscr{F}_{\lambda_n})$. 
\end{enumerate} 
\end{itemize}
To find $\lambda_n$, we extend some fixed Banach norm $\Norm{\,\cdot\,}$ on $R_{\Q}$ to compatible norms:
\begin{itemize}
\item $\Norm{\,\cdot\,}_{n,\lambda}$ on each $\Gamma(\left[Y\right]_n\cap V_\lambda,\mathcal{O}_{\fr{P}_K})$;
\item $\Norm{\,\cdot\,}_{\rho,n,\lambda}$ on each $\Gamma(\left[Y\right]_n\cap W_\rho \cap V_\lambda,\mathcal{O}_{\fr{P}_K})$;
\item $\Norm{\,\cdot\,}_{n,\lambda}$ on each $\Gamma(\left[Y\right]_n\cap V_\lambda,\mathscr{F}_{\lambda_n})$ as a Banach $\Gamma(\left[Y\right]_n\cap V_\lambda,\mathcal{O}_{\fr{P}_K})$-module;
\item $\Norm{\,\cdot\,}_{\rho,n,\lambda}$ on each $\Gamma(\left[Y\right]_n\cap  W_\rho \cap V_\lambda,\mathscr{F}_{\lambda_n})$ as a Banach $\Gamma(\left[Y\right]_n\cap  W_\rho \cap V_\lambda,\mathcal{O}_{\fr{P}_K})$-module.
\end{itemize}
Appealing to \cite[Theorem 4.3.9]{LS07}, we see that for all $n$, there exists some $\lambda'_n$ such that for all $\lambda\geq \lambda'_n$, and all $m\in\Gamma(\left[Y\right]_n\cap V_\lambda,\mathscr{F}_{\lambda'_n})$, 
\[ \Norm{\partial^{[k]}(m)}_{n,\lambda} \eta_n^{\norm{k}} \rightarrow 0\text{ as }\norm{k}\rightarrow \infty.\]
Since $\left[Y\right]_n\cap V_\lambda$ and $\left[Y\right]_n\cap W_\rho \cap V_\lambda$ are affinoids, we know that for all $\lambda\geq \lambda'_n$, the space of sections $\Gamma(\left[Y\right]_n\cap W_\rho \cap V_\lambda,\mathscr{F}'_{\lambda_n})$ is generated by $\Gamma(\left[Y\right]_n\cap V_\lambda,\mathscr{F}'_{\lambda_n})$ as a $\Gamma(\left[Y\right]_n\cap W_\rho\cap V_\lambda,\mathcal{O}_{\fr{P}_K})$-module. If we choose $c$ such that $\Norm{a_{k}}\leq c\eta^{\norm{k}}$, then, for all $f\in \Gamma(\left[Y\right]_n\cap W_\rho\cap V_\lambda,\mathcal{O}_{\fr{P}_K})$ and $m\in \Gamma(\left[Y\right]_n\cap V_\lambda,\mathscr{F}'_{\lambda_n})$, we have
\[ \Norm{a_{k}\partial^{[k]}(fm)}_{\rho,n,\lambda}  \leq \max_{k_1+k_2=k} c\eta^{\norm{k_1}}\Norm{\partial^{[k_1]}(f)}_{\rho,n,\lambda}  \eta^{\norm{k_2}}\Norm{\partial^{[k_2]}(m)}_{n,\lambda}  \]
by the Leibniz rule. Provided that $\eta\leq \eta_n$ and $\lambda\geq \lambda'_n$, the latter term always $\rightarrow 0$ as $\norm{k}\rightarrow \infty$. Hence, by writing each $m$ as a suitable $\Gamma(  \left[Y\right]_n \cap W_\rho \cap V_\lambda,\mathcal{O}_{\fr{P}_K})$-linear combination of elements of $\Gamma(\left[Y\right]_n \cap V_{\lambda_n},\mathscr{F}_{\lambda_n})$, it suffices to show that we can choose $\lambda_n\geq \lambda'_n$ such that for all $n\geq n_0$, all $\eta$-admissible $W_\rho$, all $\lambda \geq \lambda_n$, and all $f\in \Gamma(\left[Y\right]_n\cap W_{\rho}\cap V_{\lambda},\mathcal{O}_{\fr{P}_K})$,
\[ \eta^{\norm{k}}\Norm{\partial^{[k]}(f)}_{\rho,n,\lambda} \rightarrow0 \]
as $\norm{k}\rightarrow\infty$. Now, the \'etale co-ordinates $x_1,\ldots,x_d \colon \fr{P}\rightarrow \widehat{\A}^d_{\mathcal{V}}$ induce an isomorphism
\[ \tube{P}_{\fr{P}^2} \isomto \fr{P}_K \times_K \D^d_K(0;1^-)  \]
between the tube of the diagonal and the $d$-dimensional open unit polydisc over $\fr{P}$, with co-ordinates on $\D^d_K(0;1^-)$ being given by $\chi_i=1\otimes x_i-x_i\otimes 1$. Then first projection $p_1$ induces the structure morphism
\begin{align*}
 R_{\Q} &\rightarrow \Gamma(\fr{P}_K \times_K \D_K^d(0;1^-),\mathcal{O}_{\fr{P}_K \times_K \D_K^d(0;1^-)}) = \mathrm{lim}_{\delta<1} R_{\Q}\tate{\delta^{-1}\chi_1,\ldots,\delta^{-1}\chi_d},
 \end{align*}
and the second projection $p_2$ induces the Taylor series morphism
\begin{align*}
 R_{\Q} &\rightarrow \mathrm{lim}_{\delta<1} R_{\Q}\tate{\delta^{-1}\chi_1,\ldots,\delta^{-1}\chi_d} \\
 f &\mapsto \sum_{k} (-1)^{\norm{k}}\partial^{[k]}(f)\chi^{k}. 
 \end{align*}
Hence we may translate what we need to prove as follows:
\begin{itemize}
\item[$(\star\star)$] there exist $\lambda_n\geq \lambda'_{n} $, such that for all $\eta$-admissible $W_\rho$, all $n\geq n_0$, and all $\lambda\geq \lambda_n$,
\[ p_2(\left[Y\right]_n\cap  W_{\rho}\cap V_\lambda \times_K \D^d_K(0;\eta)) \subset \left[Y\right]_n\cap  W_{\rho}\cap V_\lambda.\]
\end{itemize}
The fact that $W_\rho$ is $\eta$-admissible implies that
\[ p_2(W_\rho\times_K \D^d_K(0;\eta)) \subset W_\rho, \]
and it follows from \cite[Lemma 4.3.10]{LS07} that there exists $\lambda_n$ such that
\[  p_2(\left[Y\right]_n\cap V_\lambda \times \D^d_K(0;\eta_n)) \subset \left[Y\right]_n\cap  V_\lambda \]
for $\lambda\geq \lambda_n$. Since $\eta_n\geq 	\eta$ by our choice of $n_0$, we therefore conclude that
\begin{align*} p_2(\left[Y\right]_n\cap W_{\rho}\cap V_\lambda \times_K \D^d_K(0;\eta_{n}) ) &\subset p_2(W_\rho \times_K \D^d_K(0;\eta)) \cap p_2(\left[Y\right]_n \cap V_\lambda \times_K \D_K^d(0;\eta_n))  \\ &\subset \left[Y\right]_n\cap W_{\rho} \cap V_\lambda \end{align*}
as required.
\end{proof}

\subsection{Stability of pre-$\pmb{\mathscr{D}^\dagger}$-module structures} \label{sec: Ddagger stab}

The next key point will be to show how to transport pre-$\mathscr{D}^\dagger$-module structures along the functors $j_X^\dagger$ and $\mathcal{R}_{i_Z,\infty}$ used to construct the explicit $\mathrm{sp}_*$-acyclic resolutions $\mathrm{Tot}_{\mathrm{sd}(A)}\mathcal{R}_{\overline{P}_{\bullet},\infty}j^\dagger_{P_\bullet}\mathscr{F}$ of a constructible isocrystal $\mathscr{F}$ in \S\ref{sec: calc Rsp} above. 

\subsubsection{} We first consider the case of $j_X^\dagger$. Suppose that $Y\hookrightarrow P$ is a closed immersion, and that $X\rightarrow Y$ is a strongly affine open immersion. Let $\mathscr{F}$ be a $\mathscr{D}_{\tube{Y}_\mathfrak{P}}$-module equipped with a pre-$\mathscr{D}^\dagger$-module structure $\left\{\mathscr{F}_i\right\}_{i\in I}$. 

As in \S\ref{sec: fund} above, for $\lambda<1$ sufficiently close to $1$, we let $V_\lambda$ be the subspace of $\fr{P}_K$ locally defined by $v(f)\geq \lambda$, where $f\in\cO_\fr{P}$ is such that $X=Y\cap D(f)$. Then a cofinal system of neighbourhoods of $\tube{X}_\mathfrak{P}$ inside $\tube{Y}_\mathfrak{P}$ is given by the collection of open subsets
\[ V_{\underline{\lambda}}:=\bigcup_{n\geq 0} \left[Y\right]_n \cap V_{\lambda_n} \]
indexed by \emph{increasing} sequences $\underline{\lambda}:=\{\lambda_n\}_{n\geq 1}$ of real numbers $\lambda_n < 1$. 

Let $j_{\underline{\lambda}}:V_{\underline{\lambda}}\rightarrow \tube{Y}_\mathfrak{P}$ denote the given inclusion. Then
\[ j_X^\dagger \mathscr{F}=\mathrm{colim}_{\underline{\lambda}\in\mathcal{S}} j_{\underline{\lambda}*}j^{-1}_{\underline{\lambda}*}\mathscr{F} =  \mathrm{colim}_{i\in I}\mathrm{colim}_{\underline{\lambda}\in\mathcal{S}} j_{\underline{\lambda}*}j^{-1}_{\underline{\lambda}*}\mathscr{F}_i \]
and we can consider $\left\{j_{\underline{\lambda}*}j^{-1}_{\underline{\lambda}*}\mathscr{F}_i\right\}_{(\underline{\lambda},i)\in\mathcal{S}\times I}$ as an ind-$\mathscr{D}_{\tube{Y}_\mathfrak{P}}$-module. Moreover, if $W\subset \tube{Y}_\mathfrak{P}$ is affionid, then $W\cap V_{\underline{\lambda}}$ is a finite union of affinoids. Thus
\[ \Gamma(W, j_{\underline{\lambda}*}j^{-1}_{\underline{\lambda}*}\mathscr{F}_i)= \Gamma(W\cap V_{\underline{\lambda}}, \mathscr{F}_i)\]
is a finite inverse limit of Fr\'echet spaces, and thus a Fr\'echet space. Hence we can consider $\left\{j_{\underline{\lambda}*}j^{-1}_{\underline{\lambda}*}\mathscr{F}_i\right\}_{(\underline{\lambda},i)\in\mathcal{S}\times I}$ as an ind-object in the category of Fr\'echet $\mathscr{D}_{\tube{Y}_\mathfrak{P}}$-modules.

\begin{proposition} \label{prop: Ddag jdag}The ind-object $\left\{j_{\underline{\lambda}*}j^{-1}_{\underline{\lambda}*}\mathscr{F}_i\right\}_{(\underline{\lambda},i)\in\mathcal{S}\times I}$  together with the isomorphism
\[ \mathrm{colim}_{i\in I}\mathrm{colim}_{\underline{\lambda}\in\mathcal{S}} j_{\underline{\lambda}*}j^{-1}_{\underline{\lambda}*}\mathscr{F}_i  \isomto j_X^\dagger \mathscr{F} \]
defines a pre-$\mathscr{D}^\dagger$-module structure on $j_X^\dagger\mathscr{F}$.
\end{proposition}

\begin{proof}
We need to check that $\left\{j_{\underline{\lambda}*}j^{-1}_{\underline{\lambda}*}\mathscr{F}_i\right\}_{(\underline{\lambda},i)\in\mathcal{S}\times I}$ is a pre-$\mathscr{D}^\dagger$-module. 

Let $\eta<1$ be sufficiently close to $1$, and replace $\mathfrak{P}$ by its open affine subscheme $\mathfrak{P}'$. Choose a divisor $D\subset P$ such that $Y\setminus X=Y\cap D$. 

Since $\left\{\mathscr{F}_i\right\}_{i\in I}$ is a pre-$\mathscr{D}^\dagger$-module, we may choose $i_\eta$, $T_\eta\subset \tube{Y}_\mathfrak{P}$ such that for all affine open immersions $U\hookrightarrow P$, all $\eta$-admissible neighbourhoods $W$ of $\tube{U}_{\mathfrak{P}}$ inside $\mathfrak{P}_K$, all $i\geq i_\eta$, and all affinoids $T_\eta\subset T \subset \tube{Y}_{\mathfrak{P}} $, the $\Gamma(\mathfrak{P},\mathscr{D}_{\mathfrak{P}\Q})$-module structure on $\Gamma(T\cap W,\mathscr{F}_{i})$ extends to a continuous $\Gamma(\mathfrak{P},\mathscr{D}_{\mathfrak{P}\Q}^\eta)$-module structure.

By Lemma \ref{lemma: funny order}, we may choose $\underline{\lambda}_\eta$ such that for every $\underline{\lambda}\gtrsim \underline{\lambda}_\eta$, and every open affinoid $T_\eta\subset T\subset \tube{Y}_\mathfrak{P}$, the intersection $T\cap V_{\underline{\lambda}}$ has a finite cover by opens of the form $T' \cap V_\lambda$, where $T_\eta \subset T'\subset\tube{Y}_\mathfrak{P}$ is open affinoid, and $V_\lambda$ is an $\eta$-admissible open affinoid neighbourhood of $\tube{P\setminus D}_{\mathfrak{P}_K}$ in $\mathfrak{P}_K$.

We claim that for any affine open $U\hookrightarrow P$, any $\eta$-admissible neighbourhood $W$ of $\tube{U}_{\mathfrak{P}}$ inside $\mathfrak{P}$, any $i\geq i_\eta$, any $\underline{\lambda}\gtrsim \underline{\lambda}_\eta$, and any affinoid $T_\eta\subset T \subset \tube{Y}_{\mathfrak{P}}$,
\[ \Gamma(T\cap W,j_{\underline{\lambda}*}j^{-1}_{\underline{\lambda}*}\mathscr{F}_i) = \Gamma(T\cap W\cap V_{\underline{\lambda}},\mathscr{F}_i) \] 
is naturally a $\Gamma(\mathfrak{P},\mathscr{D}_{\mathfrak{P}\Q}^{\eta})$-module. To see this, we note that by the choice of $\underline{\lambda}_\eta$, $T\cap W\cap V_{\underline{\lambda}}$ is a finite union of opens of the form $T' \cap W\cap V_\lambda$, where $\lambda$ is close enough to $1$ that $V_\lambda$ is an $\eta$-admissible open neighbourhood of $\tube{P\setminus D}_{\mathfrak{P}_K}$ in $\mathfrak{P}_K$, and $T_\eta\subset T'\subset \tube{Y}_\mathfrak{P}$ is open affinoid. Thus by Remark \ref{rem: m adm inter}(\ref{rem: m adm inter num}) we see that $T\cap W\cap V_{\underline{\lambda}}$ is a finite union of opens of the form $T'\cap  W'$, where $W'$ is an $\eta$-admissible open neighbourhood of $\tube{U \setminus D}_{\mathfrak{P}_K}$ in $\mathfrak{P}_K$, and $T_\eta\subset T'\subset \tube{Y}_\mathfrak{P}$ is open affinoid. Hence $\Gamma(T\cap W,j_{\underline{\lambda}*}j^{-1}_{\underline{\lambda}*}\mathscr{F}_i)$ is a finite limit of $\Gamma(\mathfrak{P},\mathscr{D}_{\mathfrak{P}\Q}^{\eta})$-modules, thus a $\Gamma(\mathfrak{P},\mathscr{D}_{\mathfrak{P}\Q}^{\eta})$-module.
\end{proof}

\subsubsection{}

We next consider the case of $\mathcal{R}_{i_Z,\infty}$. So let $i_Z:Z\hookrightarrow Y$ be an inclusion of closed subschemes of $P$, and $\mathscr{F}$ a $\mathscr{D}_{\tube{Z}_\mathfrak{P}}$-module equipped with a pre-$\mathscr{D}^\dagger$-module structure $\left\{\mathscr{F}_i \right\}_{i\in I}$. Let $i_{Z,n}:\left[ Z\right]_n \rightarrow \tube{Y}_\mathfrak{P}$ denote the inclusion of the closed tube of radius $\eta_n$. Then we find
\begin{align*}\mathrm{colim}_{n_0} \prod_{n\geq n_0} i_{Z,n*}\left( \mathscr{F}|_{\left[ Z\right]_n}\right) &\cong \mathrm{colim}_{n_0} \prod_{n\geq n_0} i_{Z,n*}\left( \mathrm{colim}_i \mathscr{F}_i|_{\left[ Z\right]_n}\right) \\
&\cong \mathrm{colim}_{n_0} \prod_{n\geq n_0}\mathrm{colim}_i i_{Z,n*} \left( \mathscr{F}_i|_{\left[ Z\right]_n}\right) \\
&\cong \mathrm{colim}_{n_0}\mathrm{colim}_{\underline{i}\in I^{\N_{\geq n_0}}} \prod_{n\geq n_0} i_{Z,n*} \left( \mathscr{F}_{i_n}|_{\left[ Z\right]_n}\right)
\end{align*}
where $I^{\N_{\geq n_0}}$ denotes the set of sequences in $I$, starting at $i_{n_0}$. If $W\subset \tube{Y}_\mathfrak{P}$ is affinoid, then $W\cap \left[ Z\right]_n$ is also affinoid, so
\[ \Gamma(W,\prod_{n\geq n_0} i_{Z,n*} \mathscr{F}_{i_n}|_{\left[ Z\right]_n}) = \prod_{n\geq n_0} \Gamma(W\cap \left[ Z \right]_n,\mathscr{F}_{i_n}) \]  
is a countable product of Fr\'echet spaces, and therefore a Fr\'echet space. Hence we may consider $\left\{\prod_{n\geq n_0} i_{Z,n*} \left( \mathscr{F}_{i_n}|_{\left[ Z\right]_n}\right) \right\}_{n_0\geq 1, \underline{i}\in I^{\N_{\geq n_0}}}$ as an ind-object in the category of Fr\'echet $\mathscr{D}_{\tube{Y}_\mathfrak{P}}$-modules.

\begin{proposition} \label{prop: Ddag Roos} The ind-object $\left\{\prod_{n\geq n_0} i_{Z,n*} \left( \mathscr{F}_{i_n}|_{\left[ Z\right]_n}\right) \right\}_{n_0\geq 1, \underline{i}\in I^{\N_{\geq n_0}}}$ together with the isomorphism
\[ \mathrm{colim}_{n_0} \prod_{n\geq n_0} i_{Z,n*}\left( \mathscr{F}|_{\left[ Z\right]_n}\right) \cong \mathrm{colim}_{n_0}\mathrm{colim}_{\underline{i}\in I^{\N_{\geq n_0}}} \prod_{n\geq n_0} i_{Z,n*} \left( \mathscr{F}_{i_n}|_{\left[ Z\right]_n}\right) \]
defines a pre-$\mathscr{D}^\dagger$-module structure on $\mathrm{colim}_{n_0} \prod_{n\geq n_0} i_{Z,n*}\left( \mathscr{F}|_{\left[ Z\right]_n}\right)$.
\end{proposition}

\begin{proof}
We need to check that $\left\{\prod_{n\geq n_0} i_{Z,n*} \left( \mathscr{F}_{i_n}|_{\left[ Z\right]_n}\right) \right\}_{n_0\geq 1, \underline{i}\in I^{\N_{\geq n_0}}}$ is a pre-$\mathscr{D}^\dagger$-module. 

So let $\eta<1$, and replace $\mathfrak{P}$ by its open affine subscheme $\mathfrak{P}'$. Since $\left\{\mathscr{F}_i\right\}_{i\in I}$ is a pre-$\mathscr{D}^\dagger$-module, we may choose $i_\eta$, $T_\eta\subset \tube{Z}_\mathfrak{P}$ such that for all affine open immersions $U\hookrightarrow P$, all $\eta$-admissible affinoid neighbourhoods $W$ of $\tube{U}_{\mathfrak{P}}$ inside $\mathfrak{P}_K$, all $i\geq i_\eta$, and all affinoids $T_\eta\subset T \subset \tube{Z}_{\mathfrak{P}} $, the $\Gamma(\mathfrak{P},\mathscr{D}_{\mathfrak{P}\Q})$-module structure on $\Gamma(T\cap W,\mathscr{F}_{i})$ extends to a continuous $\Gamma(\mathfrak{P},\mathscr{D}_{\mathfrak{P}\Q}^{\eta})$-module structure.

Choose $n_\eta$ such that $n\geq n_\eta \Rightarrow \left[Z\right]_n \supset T_\eta$, and choose a sequence $\underline{i}$ so that  $i_n\geq i_\eta$ for every $n\geq n_\eta$. Then for any affine open immersions $U\hookrightarrow P$, any $\eta$-admissible affinoid neighbourhood $W$ of $\tube{U}_{\mathfrak{P}}$ inside $\mathfrak{P}_K$, any $n_0\geq n_\eta$, any $\underline{i}'\geq \underline{i}$, and any affinoid $T_\eta\subset T \subset \tube{Y}_{\mathfrak{P}}$, we have
\[ \Gamma(T\cap W ,\prod_{n\geq n_0} i_{Z,n*} \mathscr{F}_{i_n}|_{\left[ Z\right]_n}) = \prod_{n\geq n_0} \Gamma(\left[ Z \right]_n\cap T \cap W   ,\mathscr{F}_{i_n}) . \]
Since $\left[ Z\right]_n\cap T \subset \tube{Z}_\mathfrak{P}$ is an open affinoid containing $T_\eta$, and $i_n\geq i_\eta$, we know that each $\Gamma(\left[ Z \right]_n \cap T\cap W  ,\mathscr{F}_{i_n})$ is naturally a  $\Gamma(\mathfrak{P},\mathscr{D}_{\mathfrak{P}\Q}^{\eta})$-module. Thus so is their product.
\end{proof}

Thus we obtain a pre-$\mathscr{D}^\dagger$-module structure on the complex $\mathcal{R}_{i_Z,\infty}(\mathscr{F})$ (or more precisely, on each of the two terms in the complex, and on the differential between them). Starting with Theorem \ref{theo: pre Ddag over} and repeatedly applying Propositions \ref{prop: Ddag jdag} and \ref{prop: Ddag Roos}, we obtain the:

\begin{theorem} Let $\mathfrak{P}$ be a smooth formal $\mathcal{V}$-scheme, $\mathscr{F}$ a constructible $\mathcal{O}_{\mathfrak{P}_K}$-module with convergent (integrable) connection, and $\left\{P_{\alpha}\right\}_{\alpha\in A}$ a good stratification for $\mathscr{F}$. Then, for every chain $\{\alpha_0<\ldots < \alpha_r\}\subset A$, the complex
\[  \mathrm{Tot}\left( \mathcal{R}_{\overline{i}_{\alpha_0},\infty}j^\dagger_{P_{\alpha_0}}\mathcal{R}_{\overline{i}_{\alpha_1\alpha_0},\infty}j^\dagger_{P_{\alpha_1}}\ldots \mathcal{R}_{\overline{i}_{\alpha_r\alpha_{r-1}},\infty}j^\dagger_{P_{\alpha_r}}\overline{i}^{-1}_{\alpha_r}\mathscr{F}\right)\]
is naturally endowed with a pre-$\mathscr{D}^\dagger$-module structure.
\end{theorem}

Hence by Lemma \ref{lemma: pre Ddag sp} (and the explicit construction of \S\ref{sec: concrete holim}) we obtain the:

\begin{corollary} Let $\mathfrak{P}$ be a smooth formal $\mathcal{V}$-scheme, $\mathscr{F}$ a constructible $\mathcal{O}_{\mathfrak{P}_K}$-module with convergent (integrable) connection, and $\left\{P_{\alpha}\right\}_{\alpha\in A}$ a good stratification for $\mathscr{F}$. Then the complex
\[ \mathrm{sp}_*\mathrm{Tot}_{\mathrm{sd}(A)} \mathcal{R}_{\overline{P}_{\bullet},\infty}j^\dagger_{P_\bullet} \mathscr{F} \]
is naturally endowed with a $\mathscr{D}^\dagger_{\mathfrak{P}\Q}$-module structure.
\end{corollary}

Of course, this $\mathscr{D}^\dagger_{\mathfrak{P}\Q}$-module structure is compatible with refining the stratification, and the category of good stratifications is filtered, so we obtain the $\mathscr{D}^\dagger$-linear lifting of $\mathbf{R}\mathrm{sp}_*$ we are after.

\begin{corollary} \label{cor: dagger lifting of sp_*}Let $\mathfrak{P}$ be a smooth formal $\mathcal{V}$-scheme. There exists a canonical lifting
\[ \mathbf{R}\mathrm{sp}_*:\mathrm{Isoc}_{\mathrm{cons}}(\mathfrak{P}) \rightarrow {\bf D}^b(\mathscr{D}^\dagger_{\mathfrak{P}\Q}) \]
of the pushforward functor
\[ \mathbf{R}\mathrm{sp}_*\colon \mathrm{Isoc}_{\mathrm{cons}}(\mathfrak{P}) \rightarrow {\bf D}^b(\mathscr{D}_{\mathfrak{P}\Q}). \]
It is explicitly defined for $\mathscr{F}\in \mathrm{Isoc}_\mathrm{cons}(\mathfrak{P})$ by choosing a good stratification $\left\{P_{\alpha}\right\}_{\alpha\in A}$ for $\mathscr{F}$ and setting
\[ \mathbf{R}\mathrm{sp}_*\mathscr{F}:= \mathrm{sp}_*\mathrm{Tot}_{\mathrm{sd}(A)} \mathcal{R}_{\overline{P}_{\bullet},\infty}j^\dagger_{P_\bullet} \mathscr{F}.\]
\end{corollary}

In order to obtain a functor with the correct essential image (see \S\ref{sec: finite} below), we define
\[ \mathrm{sp}_!:=\mathbf{R}\mathrm{sp}_*[\dim \mathfrak{P}]\colon\mathrm{Isoc}_{\mathrm{cons}}(\mathfrak{P}) \rightarrow  {\bf D}^b(\mathscr{D}^\dagger_{\mathfrak{P}\Q}) \]
to be the functor $\mathbf{R}\mathrm{sp}_*$ shifted by the relative dimension of $\mathfrak{P}$ over $\mathcal{V}$. We will sometimes write $\mathrm{sp}_{\fr{P}!}$ if we want to emphasise that we are working on the formal scheme $\mathfrak{P}$. In general, we will have to work quite hard to show compatibility of $\mathrm{sp}_!$ with other cohomological functors, but there is at least one instance where things are relatively straightforward.

\begin{proposition} \label{prop: comm finite etale} Let $u:\fr{P}'\rightarrow \fr{P}$ be a finite \'etale morphism of smooth formal schemes. Then the functor $u_{K*}$ from $\mathscr{D}_{\fr{P}'_K}$-modules to $\mathscr{D}_{\fr{P}_K}$-modules preserves constructible isocrystals, and the diagram
\[ \xymatrix{ \mathrm{Isoc}_{\mathrm{cons}}(\mathfrak{P}') \ar[r]^{\mathrm{sp}_{\fr{P}'!}}\ar[d]_{u_{K*}} & {\bf D}^b(\mathscr{D}^\dagger_{\mathfrak{P}'\Q}) \ar[d]^{u_+} \\  
\mathrm{Isoc}_{\mathrm{cons}}(\mathfrak{P}) \ar[r]^{\mathrm{sp}_{\fr{P}!}}& {\bf D}^b(\mathscr{D}^\dagger_{\mathfrak{P}\Q})
 }  \]
commutes up to natural isomorphism.
\end{proposition}

\begin{proof}
That $u_{K*}$ preserves overconvergence follows from the fact that the induced map
\[ u: \tube{P'}_{\fr{P}'^2} \rightarrow \tube{P}_{\fr{P}^2} \]
is finite \'etale, and both diagrams
\[ \xymatrix{ \tube{P'}_{\fr{P}'^2} \ar[r] \ar[d]_{p_i} & \tube{P}_{\fr{P}^2} \ar[d]^{p_i} \\ \fr{P}'\ar[r] & \fr{P} }\]
for $i=1,2$ are Cartesian.

To see that $u_{K*}$ preserves constructibility, we may replace $u$ by its Galois closure, and thus assume that $u$ is Galois, with Galois group $G$. If we let $\mathscr{F}\in \mathrm{Isoc}_{\mathrm{cons}}(\mathfrak{P}')$ and pick with a good stratification of $P'$ with respect to $\mathscr{F}$, then we may always choose a refinement which is $G$-invariant. We can then see that this image of this $G$-invariant refinement will be a good stratification for $P$ with respect to $u_{K*}\mathscr{F}$, and so $u_{K*}$ preserves constructibility.

Finally, to obtain commutativity of the diagram, we can simply compute $\mathrm{sp}_{\fr{P}'!}\mathscr{F}$ and $\mathrm{sp}_{\fr{P}!}u_{K*}\mathscr{F}$ with respect to the such a $G$-equivariant stratifications.
\end{proof}

\subsection{On the definition of $\bm{\mathrm{sp}_!}$}

One of the reasons that the above definition of $\mathrm{sp}_!$ is so involved is that we have been working with general constructible isocrystals $\mathscr{F}\in\mathrm{Isoc}_\mathrm{cons}(\fr{P})$, rather than just overconvergent isocrystals supported on some locally closed subscheme of $\fr{P}$. While we are mostly interested in this latter case, constructing $\mathrm{sp}_!$ at this greater level of generality will significantly simplify certain d\'evissage arguments that will be used later on to establish the key properties of $\mathrm{sp}_!$. For overconvergent isocrystals, it is, however, possible to give a somewhat simplified definition of $\mathrm{sp}_!$, that we briefly outline here. 

Suppose that we have a frame $(X\overset{j}{\hookrightarrow}Y\overset{i}{\hookrightarrow}\fr{P})$ and a partially overconvergent isocrystal $E\in \mathrm{Isoc}(X,Y/K)$, realised as a locally free $\mathcal{O}_{\tube{X}_\fr{P}}$-module with overconvergent connection $E_\fr{P}$. Let $\mathscr{F}=j_*E_\fr{P}$ be the associated $j_X^\dagger\cO_{\tube{Y}_\fr{P}}$-module. 

First, let us assume that the open immersions $X\hookrightarrow Y$ and $P\setminus Y \hookrightarrow P$ are strongly affine. Then the simple complex associated to the double complex
\[ \mathcal{R}_{Y,\infty}\mathscr{F} \rightarrow j_{P\setminus Y}^\dagger\mathcal{R}_{Y,\infty}\mathscr{F}.\]
has a natural pre-$\mathscr{D}^\dagger$-module structure, and we can construct a quasi-isomorphism
\[ \mathrm{sp}_!i_{Y!}\mathscr{F} \cong \mathrm{sp}_*\mathrm{Tot}\left(\mathcal{R}_{Y,\infty}\mathscr{F} \rightarrow j_{P\setminus Y}^\dagger\mathcal{R}_{Y,\infty}\mathscr{F} \right)[\dim \fr{P}]\]
of complexes of $\mathscr{D}^\dagger_{\fr{P}\Q}$-modules.

If either $X\hookrightarrow Y$ or $P\setminus Y\hookrightarrow P$ are not strongly affine, we proceed by taking suitable affine covers and using Berthelot's resolution \cite[Proposition 2.1.8]{Ber96b}. We will not use this alternative construction of $\mathrm{sp}_!$ anywhere in this article.

\section{The trace map in rigid geometry} \label{sec: trace}

We have defined, for any smooth formal scheme $\fr{P}$ a canonical functor
\[ \mathrm{sp}_!\colon \mathrm{Isoc}_\mathrm{cons}(\mathfrak{P}) \rightarrow {\bf D}^b(\mathscr{D}^\dagger_{\mathfrak{P}\Q})\]
which, up to a shift, lifts the natural functor $\mathbf{R}\mathrm{sp}_*: \mathrm{Isoc}_\mathrm{cons}(\mathfrak{P}) \rightarrow {\bf D}^b(\mathscr{D}_{\mathfrak{P}\Q})$. Our next goal is to show the overholonomicity of objects in the essential image of $\mathrm{sp}_!$. The strategy will be to reduce to the fundamental result of Caro--Tsuzuki \cite{CT12}, and the key ingredient that we will need to implement this strategy will be a $\mathscr{D}^\dagger$-linear lifting of the trace morphism in rigid analytic geometry (see Theorem \ref{theo: tr conc fin et} below).

The rigid analytic trace morphism can be described most naturally using the theory of higher direct images with compact support for sheaves on (germs of) adic spaces, the basic properties of which are developed in \cite{AL20}.

\subsection{Proper pushforwards} \label{subsec: f! prop}

We recall here the key results of \cite{AL20}. For any partially proper morphism $f\colon\mathscr{X}\rightarrow \mathscr{Y}$ in $\mathbf{Germ}_K$ (in the sense of \cite[Definition 1.10.15]{Hub96}) there is a functor
\[ \mathbf{R}f_!\colon {\bf D}^+(\mathscr{X}) \rightarrow  {\bf D}^+(\mathscr{Y}), \] 
having the following properties:
\begin{enumerate}
\item $f_!:=\mathcal{H}^0(\mathbf{R}f_!)$ is the functor of sections with proper support;
\item $\mathbf{R}f_!$ is the total derived functor of $f_!$;
\item there is a natural isomorphism $\mathbf{R}g_!\circ \mathbf{R}f_! \isomto \mathbf{R}(g\circ f)_!$ whenever $f,g$ are composable, partially proper morphisms in $\mathbf{Germ}_K$; 
\item $\mathbf{R}f_!=\mathbf{R}f_*$ whenever $f$ is proper;
\item $\mathbf{R}f_!= f_!$ is the usual extension by zero functor whenever $f$ is a (partially proper) locally closed immersion. 
\setcounter{saveenum}{\value{enumi}}
\end{enumerate}
If $f:\mathscr{X}\rightarrow \spa{K,\mathcal{V}}$ is the structure map of an object of $\mathbf{Germ}_K$, we will write $\mathbf{R}\Gamma_c(\mathscr{X},-)$ for $\mathbf{R}f_!$, and ${\rm H}^i_c(\mathscr{X},-)$ for its cohomology groups.

Now suppose that $f\colon\mathscr{X}\rightarrow \mathscr{Y}$ is smooth of relative dimension $d$, partially proper in the sense of Kiehl (see \cite[Definition 4.1.4]{AL20}, in general this is a stronger condition than being partially proper, although it is equivalent for morphisms of analytic varieties over $K$). If $\mathscr{Y}$ is overconvergent, we constructed in \cite[\S5]{AL20} a trace map
\[ \mathrm{Tr}\colon \mathbf{R}^{2d}f_! \Omega^\bullet_{\mathscr{X}/\mathscr{Y}} \rightarrow \mathcal{O}_{\mathscr{Y}}\]
such that:
\begin{enumerate} \setcounter{enumi}{\value{saveenum}} 
\item \label{num: residue}when $Y=\spa{R,R^+}$ is affinoid, and $X=\D^d_Y(0;1^-)$ is the relative open unit disc over $Y$, $\Tr$ is induced by the residue map
\begin{align*} {\rm H}^d_c(X/Y,\omega_{X/Y}) \isomto R\weak{x_1^{-1},\ldots,x_d^{-1}} \;d\log x_1\wedge \ldots \wedge d\log x_d &\rightarrow R \\ \sum_{i_1,\ldots,i_d\geq 0} r_{i_1,\ldots,i_d} x_1^{-i_1}\ldots x_d^{-i_d} \;d\log x_1\wedge \ldots \wedge d\log x_d&\mapsto r_{0,\ldots,0};  \end{align*}
\item \label{num: iso} whenever $\mathscr{X}$ is locally either a $\D^d(0;1^-)$-bundle or an $\A^{d,\mathrm{an}}$-bundle over $\mathscr{Y}$, $\Tr$ is an isomorphism.
\end{enumerate} 
Moreover, we showed that $\mathbf{R}^qf_!\Omega^p_{\mathscr{X}/\mathscr{Y}}=0$ if $q>d$, and hence we can view the trace map as a morphism
\[ \mathrm{Tr}\colon \mathbf{R}f_! \Omega^\bullet_{\mathscr{X}/\mathscr{Y}}[2d] \rightarrow \mathcal{O}_{\mathscr{Y}}. \]

In general, the proper base change theorem for $\mathbf{R}f_!$ fails, but we do have the following partial result. 

\begin{proposition}[\cite{AL20}, Corollary 3.5.1, Lemma 3.5.2] \label{prop: prop base changes} Let
\[ \xymatrix{ T_\mathscr{X} \ar[r]^{g'} \ar[d]_{f'} & \mathscr{X} \ar[d]^f \\ T \ar[r]^g & \mathscr{Y} } \]
by a Cartesian diagram in $\mathbf{Germ}_K$, such that $f$ is partially proper, and $g$ is one of the following:
\begin{enumerate}
\item a locally closed immersion onto a subspace which is closed under generalisation;
\item the inclusion of a maximal point of $\mathscr{Y}$.
\end{enumerate}
Then, for any $\mathscr{F}\in {\bf D}^+(\mathscr{X})$, the base change map
\[ g^{-1}\mathbf{R}f_!\mathscr{F} \rightarrow \mathbf{R}f'_! g'^{-1}\mathscr{F} \]
is an isomorphism.  
\end{proposition}

The functor $\mathbf{R}f_!$ preserves module structures, and we also have the following version of the projection formula for compactly supported cohomology. 

\begin{lemma}[\cite{AL20}, Corollary 3.8.2] \label{lemma: proj form} Let $f:\mathscr{X}\rightarrow \mathscr{Y}$ be a partially proper morphism in $\mathbf{Germ}_K$. For any locally free $\mathcal{O}_{\mathscr{Y}}$-module $\mathscr{E}$ of finite rank, and any complex $\mathscr{F}$ of $\mathcal{O}_{\mathscr{X}}$-modules, there is a natural isomorphism
\[ \mathscr{E}\otimes_{\mathcal{O}_\mathscr{Y}}  \mathbf{R}f_! \mathscr{F} \cong \mathbf{R}f_!(f^*\mathscr{E}\otimes_{\mathcal{O}_\mathscr{X}} \mathscr{F}).  \]
\end{lemma} 

If $f$ is moreover smooth of relative dimension $d$, and partially proper in the sense of Kiehl, we can use the projection formula to obtain a trace map
\[  \mathrm{Tr}_{\mathscr{F}}\colon \mathbf{R}^{2d}f_{!}\left(  \Omega^\bullet_{\mathscr{X}/\mathscr{Y}} \otimes_{\mathcal{O}_{\mathscr{X}}} f^*\mathscr{F}\right) \isomto   \mathbf{R}^{2d}f_{!}\Omega^\bullet_{\mathscr{X}/\mathscr{Y}} \otimes_{\mathcal{O}_{\mathscr{Y}}}\mathscr{F}\overset{\mathrm{Tr}\otimes \mathrm{id}}{\longrightarrow} \mathscr{F} \]
for any finite locally free $\mathcal{O}_{\mathscr{Y}}$-module $\mathscr{F}$. This can be viewed as a map
\[ \mathbf{R}f_{!}\left(  \Omega^\bullet_{\mathscr{X}/\mathscr{Y}} \otimes_{\mathcal{O}_{\mathscr{X}}} f^*\mathscr{F}\right)[2d] \rightarrow \mathscr{F}, \]
and whenever $\mathscr{F}$ has the structure of a $\mathscr{D}_{\mathscr{Y}}$-module, this map is $\mathscr{D}_{\mathscr{Y}}$-linear.

Finally, for any sheaf of rings $\mathcal{A}$ on $\mathscr{X}$, and any bilinear pairing
\[ \mathscr{E} \times \mathscr{F} \rightarrow \mathscr{G}  \]
of bounded complexes of $\mathcal{A}$-modules, there is an induced pairing
\[ \mathbf{R}f_*\mathscr{E}\times \mathbf{R}f_!\mathscr{F}\rightarrow \mathbf{R}f_!\mathscr{G}\]
in cohomology. See \cite[\S5.6]{AL20} for the case $\mathcal{A}=\mathcal{O}_\mathscr{X}$, the general case works in exactly the same way.

\subsection{The strong fibration theorem and compactly supported de\thinspace Rham cohomology} \label{sec: sft and csdr}

An important component of our eventual comparison of compactly supported rigid and $\mathscr{D}^\dagger$-module cohomology will be the computation of the compactly supported de\thinspace Rham cohomology of overconvergent isocrystals. That is, if $(X,Y,\mathfrak{P})$ is a frame with $Y$ proper over $k$ and $\mathfrak{P}$ smooth over $\mathcal{V}$ in a neighbourhood of $X$, and $\mathscr{F}$ is a coherent $\mathcal{O}_{\tube{X}_\mathfrak{P}}$-module with overconvergent connection, we are interested in calculating $\mathbf{R}\Gamma_c(]X[_\mathfrak{P},\Omega^\bullet_{\tube{X}_\mathfrak{P}}\otimes \mathscr{F})$. The reader should be warned that despite the terminology, compactly supported rigid cohomology of an overconvergent isocrystal is \emph{not} the same thing as compactly supported de\thinspace Rham cohomology of a given realisation. In fact, as we shall see, these two cohomologies are instead dual to one another.

In this section, we will lay the groundwork for this calculation by showing that, up to a shift, these compactly supported de\thinspace Rham cohomology groups only depend on $X$, and not on the choice of frame $(X,Y,\mathfrak{P})$ enclosing $X$. Let us suppose then, that we have a morphism of frames
\[\xymatrix{  X'  \ar[r]^{j'} \ar[d]^f & Y' \ar[r]^{i'} \ar[d]^g & \fr{P}' \ar[d]^u \\   X \ar[r]^{j} & Y \ar[r]^{i}   & \fr{P} } 
\]
over $\mathcal{V}$, such that $g$ is proper, and $u$ is smooth in a neighbourhood of $X$.

\begin{lemma} The induced morphism $\left]f\right[\colon \left]X'\right[_{\fr{P}'} \rightarrow \left]X\right[_\fr{P}$ of germs is smooth, and partially proper in the sense of Kiehl.
\end{lemma}

\begin{proof}
Recall from \cite[\S1.10]{Hub96} that a morphism of germs $T\rightarrow T'$ is smooth if it admits a representative $a:(T,\bm{T})\rightarrow (T',\bm{T}')$ such that $a:\bm{T}\rightarrow \bm{T}'$ is smooth and $T=\alpha^{-1}(T')$. Since the left hand square is Cartesian, the induced diagram
\[ \xymatrix{  \tube{X'}_{\fr{P}'} \ar[r] \ar[d] & \tube{Y'}_{\fr{P}'} \ar[d] \\ \tube{X}_\fr{P} \ar[r] & \tube{Y}_\fr{P} }\]
of germs is also Cartesian. To show smoothness, it therefore suffices to observe that $\tube{X'}_{\fr{P}'}$ admits an open neighbourhood which is smooth over $\fr{P}_K$.

Again, by the fact that the above square is Cartesian, if we want to show that $\tube{X'}_{\fr{P}'}\rightarrow \tube{X}_{\fr{P}}$ is partially proper in the sense of Kiehl, it suffices to show that $\tube{Y'}_{\fr{P}'}\rightarrow \tube{Y}_{\fr{P}}$ is. This follows from properness of $g$.
\end{proof}

We therefore obtain a canonical trace map
\[ \Tr\colon\mathbf{R}\tube{f}_{!}\Omega^\bullet_{\tube{X'}_{\fr{P}'}/\tube{X}_\fr{P}} [2d] \rightarrow \mathcal{O}_{\tube{X}_\fr{P}} \]
where $d$ is the relative dimension of $u$. As in \cite{LS11}, where it is phrased in terms of overconvergent varieties, Berthelot's strong fibration theorem \cite{Ber96b} has a particularly nice interpretation in terms of germs.

\begin{theorem}[Berthelot] \label{theo: strong}
Suppose that $f$ is an isomorphism.
\begin{enumerate}
\item If $g$ is also an isomorphism, then, locally on $X$ and on $\fr{P}$, there exists an isomorphism
\[ \tube{X}_{\fr{P}'} \isomto \tube{X}_\fr{P} \times_K \D^d_K(0;1^-) \]
of germs, identifying $\tube{f}$ with the first projection. 
\item If $u$ is \'etale in a neighbourhood of $X$, then
\[ \tube{f}\colon\tube{X}_{\fr{P}'} \isomto \tube{X}_\fr{P} \]
is an isomorphism of germs.
\end{enumerate}
\end{theorem}

\begin{corollary} Assume that $f$ is an isomorphism. Then the trace map 
\[  \Tr\colon\mathbf{R}\tube{f}_{!}\Omega^\bullet_{\tube{X'}_{\fr{P}'}/\tube{X}_\fr{P}}[2d] \rightarrow \mathcal{O}_{\tube{X}_\fr{P}} \]
is an isomorphism in $D^b(\mathcal{O}_{\tube{X}_\fr{P}})$.
\end{corollary}

\begin{proof}
This follows from Theorem \ref{theo: strong} in exactly the same way as the proof that
\[ \mathcal{O}_{\tube{X}_\fr{P}}\rightarrow \mathbf{R}\tube{f}_{*}\Omega^\bullet_{\tube{X'}_{\fr{P}'}/\tube{X}_\fr{P}}  \]
is an isomorphism, see for example \cite[\S6.5]{LS07}.
\end{proof}

If we now take a partially overconvergent isocrystal $E\in\mathrm{Isoc}(X,Y)$, with realisation $E_\fr{P}$ on $\tube{X}_\fr{P}$, and consider its pullback $f^*E\in \mathrm{Isoc}(X',Y')$, with realisation $(f^*E)_{\fr{P}'}$ on $\tube{X'}_{\fr{P}'}$ then, via the projection formula, we obtain a trace map 
\begin{align*} \Tr_E \colon \mathbf{R}\tube{f}_! \left(  \Omega^\bullet_{\tube{X'}_{\fr{P}'}/\tube{X}_\fr{P}}\otimes_{\mathcal{O}_{\tube{X'}_{\fr{P}'}}}(f^*E)_{\fr{P}'}\right)[2d] 
&\isomto   \mathbf{R}\tube{f}_{!}( \Omega^\bullet_{\tube{X'}_{\fr{P}'}/\tube{X}_\fr{P}})\otimes_{\mathcal{O}_{\tube{X}_\fr{P}}} E_\fr{P}[2d] \\
& \rightarrow E_{\fr{P}}.
 \end{align*}
Once more, this map will be an isomorphism if $f$ is.

\subsection{de\thinspace Rham cohomology of $\bm{i_{Y!}E_\fr{P}}$ and Borel--Moore homology} 

A variety $X$ over $k$ is said to be \emph{weakly realisable}\footnote{We already introduced a stronger notion of realisability earlier on.} if there exists a frame $(X\overset{j}{\hookrightarrow}Y\overset{i}{\hookrightarrow}\fr{P})$ with $Y$ proper, and $\fr{P}$ smooth in a neighbourhood of $X$. (This is the more usual definition of realisable in the theory of rigid cohomology.) Let $d$ denote the relative dimension of $\mathfrak{P}$ around $X$, and take an overconvergent isocrystal $E\in \mathrm{Isoc}^\dagger(X)$ (of Frobenius type), with realisation $E_\fr{P}$ on $\tube{X}_\fr{P}$. Let $\mathscr{F}=j_*E_\fr{P}$ be the associated $j_X^\dagger\cO_{\tube{Y}_\fr{P}}$-module. Then using the results of \S\ref{sec: sft and csdr} above, we can argue exactly as in \cite[\S\S7.4, 8.2]{LS07} to show that the compactly supported de\thinspace Rham cohomology
\[ \mathbf{R}\Gamma_c(\tube{Y}_\fr{P}, \Omega^\bullet_{\tube{Y}_\fr{P}}\otimes_{\mathcal{O}_{\tube{Y}_\fr{P}}} \mathscr{F} )[2d] = \mathbf{R}\Gamma_c(\tube{X}_\fr{P}, \Omega^\bullet_{\tube{X}_\fr{P}}\otimes_{\mathcal{O}_{\tube{X}_\fr{P}}} \mathscr{F} )[2d] \]
only depends on $X$ and $E$, and not on the choice 
of frame $(X,Y,\fr{P})$ enclosing $X$. 

\begin{definition} \label{defn: BM} We define the rigid Borel--Moore homology complex
\[ \mathbf{R}\Gamma_{\flat,\rig}(X,E) := \mathbf{R}\Gamma_c(\tube{X}_\fr{P}, \Omega^\bullet_{\tube{X}_\fr{P}}\otimes_{\mathcal{O}_{\tube{X}_\fr{P}}} E_\fr{P} )[2d]=\mathbf{R}\Gamma_c(\tube{Y}_\fr{P}, \Omega^\bullet_{\tube{Y}_\fr{P}}\otimes_{\mathcal{O}_{\tube{Y}_\fr{P}}} \mathscr{F} )[2d] , \] 
and the rigid Borel--Moore homology groups
\[  {\rm H}^\mathrm{BM}_{q,\rig}(X,E) :=  {\rm H}^{-q}\left( \mathbf{R}\Gamma_{\flat,\rig}(X,E)\right). = \]
\end{definition}

Note that when $\fr{P}$ is proper, we have
\[  \mathbf{R}\Gamma_{\flat,\rig}(X,E) = \mathbf{R}\Gamma(\fr{P}_K,  \Omega^\bullet_{\fr{P}_K}\otimes_{\mathcal{O}_{\fr{P}_K}}i_{!}\mathscr{F})[2d]. \]

\begin{remark} As usual, extending this to non-realisable varieties requires descent and the use of simplicial embeddings.
\end{remark}

Rigid Borel--Moore homology is covariantly functorial in proper morphisms of $k$-varieties. Indeed, any such morphism can be extended to a morphism
\[ \xymatrix{ X' \ar[d]_f \ar[r] & Y' \ar[d]  \ar[r] & \fr{P}' \ar[d] \\ 
X \ar[r] & Y \ar[r] & \fr{P}  } \] 
of frames, with $Y'\rightarrow Y$ proper, the left hand square Cartesian, and $\mathfrak{P}'\rightarrow \mathfrak{P}$ smooth around $X'$, of relative dimension $d$. Then the (horizontal) trace map
\[ \mathbf{R}\tube{f}_! \left( \Omega^\bullet_{\tube{X'}_{\fr{P}'}/\tube{X}_\fr{P}}\otimes_{\mathcal{O}_{\tube{X'}_{\fr{P}'}}} \tube{f}^*E_{\fr{P}} \right)[2d] \rightarrow E_{\fr{P}} \]
induces, upon taking compactly supported de\thinspace Rham cohomology, a morphism
\[ \mathbf{R}\Gamma_{\flat,\rig}(X', f^*E) \rightarrow \mathbf{R}\Gamma_{\flat,\rig}(X,E). \]
Rigid Borel--Moore homology is also contravariantly functorial in open immersions, and for a complementary pair
\[  U \rightarrow X \leftarrow Z \]
of open and closed subschemes of $X$, there is an excision exact triangle
\[ \mathbf{R}\Gamma_{\flat,\rig}(Z, E|_Z) \rightarrow \mathbf{R}\Gamma_{\flat,\rig}(X, E) \rightarrow \mathbf{R}\Gamma_{\flat,\rig}(U, E|_U) \overset{+1}{\rightarrow}  \]
obtained by simply applying $\mathbf{R}\Gamma_c(\tube{Y}_\fr{P},\Omega^\bullet_{\tube{Y}_\fr{P}}\otimes_{\mathcal{O}_{\tube{Y}_\fr{P}}}(-))$ to the localisation exact triangle
\[  \mathbf{R}\underline{\Gamma}^\dagger_{Y\setminus U} \mathscr{F}\rightarrow \mathscr{F} \rightarrow j_U^\dagger \mathscr{F} \overset{+1}{\rightarrow}   \]
on $\tube{Y}_\fr{P}$.

\subsection{Duality}

In order to justify calling the groups ${\rm H}^\mathrm{BM}_{n,\rig}(X,E)$ rigid Borel--Moore homology, we shall show that they are canonically dual to the compactly supported rigid cohomology groups ${\rm H}^n_{c,\rig}(X,E^\vee)$. 

\subsubsection{} First of all, let us construct the duality pairing. Let $(X,Y,\fr{P})$ be a frame with $Y$ proper and $\fr{P}$ smooth around $X$, of relative dimension $d$. Let $E\in \isoc{X}$ be an overconvergent isocrystal (of Frobenius type), with realisation $E_\fr{P}$ on $\tube{X}_\fr{P}$. Set $\mathscr{F}=j_*E_\fr{P}$. Choose an open neighbourhood $i_{XV}:\tube{X}_\fr{P}\rightarrow V$ of $\tube{X}_\fr{P}$ in $\tube{Y}_\fr{P}$, and a module with connection $\mathscr{E}_V$ on $V$ such that $E_\fr{P}=i_{XV}^{-1}\mathscr{E}_V$. Let 
\[ \mathbf{R}i_{XV}^!:\mathbf{Sh}(V)\rightarrow \mathbf{Sh}(\tube{X}_\fr{P}) \]
denote the extraordinary pullback functor. In other words, if $a_{V}:V\setminus \tube{X}_\fr{P} \rightarrow V$ denotew the complement, then $\mathbf{R}j_{V}^!$ is characterised by the existence of a distinguished triangle
\[ j_{V*}\mathbf{R}j_{V}^! \rightarrow \mathrm{id} \rightarrow \mathbf{R}a_{V*}a_{V}^{-1}\overset{+1}{\rightarrow} \]
Then the complex
\[ \mathbf{R}\underline{\Gamma}_{\tube{X}_\fr{P}}\mathscr{F}:=j_*\mathbf{R}j_{V}^! \mathscr{E}_V \]
on $\tube{Y}_\fr{P}$ is easily checked not to depend on the choice of $V$. Note that the use of $\mathbf{R}\underline{\Gamma}_{\tube{X}_\fr{P}}$ is conventional in the literature, although slightly confusing in our situation. The complex $\mathbf{R}\underline{\Gamma}_{\tube{X}_\fr{P}}\mathscr{F}$ is \emph{not} the the derived functor of sections with support on $\tube{X}_\fr{P}$, applied to $\mathscr{F}$. Instead, it is the derived functor of sections with support on $\tube{X}_\fr{P}$, applied to $\mathscr{E}_V$, considered as a complex on $\tube{Y}_\fr{P}$.

We similarly denote
\[  \mathbf{R}\underline{\Gamma}_{\tube{X}_\fr{P}}\left(\Omega^\bullet_{\tube{Y}_\fr{P}}\otimes_{\mathcal{O}_{\tube{Y}_\fr{P}}}\mathscr{F}\right):=j_*\mathbf{R}j_{V}^! \left( \Omega^\bullet_V\otimes_{\mathcal{O}_V}\mathscr{E}_V \right), \]
and also set
\begin{align*} \mathbf{R}\underline\Gamma_{\tube{X}_\fr{P}}\mathcal{O}_{\tube{Y}_\fr{P}}&:=j_*\mathbf{R}j^!\mathcal{O}_{\tube{Y}_\fr{P}} \\
\mathbf{R}\underline\Gamma_{\tube{X}_\fr{P}}\Omega^\bullet_{\tube{Y}_\fr{P}}&:=j_*\mathbf{R}j^!\Omega^\bullet_{\tube{Y}_\fr{P}}.
\end{align*}
In this case $\mathbf{R}\underline\Gamma_{\tube{X}_\fr{P}}\Omega^\bullet_{\tube{Y}_\fr{P}}$ really is the derived functor of sections with support on $\tube{X}_\fr{P}$, applied to $\Omega^\bullet_{\tube{Y}_\fr{P}}$.

If $E^\vee$ denotes the dual isocrystal to $E$, and $\mathscr{F}^\vee=j_*E^\vee_\fr{P}$, then there is a horizontal pairing
\[ \mathscr{F} \times \mathbf{R}\underline{\Gamma}_{\tube{X}_\fr{P}}\mathscr{F}^\vee \rightarrow \mathbf{R}\underline{\Gamma}_{\tube{X}_\fr{P}}\mathcal{O}_{\tube{Y}_\fr{P}} \]
constructed in \cite{Ber97c}, which induces a pairing 
\[  \Omega^\bullet_{\tube{Y}_\fr{P}}  \otimes \mathscr{F} \times \mathbf{R}\underline{\Gamma}_{\tube{X}_\fr{P}}(  \Omega^\bullet_{\tube{Y}_\fr{P}} \otimes \mathscr{F}^\vee)  \rightarrow \mathbf{R}\underline{\Gamma}_{\tube{X}_\fr{P}}\Omega^\bullet_{\tube{Y}_\fr{P}}  \]
on de\thinspace Rham complexes. As discussed in \S\ref{subsec: f! prop} above, we thus obtain a pairing
\[ \mathbf{R}\Gamma_c(\tube{Y}_\fr{P}, \Omega^\bullet_{\tube{Y}_\fr{P}}  \otimes \mathscr{F}) \times \mathbf{R}\Gamma(\tube{Y}_\fr{P},\mathbf{R}\underline{\Gamma}_{\tube{X}_\fr{P}}(  \Omega^\bullet_{\tube{Y}_\fr{P}} \otimes \mathscr{F}^\vee))\rightarrow \mathbf{R}\Gamma_c(\tube{Y}_\fr{P},\mathbf{R}\underline{\Gamma}_{\tube{X}_\fr{P}}\Omega^\bullet_{\tube{Y}_\fr{P}} )\]
of complexes of $K$-vector spaces. If we choose an open neighbourhood $V'$ of $\tube{X}_\fr{P}$ which is smooth and partially proper over $K$, then the composition of the `forget supports' map and the trace map for $V'$ induces
\begin{align*}
 \mathbf{R}\Gamma_c(\tube{Y}_\fr{P},\mathbf{R}\underline{\Gamma}_{\tube{X}_\fr{P}}\Omega^\bullet_{\tube{Y}_\fr{P}}) &\isomto \mathbf{R}\Gamma_c(V',j_{V'*}\mathbf{R}j_{V'}^!\Omega^\bullet_{V'}) \\ 
 &\rightarrow \mathbf{R}\Gamma_c(V',\Omega^\bullet_{V'}) \\
  &\rightarrow K[-2d].
\end{align*}
We therefore obtain, for any $q\geq0$, a pairing
\begin{equation}
{\rm H}^\mathrm{BM}_{q,\rig}(X,E) \times {\rm H}^q_{c,\rig}(X,E^\vee) \rightarrow K \label{eqn: pairing}
\end{equation} 
where
\[ {\rm H}^q_{c,\rig}(X,E^\vee) := {\rm H}^q(\tube{Y}_\fr{P},\mathbf{R}\underline{\Gamma}_{\tube{X}_\fr{P}}(  \Omega^\bullet_{\tube{Y}_\fr{P}}\otimes \mathscr{F}^\vee))\]
denotes the compactly supported rigid cohomology with coefficients in $E^\vee$, as defined in, for example, \cite[\S\S6.4, 7.4, 8.2]{LS07}.

\begin{proposition} \label{prop: BM} The duality pairing (\ref{eqn: pairing}) is perfect.
\end{proposition}

\begin{proof}
Both ${\rm H}_{q,\rig}^{\mathrm{BM}}(X,E)$ and ${\rm H}^q_{c,\rig}(X,E^\vee)$ sit in excision exact sequences, and the duality pairing is compatible with these sequences. Hence, by Noetherian induction, we can reduce to the case when $X$ is smooth, affine and irreducible. But in this case we can calculate everything in sight using a suitable `Monsky--Washnitzer' frame $(X,Y,\fr{P})$. That is, one where $\fr{P}$ is projective, and smooth over $\mathcal{V}$ in a neighbourhood of $X$, $Y=P$, and $X$ is dense in $Y$. In this case, it is clear that
\[ {\rm H}_{q,\rig}^{\mathrm{BM}}(X,E)\cong {\rm H}^{2d-q}_\rig(X,E)\]
coincides with usual rigid cohomology, and the pairing constructed above is just the usual Poincar\'e pairing. Hence the the claim follows from \cite[Theorem 1.2.3]{Ked06a}.
\end{proof}

\begin{remark} Kedlaya's theorem is only stated for objects admitting a Frobenius structure. It is easily seen to extend to objects which are of Frobenius type in our sense. This is also the first occasion on which we have actually used the assumption that our isocrystals are of Frobenius type, it will also be used crucially in the proof of Theorem \ref{theo: DCon} below.
\end{remark}

\section{The trace map and pre-\texorpdfstring{$\mathscr{D}^\dagger$}{Ddag}-modules} \label{sec: trace Ddagger}

Our next goal is to show that the essential image of $\mathrm{sp}_!$ consists of overholonomic complexes (see Theorem \ref{theo: DCon} below). A crucial component of the proof will be to show that our construction of $\mathrm{sp}_!$ is `compatible' with the trace map in a suitable sense (for a precise statement, see Theorem \ref{theo: trace Ddag} below). Our main aim in this section is to prove this compatibility.

\subsection{Statement of compatibility of $\bm{\mathrm{sp}_!}$ with $\bm{\Tr}$}

To formulate the required compatibility result, we start by defining the $\mathscr{D}$-linear trace map for constructible isocrystals. Let $(X,Y,\fr{P})$ be a frame, with $\fr{P}$ smooth over $\mathcal{V}$. Suppose that $u:\fr{P}'\rightarrow \fr{P}$ is a smooth and proper morphism of formal schemes, of relative dimension $d$, and let $\mathscr{F}$ be a constructible $\cO_{\tube{X}_\fr{P}}$-module with overconvergent connection. Let us denote by $u:\tube{u^{-1}(X)}_{\fr{P}'}\rightarrow \tube{X}_\fr{P}$ the induced smooth and proper morphism of germs.

Since $\mathscr{F}$ is flat as an $\mathcal{O}_{\tube{X}_\fr{P}}$-module by Lemma \ref{lemma: flat}, the adjunction between $\mathbf{R}u_*$ and $\mathbf{L}u^*$ induces a morphism
\[ \mathbf{R}u_{*}(\Omega^\bullet_{\tube{u^{-1}(X)}_{\fr{P}'}/\tube{X}_\fr{P}}) \otimes_{\mathcal{O}_{\tube{X}_\fr{P}}} \mathscr{F} \rightarrow \mathbf{R}u_{*}(\Omega^\bullet_{\tube{u^{-1}(X)}_{\fr{P}'}/\tube{X}_\fr{P}}\otimes_{\mathcal{O}_{\tube{u^{-1}(X)}_{\fr{P}'}}} u^*\mathscr{F} ). \]

\begin{lemma} \label{lemma: projection formula constructible} This map is an isomorphism. 
\end{lemma}

\begin{proof} 
By d\'evissage we may reduce to the case where there exists a locally closed subscheme $a\colon Z\hookrightarrow X$, and a locally free $\mathcal{O}_{\tube{Z}_\fr{P}}$-module $\mathscr{G}$ such that $\mathscr{F}= a_!\mathscr{G}$. Let $a:u^{-1}(Z)\rightarrow u^{-1}(Z)$ denote the induced immersion. In this case, we may identify the LHS with
\[ \mathbf{R}u_{*}(\Omega^\bullet_{\tube{u^{-1}(X)}_{\fr{P}'}/\tube{X}_\fr{P}}) \otimes_{\mathcal{O}_{\tube{X}_\fr{P}}} a_!\mathscr{G} \cong  a_!\left(a^{-1}\mathbf{R}u_{*}\Omega^\bullet_{\tube{u^{-1}(X)}_{\fr{P}'}/\tube{X}_\fr{P}} \otimes_{\mathcal{O}_{\tube{Z}_\fr{P}}} \mathscr{G}\right)\]
and the RHS with
\begin{align*}
  \mathbf{R}u_{*}(\Omega^\bullet_{\tube{u^{-1}(X)}_{\fr{P}'}/\tube{X}_\fr{P}}\otimes_{\mathcal{O}_{\tube{u^{-1}(X)}_{\fr{P}'}}} u^*a_!\mathscr{G} ) &= \mathbf{R}u_{*}(\Omega^\bullet_{\tube{u^{-1}(X)}_{\fr{P}'}/\tube{X}_\fr{P}}\otimes_{\mathcal{O}_{\tube{u^{-1}(X)}_{\fr{P}'}}} a_!u^*\mathscr{G} ) \\
&= \mathbf{R}u_{*}a_!(\Omega^\bullet_{\tube{u^{-1}(Z)}_{\fr{P}'}/\tube{Z}_\fr{P}}\otimes_{\mathcal{O}_{\tube{u^{-1}(Z)}_{\fr{P}'}}} u^*\mathscr{G} ) \\
&= a_!\mathbf{R}u_{*}(\Omega^\bullet_{\tube{u^{-1}(Z)}_{\fr{P}'}/\tube{Z}_\fr{P}}\otimes_{\mathcal{O}_{\tube{u^{-1}(Z)}_{\fr{P}'}}} u^*\mathscr{G} ).
\end{align*}
Since $\mathscr{G}$ is locally free, the projection formula identifies the latter with
\[ a_!\left(\mathbf{R}u_{*} \Omega^\bullet_{\tube{u^{-1}(Z)}_{\fr{P}'}/\tube{Z}_\fr{P}} \otimes_{\mathcal{O}_{\tube{Z}_{\fr{P}}}} \mathscr{G} \right) ,\]
and it therefore suffices to show that the natural map
\[ a^{-1}\mathbf{R}u_{*}\Omega^\bullet_{\tube{u^{-1}(X)}_{\fr{P}'}/\tube{X}_\fr{P}} \rightarrow \mathbf{R}u_{*} \Omega^\bullet_{\tube{u^{-1}(Z)}_{\fr{P}'}/\tube{Z}_\fr{P}} \]
is an isomorphism. But since $\tube{Z}_\fr{P}$ is an stable under generalisation in $\tube{X}_\fr{P}$, this follows from Proposition \ref{prop: prop base changes}.
\end{proof}

We now take $X=P$, and consider the ($\mathscr{D}_{\fr{P}_K}$-linear) trace map 
\[ \mathrm{Tr}_{\fr{P}'_K/\fr{P}_K}:\mathbf{R}u_{*}(\Omega^\bullet_{\mathfrak{P}'_K/\mathfrak{P}_K})[2d] \rightarrow \mathcal{O}_{\fr{P}_K}. \] 
Tensoring with $\mathscr{F}\in\mathrm{Isoc}_\mathrm{cons}(\fr{P})$, we therefore obtain
\[  \Tr_{\mathscr{F}}\colon \mathbf{R}u_{*}(\Omega^\bullet_{\mathfrak{P}'_K/\mathfrak{P}_K}\otimes_{\mathcal{O}_{\mathfrak{P}'_K}} u^*\mathscr{F} )[2d] \rightarrow \mathscr{F}. \]
Pushing this forward along $\mathrm{sp}_\fr{P}:\fr{P}_K\rightarrow\fr{P}$, we obtain a $\mathscr{D}_{\fr{P}\Q}$-linear map
\[ u_+\mathrm{sp}_{\fr{P}'!}u^*\mathscr{F} \rightarrow \mathrm{sp}_{\fr{P}!}\mathscr{F}, \]
where
\[ u_+(-):= \mathbf{R}u_*\left( \mathscr{D}_{\fr{P}\leftarrow \fr{P}',\Q} \otimes^{\mathbf{L}}_{\mathscr{D}_{\fr{P}'\Q}}(-) \right) \cong \mathbf{R}u_*\left(  \Omega^\bullet_{\fr{P}'/\fr{P}} \otimes_{\mathcal{O}_{\fr{P}'}}(-) \right)[d] \]
is the usual direct image functor for $\mathscr{D}_{\fr{P}\Q}$-modules. This is the map that we wish to lift $\mathscr{D}^\dagger_{\fr{P}\Q}$-linearly. 

To do so, we define, for an arbitrary $\mathscr{D}_{\fr{P}'\Q}^\dagger$-module $\mathcal{M}$,
\[ u_+\mathcal{M}:= \mathbf{R}u_* \left(\mathscr{D}^\dagger_{\fr{P}\leftarrow\fr{P}',\Q} \otimes^{\mathbf{L}}_{\mathscr{D}^\dagger_{\fr{P}'\Q}} \mathcal{M} \right) \]
using the $(u^{-1}\mathscr{D}^\dagger_{\fr{P}\Q},\mathscr{D}^\dagger_{\fr{P}'\Q})$-bimodule $\mathscr{D}^\dagger_{\fr{P}\leftarrow\fr{P}',\Q}$. When $\mathcal{M}$ is coherent, this coincides with Berthelot's ind-pro definition of $u_+$ \cite[(4.3.7)]{Ber02}. Via the overconvergent Spencer resolution
\[ \mathscr{D}^\dagger_{\fr{P}\leftarrow\fr{P}',\Q}\isomto \Omega^\bullet_{\fr{P}'/\fr{P}}\otimes_{\mathcal{O}_\fr{P}'} \mathscr{D}^\dagger_{\fr{P}'\Q}[d] \] (see for example \cite[(4.2.1.1)]{Ber02}), we deduce that this also coincides with the relative de\thinspace Rham cohomology
\[ \mathbf{R}u_*\left(\Omega^\bullet_{\fr{P}'/\fr{P}} \otimes_{\mathcal{O}_\fr{P}} \mathcal{M} \right)[d].\]
In particular this shows that $u_+$ commutes with restriction of scalars along $\mathscr{D}_{(-)\Q}\rightarrow \mathscr{D}^\dagger_{(-)\Q}$. Since $\mathrm{sp}_{\fr{P}'!}u^*\mathscr{F}$ and $\mathrm{sp}_{\fr{P}!}\mathscr{F}$ are both naturally (complexes of) $\mathscr{D}^\dagger$-modules, via Corollary \ref{cor: dagger lifting of sp_*}, it makes sense to ask for a $\mathscr{D}^\dagger$-linear lifting of the trace map.

\begin{theorem} \label{theo: trace Ddag} Let $\mathfrak{P}$ and $\mathfrak{Q}$ be smooth formal schemes over $\mathcal{V}$, with $\mathfrak{Q}$ proper. Let $u\colon \mathfrak{P}':=\mathfrak{P} \times_\mathcal{V} \fr{Q}\rightarrow \fr{P}$ be the projection, and $\mathscr{F}\in \mathrm{Isoc}_\mathrm{cons}(\fr{P})$. Then there exists a $\mathscr{D}^\dagger_{\fr{P}\Q}$-linear morphism
\[ u_+\mathrm{sp}_{\fr{P}'!}u^*\mathscr{F}\rightarrow \mathrm{sp}_{\fr{P}!}\mathscr{F} \]
lifting the $\mathscr{D}_{\fr{P}\Q}$-linear trace map.
\end{theorem}

Note that we are not claiming any kind of `canonicity' of the $\mathscr{D}^\dagger$-linear lift, all we shall need is its existence.

\subsection{Categories of ind-$\pmb{\mathscr{D}}$-modules} \label{sec: cat Ddag}

Suppose that we have $u:\fr{P}'=\fr{P}\times_\mathcal{V}\fr{Q}\rightarrow \fr{P}$ as in the statement of Theorem \ref{theo: trace Ddag}. Let $Y\hookrightarrow \fr{P}$ be a closed immersion, and let us write $u$ as well for the induced morphism $\tube{u^{-1}(Y)}_{\fr{P}'}\rightarrow \tube{Y}_\fr{P}$. We will need a slightly more careful analysis of the construction in \S\ref{sec: pre Ddag sub} of a $\mathscr{D}^\dagger$-module on $\mathfrak{P}$ from a pre-$\mathscr{D}^\dagger$-module on $\tube{Y}_\mathfrak{P}$. In particular we consider the following categories:
\begin{itemize}
\item the category $\mathrm{Ind}_{\mathbf{LC}_K}(\mathscr{D}_{\tube{Y}_\fr{P}})$ of ind-objects in locally convex $\mathscr{D}_{\tube{Y}_\fr{P}}$-modules;
\item the full subcategory $\mathrm{Ind}_{\mathbf{Fre}_K}(\mathscr{D}_{\tube{Y}_\fr{P}})\subset \mathrm{Ind}_{\mathbf{LC}_K}(\mathscr{D}_{\tube{Y}_\fr{P}})$ of ind-Fr\'echet $\mathscr{D}_{\tube{Y}_\fr{P}}$-modules, that is, objects $\{\mathscr{F}_i\}_{i\in I}$ such that each $\mathscr{F}_i$ is a Fr\'echet $\mathscr{D}_{\tube{Y}_\fr{P}}$-module.
\end{itemize}
Similarly, we consider the following variants on the formal scheme $\fr{P}$:
\begin{itemize}
\item the category $\mathrm{Ind}_{\mathbf{LC}_K}(\mathscr{D}_{\fr{P}\Q})$ of ind-objects in locally convex $\mathscr{D}_{\fr{P}\Q}$-modules;
\item the full subcategory $\mathrm{Ind}_{\mathbf{Fre}_K}(\mathscr{D}_{\fr{P}\Q})\subset \mathrm{Ind}_{\mathbf{LC}_K}(\mathscr{D}_{\fr{P}\Q})$ of ind-Fr\'echet $\mathscr{D}_{\fr{P}\Q}$-modules, that is, objects $\{\mathscr{F}_i\}_{i\in I}$ such that each $\mathscr{F}_i$ is a Fr\'echet $\mathscr{D}_{\fr{P}\Q}$-module.
\end{itemize}
We require two slightly more subtle analogues on $\fr{P}$, which give categories of ind-modules over the `ind-ring' $\mathscr{D}^\bullet_{\fr{P}\Q}=\{\mathscr{D}^{\eta}_{\fr{P}\Q}\}_{\eta}$.

\begin{definition} We define the category $\mathrm{Ind}_{\mathbf{LC}_K}(\mathscr{D}^{\bullet}_{\fr{P}\Q})$ of ind-locally convex $\mathscr{D}^{\bullet}_{\fr{P}\Q}$-modules as follows. An object is an ind-locally convex $\mathcal{O}_{\fr{P}\Q}$-module $\{\mathscr{F}_i\}_{i\in I}$, together with, for each $\eta$, an element $i_\eta\in I$, and a continuous $\mathscr{D}^{\eta}_{\fr{P}\Q}$-module structure on $\mathscr{F}_i$ for $i\geq i_\eta$, which are compatible as $i$ and $\eta$ vary. The morphisms are the obvious ones. 

We define the full subcategory
\[ \mathrm{Ind}_{\mathbf{Fre}_K}(\mathscr{D}^{\bullet}_{\fr{P}\Q}) \subset \mathrm{Ind}_{\mathbf{LC}_K}(\mathscr{D}^{\bullet}_{\fr{P}\Q}) \]
of ind-Fr\'echet $\mathscr{D}^{\bullet}_{\fr{P}\Q}$-modules, which consists of objects $\{\mathscr{F}_i\}_{i\in I}$ such that each $\mathscr{F}_i$ is a Fr\'echet $\mathscr{D}^{\eta}_{\fr{P}\Q}$-module whenever $i\geq i_\eta$. 
\end{definition}

Note that taking the colimit of an ind-object and then forgetting the topology gives rise to a natural functor
\[ \mathrm{colim}: \mathrm{Ind}_{\mathbf{LC}_K}(\mathscr{D}^\bullet_{\fr{P}\Q})\rightarrow \mathrm{Mod}(\mathscr{D}^\dagger_{\fr{P}\Q}).\]
A pre-$\mathscr{D}^\dagger$-module $\left\{\mathscr{F}_i \right\}_{i\in I}$ on $\tube{Y}_\fr{P}$ is, in particular, an object of $\mathrm{Ind}_{\mathbf{Fre}}(\mathscr{D}_{\tube{Y}_\fr{P}})$, and in the special case that $Y=P$, Lemma \ref{lemma: pre Ddag sp} amounts to observing that, in this case, $\left\{\mathrm{sp}_{*}\mathscr{F}_i \right\}_{i\in I}\in \mathrm{Ind}_{\mathbf{Fre}_K}(\mathscr{D}_{\fr{P}\Q})$ can be uniquely promoted to an object of $\mathrm{Ind}_{\mathbf{Fre}_K}(\mathscr{D}^{\bullet}_{\fr{P}\Q})$. The uniqueness here is a consequence of the fact that Fr\'echet spaces are separated, which implies that the forgetful functor
\[\mathrm{Ind}_{\mathbf{Fre}_K}(\mathscr{D}^{\bullet}_{\fr{P}\Q}) \rightarrow \mathrm{Ind}_{\mathbf{LC}_K}(\mathscr{D}_{\fr{P}\Q}) \]
is fully faithful. We will, however, need something slightly stronger than this.

\begin{lemma} \label{lemma: ff top fre}Let $\mathscr{F} \in \mathrm{Ind}_{\mathbf{LC}_K}(\mathscr{D}^{\bullet}_{\fr{P}\Q})$ and $\mathscr{G}\in \mathrm{Ind}_{\mathbf{Fre}_K}(\mathscr{D}^{\bullet}_{\fr{P}\Q})$. Then the map
\[ \mathrm{Hom}_{\mathrm{Ind}_{\mathbf{LC}_K}(\mathscr{D}^{\bullet}_{\fr{P}\Q})}(\mathscr{F},\mathscr{G}) \rightarrow \mathrm{Hom}_{\mathrm{Ind}_{\mathbf{LC}_K}(\mathscr{D}_{\fr{P}\Q})}(\mathscr{F},\mathscr{G}) \]
is bijective.
\end{lemma}

\begin{proof}
The map is clearly injective, we wish to show that it is surjective. In other words, we need to show that any continuous $\mathscr{D}_{\fr{P}\Q}$-linear map $\alpha:\mathscr{F}\rightarrow \mathscr{G}$ is automatically $\mathscr{D}^\bullet_{\fr{P}\Q}$-linear. We consider the map
\begin{align*}
\mathscr{D}^\bullet_{\fr{P}\Q} \times \mathscr{F} &\rightarrow \mathscr{G} \\
(P,s)&\mapsto P\alpha(s)-\alpha(Ps)
\end{align*} 
measuring the failure of $\alpha$ to be $\mathscr{D}^\bullet_{\fr{P}\Q}$-linear, and note that it vanishes on the dense subspace $\mathscr{D}_{\fr{P}\Q} \times \mathscr{F}$ of the source. Since the target $\mathscr{G}$ is separated, it therefore vanishes on all of $\mathscr{D}^\bullet_{\fr{P}\Q} \times \mathscr{F}$, and thus $\alpha$ is $\mathscr{D}^\bullet_{\fr{P}\Q}$-linear.
\end{proof}

We also need the following ind-version of affinoid-acyclicity.

\begin{definition} We say that an ind-sheaf $\{\mathscr{F}_i\}_{i\in I}$ on an analytic variety $\mathscr{X}\in \mathbf{An}_K$ is ind-affinoid-acyclic if for all affinoids $V\subset\mathscr{X}$ there exists a cofinal subcategory $J\subset I$ such that, for all $i\in J$, and all $q>0$, ${\rm H}^q(V,\mathscr{F}_i)=0$.
\end{definition}

Note that this is stronger than saying that $\mathrm{colim}_{i\in I}\mathscr{F}_i$ is affinoid acyclic, but weaker than saying $\{\mathscr{F}_i\}_{i\in I}$ is isomorphic (as an ind-object) to one whose individual pieces are affinoid-acyclic (since $J$ is allowed to vary with $V$). Another way of phrasing the definition would be to say that for each affinoid $V\subset \mathscr{X}$, the ind-object $\{{\rm H}^q(V,\mathscr{F}_i)\}_{i\in I}$ is zero.

\subsection{$\pmb{\mathscr{D}^\dagger}$ lifts of trace morphisms: the general setup} \label{sec: dear god}

In order to prove Theorem \ref{theo: trace Ddag}, we will consider the following situation:
\begin{itemize}
\item $\mathscr{F}$ is an affinoid-acyclic $\mathscr{D}_{\tube{Y}_\fr{P}}$-module;
\item $\mathscr{G}$ is an affinoid-acyclic $\mathscr{D}_{\tube{u^{-1}(Y)}_{\fr{P}'}}$-module;
\item $\alpha:\mathbf{R}u_{\dR*}\mathscr{G}[2d] \rightarrow \mathscr{F}$ is a morphism in ${\bf D}^b(\mathscr{D}_{\tube{Y}_\fr{P}})$;
\item $\{\mathscr{F}_i\}_{i\in I}$ is an ind-affinoid-acyclic pre-$\mathscr{D}^\dagger$-module structure on $\mathscr{F}$;
\item $\{\mathscr{G}_i\}_{i\in I}$ is an ind-affinoid-acyclic pre-$\mathscr{D}^\dagger$-module structure on $\mathscr{G}$ (for simplicity, the two ind-objects are assumed to be indexed by the same category).
\end{itemize}
We are interested in formulating what it means to give a pre-$\mathscr{D}^\dagger$ lift of $\alpha$.

\subsubsection{}\label{sec: Cech} In the situation of \S\ref{sec: dear god}, we may compute the de\thinspace Rham pushforward $\mathbf{R}u_{\dR*}\left\{\mathscr{G}_i\right\}_{i\in I}$ of the ind-object $\{\mathscr{G}_i\}_{i\in I}$ via the relative \v{C}ech complex associated to an open affine cover of $\fr{P}'$. Indeed, for any such cover $\mathcal{U}=\left\{ \fr{U}_\alpha \right\}_{\alpha=1}^n$, and any $0\leq j\leq d$, we can consider the \v{C}ech complex
\begin{align*} \prod_{\alpha=1}^n u_{*} &\left[ \left.\left(\Omega^j_{\tube{u^{-1}(Y)}_{\fr{P}'}/\tube{Y}_\fr{P}}\otimes_{\mathcal{O}_{\tube{u^{-1}(Y)}_{\fr{P}'}}} \mathscr{G}_i\right)\right\vert_{\tube{u^{-1}(Y)}_{\fr{U}_\alpha}} \right] \rightarrow \ldots  \\ 
\ldots &\rightarrow u_{*}\left[\left.\left(\Omega^j_{\tube{u^{-1}(Y)}_{\fr{P}'}/\tube{Y}_\fr{P}}\otimes_{\mathcal{O}_{\tube{u^{-1}(Y)}_{\fr{P}'}}} \mathscr{G}_i\right)\right\vert_{\tube{u^{-1}(Y)}_{\fr{U}_{1}\cap\ldots\cap\, \fr{U}_{n}}}\right]
\end{align*}
which is a complex of locally convex $\mathcal{O}_{\tube{Y}_\fr{P}}$-modules.\footnote{This is the reason for using \v{C}ech cohomology, since it is not clear how to endow the abstract higher cohomology groups (or the abstract higher direct images) of a topological sheaf with any `natural' topoology.} Letting $j$ vary, we therefore obtain a double complex in the category of locally convex $\mathcal{O}_{\tube{Y}_\fr{P}}$-modules. We denote the associated simple complex by $\check{\mathbf{R}}_\mathcal{U}u_{\dR*}\mathscr{G}_i$. We can therefore take the truncation 
\[ \left(\tau^{\leq 2d}\check{\mathbf{R}}_\mathcal{U}u_{\dR*}\mathscr{G}_i\right)^n = \begin{cases}
\left(\check{\mathbf{R}}_\mathcal{U}u_{\dR*}\mathscr{G}_i\right)^n & n<2d\\
 \ker {\rm d}^{2d} & n=2d \\
 0 & n>2d
\end{cases}\]
which, as a subcomplex of $\check{\mathbf{R}}_\mathcal{U}u_{\dR*}\mathscr{G}_i$, is a complex of locally convex $\mathcal{O}_{\tube{Y}_\fr{P}}$-modules. 

Letting $i$ vary, we obtain a complex of ind-objects $\left\{\tau^{\leq 2d}\check{\mathbf{R}}^{2d}_\mathcal{U}u_{\dR*}\mathscr{G}_i\right\}_{i\in I}$ in the category of locally convex $\mathcal{O}_{\tube{Y}_\fr{P}}$-modules.  Since $\fr{P}'\cong  \fr{P}\times_\mathcal{V}\fr{Q}$ is a product, there is a ring homomorphism $u^{-1}_K\mathscr{D}_{\tube{Y}_\fr{P}} \rightarrow \mathscr{D}_{\tube{u^{-1}(Y)}_{\fr{P}'}}$ which endows each term in the \v{C}ech complex  
\[ \prod_{\alpha_1<\ldots<\alpha_l} u_{*}\left[\left.\left(\Omega^j_{\tube{u^{-1}(Y)}_{\fr{P}'}/\tube{Y}_\fr{P}}\otimes_{\mathcal{O}_{\tube{u^{-1}(Y)}_{\fr{P}'}}} \mathscr{G}_i\right)\right\vert_{\tube{u^{-1}(Y)}_{\fr{U}_{\alpha_1}\cap\ldots\cap \,\fr{U}_{\alpha_l}}}\right] \]
with a continuous action of $\mathscr{D}_{\tube{Y}_\fr{P}}$ via its action on the second factor of the tensor product. We therefore obtain an object 
\[ \left\{\tau^{\leq 2d}\check{\mathbf{R}}_\mathcal{U}u_{\dR*}\mathscr{G}_i\right\}_{i\in I} \in \mathbf{Ch}^{+}(\mathrm{Ind}_{\mathbf{LC}_K}(\mathscr{D}_{\tube{Y}_\fr{P}}))\]
which, by the ind-affinoid acyclicity assumption on $\{\mathscr{G}_i\}_{i\in I}$, recovers the relative de\thinspace Rham cohomology $ \mathbf{R}u_{\dR*}\mathscr{G}$ upon taking the colimit.

\begin{definition} \label{defn: lift pre Ddag} In the situation of \S\ref{sec: dear god}, a pre-$\mathscr{D}^\dagger$ lifting of $\alpha$ (relative to $\mathcal{U}$) is a morphism
\[ \alpha^\dagger \colon \left\{\tau^{\leq 2d}\check{\mathbf{R}}_\mathcal{U}u_{\dR*}\mathscr{G}_i[2d]\right\}_{i\in I}  \rightarrow \left\{\mathscr{F}_i\right\}_{i\in I} \]
in $\mathbf{Ch}^{+}(\mathrm{Ind}_{\mathbf{LC}_K}(\mathscr{D}_{\tube{Y}_\fr{P}}))$, which recovers $\mathrm{\alpha}$ upon taking the colimit (and passing to the derived category).
\end{definition}

\begin{remark} \label{rem: ind cech general} The construction of the ind-object $\left\{\tau^{\leq 2d}\check{\mathbf{R}}_\mathcal{U}u_{\dR*}\mathscr{G}_i\right\}_{i\in I}$ works more generally if $\tube{Y}_\fr{P}$ is replaced by an arbitrary open subspace $V\subset \fr{P}_K$, and $\tube{u^{-1}(Y)}_{\fr{P}'}$ is replaced by $u^{-1}(V)$. However, it only makes sense to talk about pre-$\mathscr{D}^\dagger$-modules on tube of closed subspaces of the special fibre. 
\end{remark}

\begin{remark} \label{rem: Ddag lift complex} We shall also be interested in situations where $\mathscr{F},\mathscr{G}, \left\{\mathscr{F}_i\right\}_{i\in I}$ and $\left\{\mathscr{G}_i\right\}_{i\in I}$ are complexes of $\mathscr{D}$-modules, rather than just sheaves. In this case, we shall understand 
\[ \left\{\tau^{\leq 2d}\check{\mathbf{R}}_\mathcal{U}u_{\dR*}\mathscr{G}_i\right\}_{i\in I} \]
to mean the complex obtained by applying $\tau^{\leq 2d}\check{\mathbf{R}}_\mathcal{U}u_{\dR*}$ to each of the individual terms of the complex $\left\{\mathscr{G}_i\right\}_{i\in I}$, and then taking te associated simple complex. Thus we still recover the derived pushforward $\mathbf{R}u_{\dR*}\mathscr{G}$ upon taking the colimit. However, the reader should be warned that this is \emph{not} the same as truncating the relative \v{C}ech complex of $\left\{\mathscr{G}_i\right\}_{i\in I}$, in particular, it need not be concentrated in degrees $[0,2d]$.
\end{remark}

\subsection{Pushing forward along $\bm{\mathrm{sp}}$} \label{subsubsec: u+}

Suppose we are in the situation of \S\ref{sec: dear god}, but assume moreover that $Y=P$. Via affinoid acyclicity of $\mathscr{F}$ and $\mathscr{G}$, we can view $\mathbf{R}\mathrm{sp}_{\fr{P}*}\alpha$ as a morphism
\[ \mathbf{R}\mathrm{sp}_{\fr{P}*}\alpha\colon \mathbf{R}\mathrm{sp}_{\fr{P}'*}\mathbf{R}u_{ \dR }\mathscr{G}[2d] = \mathbf{R}u_{\dR *}\mathrm{sp}_{\fr{P}*}\mathscr{G}[2d] \rightarrow \mathrm{sp}_{\fr{P}*}\mathscr{F}, \]
or in other words as a morphism
\[ u_+\mathrm{sp}_{\fr{P}'*}\mathscr{G}[d] \rightarrow \mathrm{sp}_{\fr{P}*}\mathscr{F}  \]
in ${\bf D}^b(\mathscr{D}_{\fr{P}\Q})$. Since $\mathrm{sp}_{\fr{P}'*}\mathscr{G}$  is a $\mathscr{D}^\dagger_{\fr{P}'\Q}$-module, and $\mathrm{sp}_{\fr{P}*}\mathscr{F}$ is a $\mathscr{D}^\dagger_{\fr{P}'\Q}$, be Lemma \ref{lemma: pre Ddag sp}, it makes sense to ask for a $\mathscr{D}_{\fr{P}\Q}^\dagger$-linear lift of $\mathbf{R}\mathrm{sp}_{\fr{P}*}\alpha$. We will show that any pre-$\mathscr{D}^\dagger$ lifting of $\alpha$ can be used to construct just such a lift.

\subsubsection{} 

To do this, we first construct a similar \v{C}ech version of the $\mathscr{D}^\bullet$-module pushforward of the ind-object $\{\mathrm{sp}_{\fr{P}'*}\mathscr{G}_i\}_{i\in I}\in \mathrm{Ind}_{\mathbf{Fre}_K}(\mathscr{D}^\bullet_{\fr{P}'\Q})$ as follows. The fact that $\mathfrak{P}'\cong \fr{P}\times_\mathcal{V}\fr{Q}$ is a product means that there is a ring homomorphism 
\[ u^{-1}\mathscr{D}^\eta_{\fr{P}\Q}\rightarrow \mathscr{D}^\eta_{\fr{P}'\Q} \]
for each $\eta$, and these are compatible as $\eta$ varies. If $\mathcal{U}=\left\{ \fr{U}_\alpha \right\}_{\alpha=1}^n$ is an open cover of $\mathfrak{P}'$, this enables us to view each ind-object
\[  \left\{ u_* \left[\left.\left( \Omega^j_{\fr{P}'/\fr{P}} \otimes_{\mathcal{O}_{\fr{P}'}} \mathrm{sp}_{\fr{P}'*}\mathscr{G}_i \right)\right\vert_{\fr{U}_{\alpha_1}\cap\ldots\cap \,\fr{U}_{\alpha_l}} \right]\right\}_{i\in I}  \]  
as a $\mathscr{D}^\bullet_{\fr{P}\Q}$-module via the action on the second factor. As in \S\ref{sec: Cech} above, we can then form the \v{C}ech-de\thinspace Rham complex $\left\{\check{\mathbf{R}}_\mathcal{U}u_{\dR*}\mathrm{sp}_{\fr{P}'*}\mathscr{G}_i\right\}_{i\in I}$, and pass to the truncation $\left\{\tau^{\leq 2d}\check{\mathbf{R}}^{2d}_\mathcal{U}u_{\dR*}\mathrm{sp}_{\fr{P}'*}\mathscr{G}_i\right\}_{i\in I}$, which is naturally endowed with a continuous $\mathscr{D}^\bullet_{\fr{P}\Q}$-action. 

Thus we may view $\left\{\tau^{\leq 2d}\check{\mathbf{R}}_\mathcal{U}u_{\dR*}\mathrm{sp}_{\fr{P}'*}\mathscr{G}_i\right\}_{i\in I}$ as an object of $\mathbf{Ch}^{+}(\mathrm{Ind}_{\mathbf{LC}_K}(\mathscr{D}^\bullet_{\fr{P}\Q}))$. Since $\{\mathscr{G}_i\}_{i\in I}$ is ind-affinoid acyclic, taking the colimit (and passing to the derived category) then recovers the $\mathscr{D}^\dagger$-module pushforward $u_+\mathrm{sp}_{\fr{P}'*}\mathscr{G}[d]$. On the other hand, since $\left\{\tau^{\leq 2d}\check{\mathbf{R}}_\mathcal{U}u_{\dR*}\mathscr{G}_i\right\}_{i\in I}$ is a chain complex of objects in $\mathrm{Ind}_{\mathbf{LC}_K}(\mathscr{D}_{\fr{P}_K})$, we can also view $\left\{\mathrm{sp}_{\fr{P}*}\tau^{\leq 2d}\check{\mathbf{R}}_\mathcal{U}u_{\dR*}\mathscr{G}_i\right\}_{i\in I}$
as an object of $\mathbf{Ch}^{+}(\mathrm{Ind}_{\mathbf{LC}_K}(\mathscr{D}_{\fr{P}\Q}))$.

 \begin{lemma} There is an isomorphism of complexes
\[ \left\{\tau^{\leq 2d}\check{\mathbf{R}}_\mathcal{U}u_{\dR*}\mathrm{sp}_{\fr{P}'*}\mathscr{G}_i\right\}_{i\in I}\cong \left\{ \mathrm{sp}_{\fr{P}*} \tau^{\leq 2d}\check{\mathbf{R}}_\mathcal{U}u_{\dR*}\mathscr{G}_i\right\}_{i\in I} \]
in $\mathrm{Ind}_{\mathbf{LC}_K}(\mathscr{D}_{\fr{P}\Q})$.
 \end{lemma} 
 
 \begin{proof}
 Since $\mathrm{sp}_{\fr{P}*}$ is left exact, it commutes with truncation $ \tau^{\leq 2d}$, so it suffices to show the claim for the untruncated \v{C}ech complexes. For each $i\in I$, the complex on the LHS is the total complex of a double complex with terms of the form
 \[ \prod_{\alpha_1<\ldots<\alpha_l}u_* \left[\left.\left( \Omega^j_{\fr{P}'/\fr{P}} \otimes_{\mathcal{O}_{\fr{P}'}} \mathrm{sp}_{\fr{P}'*}\mathscr{G}_i \right)\right\vert_{\fr{U}_{\alpha_1}\cap\ldots\cap \,\fr{U}_{\alpha_l}} \right] \]
 whereas the complex on the RHS is the total complex of a double complex with terms of the form
 \begin{align*} \mathrm{sp}_{\fr{P}*}\prod_{\alpha_1<\ldots<\alpha_l} u_{*}&  \left[\left.  \left(  \Omega^j_{\fr{P}'_K/\fr{P}_K}\otimes_{\mathcal{O}_{\fr{P}'_K}} \mathscr{G}_i\right)\right\vert_{\fr{U}_{\alpha_1K}\cap\ldots\cap \,\fr{U}_{\alpha_lK}}\right] 
 \\ &= \prod_{\alpha_1<\ldots<\alpha_l} u_*\left[\left.  \left( \mathrm{sp}_{\fr{P}'*}\left(\Omega^j_{\fr{P}'_K/\fr{P}_K}\otimes_{\mathcal{O}_{\fr{P}'_K}} \mathscr{G}_i \right) \right)  \right\vert_{\fr{U}_{\alpha_1}\cap\ldots\cap \,\fr{U}_{\alpha_l}} \right].
 \end{align*}
 It therefore suffices to show that
\[ \Omega^j_{\fr{P}'/\fr{P}} \otimes_{\mathcal{O}_{\fr{P}'}} \mathrm{sp}_{\fr{P}'*}\mathscr{G}_i  \cong \mathrm{sp}_{\fr{P}'*}\left(\Omega^j_{\fr{P}'_K/\fr{P}_K}\otimes_{\mathcal{O}_{\fr{P}'_K}} \mathscr{G}_i\right) \]
as locally convex $\mathscr{D}_{\fr{P}'\Q}$-modules. Since $\Omega^j_{\fr{P}'/\fr{P}}$ is a locally finite free $\mathcal{O}_{\fr{P}'}$-module, and $\Omega^j_{\fr{P}'_K/\fr{P}_K}=\mathrm{sp}_{\fr{P}'}^*\Omega^j_{\fr{P}'/\fr{P}}$ (via module pullback along $\mathrm{sp}_{\fr{P}'}$), this is clear.
 \end{proof}
 
\subsubsection{}\label{sec: this used to be shorter}
 
Hence, if we are given a pre-$\mathscr{D}^\dagger$-lifting of $\alpha$ to a morphism
\[ \alpha^\dagger :\left\{\tau^{\leq 2d}\check{\mathbf{R}}_\mathcal{U}u_{\dR*}\mathscr{G}_i\right\}_{i\in I}[2d] \rightarrow \left\{\mathscr{F}_i \right\}_{i\in I} \]
in $\mathrm{Ind}_{\mathbf{LC}_K}(\mathscr{D}_{\fr{P}_K})$, then we may consider the composite
\[  \left\{\tau^{\leq 2d}\check{\mathbf{R}}_\mathcal{U}u_{\dR*}\mathrm{sp}_{\fr{P}'*}\mathscr{G}_i[2d]\right\}_{i\in I} \cong\left\{ \mathrm{sp}_{\fr{P}*}\tau^{\leq 2d}\check{\mathbf{R}}_\mathcal{U}u_{\dR*}\mathscr{G}_i[2d]\right\}_{i\in I} \overset{\mathrm{sp}_{\fr{P}*}\alpha^\dagger}{\longrightarrow}  \left\{\mathrm{sp}_{\fr{P}*}\mathscr{F}_i \right\}_{i\in I}\]
in $\mathbf{Ch}^+(\mathrm{Ind}_{\mathbf{LC}_K}(\mathscr{D}_{\fr{P}\Q}))$. By Lemma \ref{lemma: ff top fre} this map lifts uniquely to $\mathbf{Ch}^+(\mathrm{Ind}_{\mathbf{LC}_K}(\mathscr{D}^{\bullet}_{\fr{P}\Q}))$, and by passing to the colimit, we therefore obtain a $\mathscr{D}^\dagger_{\fr{P}\Q}$-linear map 
\[ u_+\mathrm{sp}_{\fr{P}'*}\mathscr{G}[d] \rightarrow \mathrm{sp}_{\fr{P}*}\mathscr{F} \]
lifting the $\mathscr{D}_{\fr{P}\Q}$-linear map $\mathbf{R}\mathrm{sp}_*\alpha$. 

\subsection{$\pmb{\mathscr{D}^\dagger}$ lifts for overconvergent isocrystals} 

 \label{sec: pre Ddag lift start} To prove Theorem \ref{theo: trace Ddag}, we now want to apply the construction of \S\ref{subsubsec: u+} to the explicit complexes computing $\mathrm{sp}_!$ from \S\ref{sec: calc Rsp} and \S\ref{sec: pre Dd}. The starting point for this will be to consider the case of locally free isocrystals supported on locally closed subschemes of $\fr{P}$. So let $Y\hookrightarrow \fr{P}$ be a closed subscheme, $X\hookrightarrow Y$ a strongly affine open immersion, and $\mathscr{F}$ a coherent $j_X^\dagger\mathcal{O}_{\tube{Y}_\fr{P}}$-module with overconvergent connection. We will take $\mathscr{G}=u^*\mathscr{F}$. We then have the trace map
\[ \mathrm{Tr}_{\mathscr{F}}: \mathbf{R}u_{\dR*}u^*\mathscr{F}[2d] \rightarrow \mathscr{F} \]
thanks to Lemma \ref{lemma: projection formula constructible}.

Let $j_{\underline{\lambda}}\colon V_{\underline{\lambda}}\hookrightarrow \tube{Y}_\fr{P}$ for $\underline{\lambda}\in \mathcal{S}$ be the usual cofinal system of neighbourhoods of $\tube{X}_\fr{P}$ inside $\tube{Y}_\fr{P}$, as in \S\ref{sec: fund} above. Then $j_{\underline{\lambda}}\colon u^{-1}(V_{\underline{\lambda}})\rightarrow \tube{u^{-1}(Y)}_{\fr{P}'}$ is a cofinal system of neighbourhoods of $\tube{u^{-1}(X)}_{\fr{P}'}$ in $\tube{u^{-1}(Y)}_{\fr{P}'}$. After passing to a cofinal subset of such $\underline{\lambda}$, we may assume that $\mathscr{F}$ extends to a module with connection $\mathscr{F}_{\underline{\lambda}}$ on each $V_{\underline{\lambda}}$. Thus, by Theorem \ref{theo: pre Ddag over}, $\{j_{\underline{\lambda}*}\mathscr{F}_{\underline{\lambda}}\}_{\underline{\lambda}\in \mathcal{S}}$ is a pre-$\mathscr{D}^\dagger$-module structure on $\mathscr{F}$, and $\{j_{\underline{\lambda}*}u^*\mathscr{F}_{\underline{\lambda}}\}_{\underline{\lambda}\in \mathcal{S}}$ is a pre-$\mathscr{D}^\dagger$-module structure on $u^*\mathscr{F}$.

Both of these pre-$\mathscr{D}^\dagger$-modules are ind-affinoid acyclic, since on each affinoid subset $W\subset \tube{Y}_\fr{P}$ (resp. $W\subset \tube{u^{-1}(Y)}_{\fr{P}'}$) there is a cofinal system of $V_{\underline{\lambda}}$ (resp. $u^{-1}(V_{\underline{\lambda}})$) such that $W\cap V_{\underline{\lambda}}$ (resp. $W\cap u^{-1}(V_{\underline{\lambda}})$) is affinoid. Hence we are indeed in the situation of \S\ref{sec: dear god}. We want to produce a pre-$\mathscr{D}^\dagger$ lifting of $\mathrm{Tr}_\mathscr{F}$ as in Definition \ref{defn: lift pre Ddag}. 

\begin{lemma} \label{lem: cech commutes with jdag 1} For any open affine cover $\mathcal{U}=\{\fr{U}_\alpha\}$ of $\fr{P}'$, there is an isomorphism of complexes
\[ \left\{\tau^{\leq 2d}\check{\mathbf{R}}_\mathcal{U}u_{\dR*}j_{\underline{\lambda}*}u^*\mathscr{F}_{\underline{\lambda}}\right\}_{\underline{\lambda}\in \mathcal{S}} \cong \left\{ j_{\underline{\lambda}*}\tau^{\leq 2d}\check{\mathbf{R}}_\mathcal{U}u_{\dR*}u^*\mathscr{F}_{\underline{\lambda}} \right\}_{\underline{\lambda}\in \mathcal{S}}  \]
in $\mathrm{Ind}_{\mathbf{LC}_K}(\mathscr{D}_{\tube{Y}_\fr{P}})$.
\end{lemma}

\begin{remark} The RHS here makes sense by Remark \ref{rem: ind cech general}.
\end{remark}

\begin{proof}
Again, since $j_{\underline{\lambda}}$ is left exact, it suffices to prove the statement for the untruncated complexes. For each $\underline{\lambda}\in \mathcal{S}$, the complex on the LHS is the total complex of a double complex with terms of the form
 \[ \prod_{\alpha_1<\ldots<\alpha_l}u_{*}\left[\left.\left( \Omega^j_{\tube{u^{-1}(Y)}_{\fr{P}'}/\tube{Y}_\fr{P}}\otimes_{\mathcal{O}_{\tube{u^{-1}(Y)}_{\fr{P}'}}} j_{\underline{\lambda}*}u^*\mathscr{F}_{\underline{\lambda}} \right)\right\vert_{\tube{u^{-1}(Y)}_{\fr{U}_{\alpha_1}\cap\ldots\cap \,\fr{U}_{\alpha_l}}}\right] \]
 whereas the complex on the RHS is the total complex of a double complex with terms of the form
 \begin{align*}  j_{\underline{\lambda}*}\prod_{\alpha_1<\ldots<\alpha_l}u_{*}&\left[\left.\left( \Omega^j_{u^{-1}(V_{\underline{\lambda}})/V_{\underline{\lambda}}}\otimes_{\mathcal{O}_{u^{-1}(V_{\underline{\lambda}})}} u^*\mathscr{F}_{\underline{\lambda}} \right)\right\vert_{u^{-1}(V_{\underline{\lambda}})\,\cap\,\fr{U}_{\alpha_1K}\cap\ldots\cap \,\fr{U}_{\alpha_lK}} \right] \\
 &=\prod_{\alpha_1<\ldots<\alpha_l}u_{*}\left[ \left(j_{\underline{\lambda}*}\left.\left( \Omega^j_{u^{-1}(V_{\underline{\lambda}})/V_{\underline{\lambda}}}\otimes_{\mathcal{O}_{u^{-1}(V_{\underline{\lambda}})}} u^*\mathscr{F}_{\underline{\lambda}} \right)\right)\right\vert_{\tube{u^{-1}(Y)}_{\fr{U}_{\alpha_1}\cap\ldots\cap \,\fr{U}_{\alpha_l}}} \right].
 \end{align*}
 It therefore suffices to show that
\[  \Omega^j_{\tube{u^{-1}(Y)}_{\fr{P}'}/\tube{Y}_\fr{P}}\otimes_{\mathcal{O}_{\tube{u^{-1}(Y)}_{\fr{P}'}}} j_{\underline{\lambda}*}u^*\mathscr{F}_{\underline{\lambda}}  \cong j_{\underline{\lambda}*}\left( \Omega^j_{u^{-1}(V_{\underline{\lambda}})/V_{\underline{\lambda}}}\otimes_{\mathcal{O}_{u^{-1}(V_{\underline{\lambda}})}} u^*\mathscr{F}_{\underline{\lambda}} \right) \]
as locally convex $\mathscr{D}_{\tube{u^{-1}(Y)}_{\fr{P}'}}$-modules. Again, this follows from the fact that $\Omega^j_{\tube{u^{-1}(Y)}_{\fr{P}'}/\tube{Y}_\fr{P}}$ is locally finite free.
\end{proof}

Since each $\mathscr{F}_{\underline{\lambda}}$ is coherent on $V_{\underline{\lambda}}$, we have isomorphisms
\[ \mathcal{H}^{2d}(\check{\mathbf{R}}_\mathcal{U}u_{K\dR*}u^*\mathscr{F}_{\underline{\lambda}}) \cong \mathbf{R}^{2d}u_{*}u^*\mathscr{F}_{\underline{\lambda}} \]
between \v{C}ech and derived functor cohomology for the morphism $u:u^{-1}(V_{\underline{\lambda}})\rightarrow V_{\underline{\lambda}}$. Moreover, since $\mathscr{F}_{\underline{\lambda}}$ is locally free, we may apply the projection formula, Lemma \ref{lemma: proj form}, to define a trace map
\[ \mathrm{Tr}_{\mathscr{F}_{\underline{\lambda}}}\colon \mathcal{H}^{2d}(\check{\mathbf{R}}_\mathcal{U}u_{\dR*}u^*\mathscr{F}_{\underline{\lambda}}) \rightarrow  \mathscr{F}_{\underline{\lambda}}.\]

\begin{lemma} For each $\underline{\lambda}$, the trace map $\mathrm{Tr}_{\mathscr{F}_{\underline{\lambda}}}$ is continuous.
\end{lemma}

\begin{proof} Note that $u:u^{-1}(V_{\underline{\lambda}})\rightarrow V_{\underline{\lambda}}$ is smooth and proper, and that $\mathscr{F}_{\underline{\lambda}}$ and $u^*\mathscr{F}_{\underline{\lambda}}$ are coherent sheaves, endowed with their canonical topologies. Also note that we may localise on the base $V_{\underline{\lambda}}$ and therefore replace it by a suitable open affinoid $U=\spa{A,A^+} \subset V_{\underline{\lambda}}$. In this case, the sections of $\check{\mathbf{R}}_\mathcal{U}^{2d}u_{\dR*}u^*\mathscr{F}_{\underline{\lambda}}$ on $U$ form Banach spaces, and by Kiehl's theorem \cite{Kie67a} their cohomology groups are finitely generated $A$-modules. It therefore suffices to show that the given topology on these cohomology groups is separated. Indeed, if so, then it is the canonical topology as finitely generated $A$-modules, with respect to which any $A$-linear homomorphism is continuous. 

From here we could proceed in a couple of ways. Firstly, we could observe that Kiehl's original argument actually proves separatedness of the relative \v{C}ech cohomology groups (see the last paragraph before Satz 2.6 in \cite{Kie67a}). Although this is only for coherent cohomology, rather than de\thinspace Rham cohomology, it is easy to extend the result to the case of de\thinspace Rham cohomology. Alternatively, we can appeal to Lemma \ref{lemma: finite closed} below.
\end{proof}

\begin{lemma} \label{lemma: finite closed} Let $A$ be a Noetherian Banach $K$-algebra, and $f:M\rightarrow N$ a continuous morphism of Banach $A$-modules. If $\mathrm{coker}(f)$ is finitely generated as an $A$-module, then $f(M)$ is closed in $N$.
\end{lemma}

\begin{proof}
This is a straightforward generalisation of an argument of Crew \cite[Lemma 1.1]{Cre06}. Let $T\subset N$ be a finitely generated submodule surjecting onto $\mathrm{coker}(f)$. If we equip $T$ with its canonical (Banach) topology as a finitely generated $A$-module (so that any $A$-linear map out of $T$ is continuous), then the map $\pi: M\oplus T\rightarrow N$ is a continuous surjection of Banach spaces, and hence strict by the open mapping theorem. Let $K$ denote the kernel of $\pi$.

Then $M\subset M+K \subset M\oplus T$, and $M+K$ is the preimage of its image inside $(M\oplus T)/M\cong T$. Since $A$ is Noetherian, this image is finitely generated, and hence a closed submodule of $T$; it therefore follows that $M+K$ is closed in $M\oplus T$. Hence we have closed submodules $K\subset M+K \subset M\oplus T$, and so quotienting out by $K$ we see that $f(M)=(M+K)/K$ is closed in $N\cong (M\oplus T)/K$.
\end{proof}

\subsubsection{}  \label{sec: pre Ddag lift start2} 
Hence we can view each individual trace map $\mathrm{Tr}_{\mathscr{F}_{\underline{\lambda}}}$ as a morphism
\[ \tau^{\leq 2d}\check{\mathbf{R}}_\mathcal{U}u_{\dR*}u^*\mathscr{F}_{\underline{\lambda}}[2d] \rightarrow \mathscr{F}_{\underline{\lambda}} \]
of locally convex $\mathscr{D}_{\tube{Y}_\fr{P}}$-modules. Letting $\underline{\lambda}$ vary then gives rise to a morphism
\begin{align*} \left\{j_{\underline{\lambda}*}\mathrm{Tr}_{\mathscr{F}_{\underline{\lambda}}}\right\}_{\underline{\lambda}\in\mathcal{S}}\colon \left\{\tau^{\leq 2d}\check{\mathbf{R}}_\mathcal{U}u_{\dR*}j_{\underline{\lambda}*}u^*\mathscr{F}_{\underline{\lambda}}[2d]\right\}_{\underline{\lambda}\in\mathcal{S}} &\cong \left\{ j_{\underline{\lambda}*}\tau^{\leq 2d}\check{\mathbf{R}}_\mathcal{U}u_{\dR*}u^*\mathscr{F}_{\underline{\lambda}}[2d] \right\}_{\underline{\lambda}\in \mathcal{S}}  \\
&\rightarrow \{j_{\underline{\lambda}*}\mathscr{F}_{\underline{\lambda}}\}_{\underline{\lambda}\in\mathcal{S}}
\end{align*}
in $\mathbf{Ch}^+(\mathrm{Ind}_{\mathbf{LC}_K}(\mathscr{D}_{\tube{Y}_\fr{P}}))$. This is the pre-$\mathscr{D}^\dagger$-lifting of 
\[ \Tr_{\mathscr{F}}: \mathbf{R}u_{\dR*}u^*\mathscr{F}[2d] \rightarrow \mathscr{F}\]
we are after.

\subsection{Transporting $\pmb{\mathscr{D}^\dagger}$ lifts along $\bm{j_X^\dagger}$ and $\bm{\mathcal{R}_{i_Z,\infty}}$}

 \label{sec: pre Ddag lfit jdag} The key point now will be to show that pre-$\mathscr{D}^\dagger$ lifts of a morphism $\alpha$ as in Definition \ref{defn: lift pre Ddag} can be transported along the functors $j_X^\dagger$ and $\mathcal{R}_{i_Z,\infty}$ used to construct the complexes $\mathcal{R}_{\overline{P}_{\bullet},\infty}j^\dagger_{P_\bullet} \mathscr{F}$ appearing in the definition of the $\mathscr{D}^\dagger$-linear version of $\mathbf{R}\mathrm{sp}_*$. 

We first treat the case of $j_X^\dagger$. Suppose therefore that we are in the situation of \S\ref{sec: dear god}, with affinoid acyclic $\mathscr{D}$-modules $\mathscr{F}$ and $\mathscr{G}$ on $\tube{Y}_\fr{P}$ and $\tube{u^{-1}(Y)}_{\fr{P}'}$ respectively. . Suppose furthermore that we have a pre-$\mathscr{D}^\dagger$-lifting of some map 
\[ \alpha:\mathbf{R}u_{\dR*}\mathscr{G}[2d]\rightarrow \mathscr{F},\]
in the sense of Definition \ref{defn: lift pre Ddag}. Given a strongly affine open immersion $j:X\hookrightarrow Y$, we shall show how to construct a natural pre-$\mathscr{D}^\dagger$ lifting of the induced map
\[ j_X^\dagger\alpha:\mathbf{R}u_{\dR*}j_{u^{-1}(X)}^\dagger\mathscr{G}[2d] \cong j_{X}^\dagger\mathbf{R}u_{\dR*}\mathscr{G}[2d] \rightarrow \mathscr{F}. \]
(Note that commuting $\mathbf{R}u_{\dR*}$ and $j^\dagger$ is justified by Proposition \ref{prop: prop base changes}.)

We are given ind-affinoid-acyclic pre-$\mathscr{D}^\dagger$-module structures $\{\mathscr{F}_i\}_{i\in I}$ and $\{\mathscr{G}_i\}_{i\in I}$ on $\mathscr{F}$ and $\mathscr{G}$ as in \S\ref{sec: Cech}, an open cover  $\mathcal{U}=\left\{\fr{U}_\alpha\right\}$ of $\fr{P}'$, and a map
\[ \alpha^\dagger \colon \left\{\tau^{\leq 2d}\check{\mathbf{R}}_\mathcal{U}u_{\dR*}\mathscr{G}_i[2d]\right\}_{i\in I}  \rightarrow \left\{\mathscr{F}_i\right\}_{i\in I}   \]
in $\mathbf{Ch}^{+}(\mathrm{Ind}_{\mathbf{LC}_K}(\mathscr{D}_{\fr{P}\Q}))$ lifting $\alpha$.

Again, we let $j_{\underline{\lambda}}\colon V_{\underline{\lambda}}\hookrightarrow \tube{Y}_\fr{P}$ be the usual cofinal system of neighbourhoods of $\tube{X}_\fr{P}$ inside $\tube{Y}_\fr{P}$ as in \S\ref{sec: fund} above. Thus $j_{\underline{\lambda}}\colon u^{-1}(V_{\underline{\lambda}})\rightarrow \tube{u^{-1}(Y)}_{\fr{P}'}$ is a cofinal system of neighbourhoods of $\tube{u^{-1}(X)}_{\fr{P}'}$ in $\tube{u^{-1}(Y)}_{\fr{P}'}$. Then the pre-$\mathscr{D}^\dagger$-module structures $\{j_{\underline{\lambda}*}j_{\underline{\lambda}}^{-1}\mathscr{F}_i\}_{(\underline{\lambda},i)\in \mathcal{S}\times I}$ and $\{j_{\underline{\lambda}*}j_{\underline{\lambda}}^{-1}\mathscr{G}_i\}_{(\underline{\lambda},i)\in \mathcal{S}\times I}$ on $j_X^\dagger\mathscr{F}$ and $j_{u^{-1}(X)}^\dagger\mathscr{G}$ are also ind-affinoid-acyclic, since on each affinoid subset of $\tube{Y}_\fr{P}$ (resp. $\tube{u^{-1}(Y)}_{\fr{P}'}$) there is a cofinal system of $V_{\underline{\lambda}}$ (resp. $u^{-1}(V_{\underline{\lambda}})$) which are affinoid.

\begin{lemma} There is an isomorphism of complexes
\[\left\{\tau^{\leq 2d}\check{\mathbf{R}}_\mathcal{U}u_{\dR*} j_{\underline{\lambda}*}j_{\underline{\lambda}}^{-1}\mathscr{G}_i \right\}_{(\underline{\lambda},i)\in \mathcal{S}\times I} \cong \left\{ j_{\underline{\lambda}*}j_{\underline{\lambda}}^{-1}\tau^{\leq 2d}\check{\mathbf{R}}_\mathcal{U}u_{\dR*}\mathscr{G}_i \right\}_{(\underline{\lambda},i)\in \mathcal{S}\times I} \]
in $\mathrm{Ind}_{\mathbf{LC}_K}(\mathscr{D}_{\tube{Y}_\fr{P}})$. 
\end{lemma}

\begin{proof}
This is entirely similar to the proof of Lemma \ref{lem: cech commutes with jdag 1}.
\end{proof}

Thus applying $j_{\underline{\lambda}*}j_{\underline{\lambda}}^{-1}$ to the given map
\[ \left\{\tau^{\leq 2d}\check{\mathbf{R}}_\mathcal{U}u_{\dR*}\mathscr{G}_i[2d]\right\}_{i\in I} \rightarrow \left\{\mathscr{F}_i\right\}_{i\in I}   \]
gives us a morphism
\[ \left\{\tau^{\leq 2d}\check{\mathbf{R}}_\mathcal{U}u_{\dR*} j_{\underline{\lambda}*}j_{\underline{\lambda}}^{-1}\mathscr{G}_i[2d] \right\}_{(\underline{\lambda},i)\in \mathcal{S}\times I} \rightarrow \left\{j_{\underline{\lambda}*}j_{\underline{\lambda}}^{-1} \mathscr{F}_i\right\}_{(\underline{\lambda},i)\in \mathcal{S}\times I} \]
in 
$\mathbf{Ch}^{+}(\mathrm{Ind}_{\mathbf{LC}_K}(\mathscr{D}_{\fr{P}\Q}))$
lifting
\[ j_X^\dagger \alpha \colon \mathbf{R}^{2d}u_{\dR*}j_{u^{-1}(X)}^\dagger\mathscr{G} = j_X^\dagger\mathbf{R}^{2d}u_{\dR*}\mathscr{G} \rightarrow j_X^\dagger\mathscr{F}.\]

\subsubsection{}\label{sec: pre Ddag lfit closed}

We now treat the case of $\mathcal{R}_{i_Z,\infty}$. So suppose that we are given a closed immersion $i_Z:Z\hookrightarrow Y$, with pullback $i_{u^{-1}(Z)}\colon u^{-1}(Z)\hookrightarrow u^{-1}(Y)$, and affinoid-acyclic $\mathscr{D}$-modules $\mathscr{F}$ and $\mathscr{G}$ on $\tube{Z}_\fr{P}$ and $\tube{u^{-1}(Z)}_{\fr{P}'}$ respectively. Let
\[ \alpha:\mathbf{R}u_{\dR*}\mathscr{G}[2d]\rightarrow \mathscr{F}\]
by a morphism in ${\bf D}^b(\mathscr{D}_{\tube{Z}_\fr{P}})$. Let $\{\mathscr{F}_i\}_{i\in I}$ and $\{\mathscr{G}_i\}_{i\in I}$ be ind-affinoid-acyclic pre-$\mathscr{D}^\dagger$-module structures  on $\mathscr{F}$ and $\mathscr{G}$, and let
\[\alpha^\dagger \colon \tau^{\leq 2d}\check{\mathbf{R}}_\mathcal{U}u_{\dR*}\left\{\mathscr{G}_i\right\}_{i\in I}  \rightarrow \left\{\mathscr{F}_i\right\}_{i\in I} \]
a pre-$\mathscr{D}^\dagger$-lifting of $\alpha$. Applying $\mathbf{R}i_{Z*}$ to $\alpha$ gives a morphism
\[ \mathbf{R}i_{Z*}\alpha: \mathbf{R}i_{Z*}\mathbf{R}u_{\dR*}\mathscr{G}[2d]=\mathbf{R}u_{\dR*}\mathbf{R}i_{u^{-1}(Z)*}\mathscr{G}[2d] \rightarrow \mathbf{R}i_{Z*}\mathscr{F} \]
and we shall use $\alpha^\dagger$ to construct a pre-$\mathscr{D}^\dagger$ lifting of $\mathbf{R}i_{Z*}\alpha$. Recall that the pre-$\mathscr{D}^\dagger$-module structure on $\mathbf{R}i_{u^{-1}(Z)*}\mathscr{G}$ is given by the complex
\[ \left\{ \prod_{n\geq n_0} i_{u^{-1}(Z),n*}\left( \mathscr{G}_{i_n}|_{[u^{-1}(Z)]_n}\right) \rightarrow \prod_{n\geq n_0} i_{u^{-1}(Z),n*} \left(\mathscr{G}_{i_n}|_{[u^{-1}(Z)]_n}\right) \right\}_{n_0,(i_n)_{n\geq n_0}} \]

\begin{lemma} For each integer $n_0\geq 1$, and each sequence $(i_n)_{n\geq n_0}$, there is an isomorphism of complexes
\[ \tau^{\leq 2d}\check{\mathbf{R}}_\mathcal{U}u_{\dR*}\prod_{n\geq n_0} i_{u^{-1}(Z),n*} \left(\mathscr{G}_{i_n}|_{[u^{-1}(Z)]_n} \right) \cong  \prod_{n\geq n_0} i_{Z,n*}  \left((\tau^{\leq 2d}\check{\mathbf{R}}_\mathcal{U}u_{\dR*}\mathscr{G}_{i_n})|_{[Z]_n} \right) \]
in $\mathrm{Ind}_{\mathbf{LC}_K}(\mathscr{D}_{\tube{Y}_\fr{P}})$. 
\end{lemma}

\begin{proof}
As usual, since the functor $\prod_{n\geq n_0} i_{Z,n*}\left(-|_{[Z]_n}\right)$ is left exact, it suffices to prove the untruncated version, and since $\check{\mathbf{R}}_\mathcal{U}u_{\dR*}$ commutes with products, it suffices to show that
\[ \check{\mathbf{R}}_\mathcal{U}u_{\dR*}i_{u^{-1}(Z),n*} \left(\mathscr{G}_{i_n}|_{[u^{-1}(Z)]_n} \right) \cong i_{Z,n*} \left((\check{\mathbf{R}}_\mathcal{U}u_{\dR*}\mathscr{G}_{i_n})|_{[Z]_n}\right)  \]
for each $n$. The LHS consists of terms of the form
\[ \prod_{\alpha_1<\ldots<\alpha_l}u_{*}\left[\left.\left( \Omega^j_{\tube{u^{-1}(Y)}_{\fr{P}'}/\tube{Y}_\fr{P}}\otimes_{\mathcal{O}_{\tube{u^{-1}(Y)}_{\fr{P}'}}} i_{u^{-1}(Z),n*} \left(\mathscr{G}_{i_n}|_{[u^{-1}(Z)]_n} \right) \right)\right\vert_{\tube{u^{-1}(Y)}_{\fr{U}_{\alpha_1}\cap\ldots\cap \,\fr{U}_{\alpha_l}}}\right] \]
whereas the RHS consists of terms of the form
 \begin{align*} &i_{Z,n*}\left.\left(\prod_{\alpha_1<\ldots<\alpha_l}  u_{*}  \left[\left.\left( \Omega^j_{\tube{u^{-1}(Y)}_{\fr{P}'}/\tube{Y}_\fr{P}}\otimes_{\mathcal{O}_{\tube{u^{-1}(Y)}_{\fr{P}'}}}\mathscr{G}_{i_n} \right)\right\vert_{\tube{u^{-1}(Y)}_{\fr{U}_{\alpha_1}\cap\ldots\cap \,\fr{U}_{\alpha_l}}}\right] \right) \right\vert_{[Z]_n} \\
 =& \prod_{\alpha_1<\ldots<\alpha_l}  u_{*} i_{u^{-1}(Z),n*} \left.\left[\left.\left( \Omega^j_{\tube{u^{-1}(Y)}_{\fr{P}'}/\tube{Y}_\fr{P}}\otimes_{\mathcal{O}_{\tube{u^{-1}(Y)}_{\fr{P}'}}}\mathscr{G}_{i_n} \right)\right\vert_{\tube{u^{-1}(Y)}_{\fr{U}_{\alpha_1}\cap\ldots\cap \,\fr{U}_{\alpha_l}}}\right]\right\vert_{[u^{-1}(Z)]_n}
  \end{align*}
 Thus we need to prove that
 \begin{align*} &\left( \Omega^j_{\tube{u^{-1}(Y)}_{\fr{P}'}/\tube{Y}_\fr{P}}\otimes_{\mathcal{O}_{\tube{u^{-1}(Y)}_{\fr{P}'}}} i_{u^{-1}(Z),n*} \left(\mathscr{G}_{i_n}|_{[u^{-1}(Z)]_n} \right) \right)  \\ \cong  & i_{u^{-1}(Z),n*} \left.\left( \Omega^j_{\tube{u^{-1}(Y)}_{\fr{P}'}/\tube{Y}_\fr{P}}\otimes_{\mathcal{O}_{\tube{u^{-1}(Y)}_{\fr{P}'}}}\mathscr{G}_{i_n} \right)\right\vert_{[u^{-1}(Z)]_n}
 \end{align*}
 which again is a simple consequence of the local finite freeness of $\Omega^\bullet_{\fr{P}'_K/\fr{P}_K}$. 
\end{proof}

\subsubsection{} \label{sec: is it finally over?}

We thus obtain a pre-$\mathscr{D}^\dagger$ lifting of 
\[ \mathbf{R}i_{Z*}\mathbf{R}i_{Z*}\alpha: \mathbf{R}u_{\dR*}\mathscr{G}[2d]=\mathbf{R}u_{K\dR*}\mathbf{R}i_{u^{-1}(Z)*}\mathscr{G}[2d] \rightarrow  \mathbf{R}i_{Z*}\mathscr{F}\]
by applying the functor
\[ \left(\prod_{n\geq n_0}i_{Z,n*}(-)|_{[Z]_n}\rightarrow  \prod_{n\geq n_0}i_{Z,n*}(-)|_{[Z]_n}\right)_{n_0,(i_n)_{n\geq n_0}}  \]
 to the given map 
\[ \alpha^\dagger \colon \left\{\tau^{\leq 2d}\check{\mathbf{R}}_\mathcal{U}u_{\dR*}\mathscr{G}_i[2d]\right\}_{i\in I}  \rightarrow \left\{\mathscr{F}_i\right\}_{i\in I}. \]

\subsubsection{}

We can now finally complete the proof of Theorem \ref{theo: trace Ddag}.
 
\begin{proof}[Proof of Theorem \ref{theo: trace Ddag}] If we now take a constructible isocrystal $\mathscr{F}$ on $\fr{P}$, and choose a good stratification $\{P_\alpha\}$ for $\mathscr{F}$, then $\{P'_\alpha:=u^{-1}(P_\alpha)\}$ is a good stratification for $u^*\mathscr{F}$. If we choose an open affine cover of $\fr{P}'$, then starting with the construction of \S\ref{sec: pre Ddag lift start}, and repeatedly apply the constructions of \S\ref{sec: pre Ddag lfit jdag} (and bearing in mind Remark \ref{rem: Ddag lift complex}), we therefore obtain a pre-$\mathscr{D}^\dagger$-lifting of the trace map
\[ \Tr: \mathbf{R}u_{\dR*}\mathcal{R}_{\overline{P}'_{\bullet},\infty}j^\dagger_{P'_\bullet} u^*\mathscr{F}[2d] \rightarrow \mathcal{R}_{\overline{P}_{\bullet},\infty}j^\dagger_{P_\bullet} \mathscr{F}. \]
By \S\ref{subsubsec: u+} this then gives rise to a $\mathscr{D}_{\fr{P}\Q}^\dagger$-linear lifting of 
\[ u_+\mathbf{R}\mathrm{sp}_{\fr{P}'*}u^*\mathscr{F}[d]\rightarrow \mathbf{R}\mathrm{sp}_{\fr{P}*}\mathscr{F}. \]
This construction is compatible with refining the choice of open cover of $\fr{P}'$, as well as refining the choice of good stratification for $\mathscr{F}$, and hence is independent of these choices.
\end{proof}

\subsection{Commutativity of $\bm{\mathrm{sp}_!}$ with finite \'etale pushforward for overconvergent isocrystals} \label{sec: trace over}

As mentioned at the beginning of \S\ref{sec: trace Ddagger}, Theorem \ref{theo: trace Ddag} will form a key component of the proof of overholonomicity of the essential image of $\mathrm{sp}_!$. In the actual proof, however, we will mainly use a corollary of Theorem \ref{theo: trace Ddag}. Roughly speaking, this corollary states that if $(f,g,u)\colon (X',Y',\fr{P}')\rightarrow (X,Y,\fr{P})$ is a morphism of frames, then $\mathrm{sp}_!$ commutes with $u_+$ for objects supported on $X'$ (for a precise statement, see Theorem \ref{theo: tr conc fin et} below). Again, this result will take a little effort to set up carefully, and will require some additional hypotheses on the morphism of frames under consideration.

\begin{setup} \label{setup: conc tr} In particular, we will consider a morphism of frames
\[\xymatrix@C=1em@R=1em{  X'  \ar[rr]^{j} \ar[dd]^f & & Y' \ar[rr]^{i'} \ar[dd]^g &  & \fr{P}' \ar[dd]^u \\ & \square  \\  X \ar[rr]^{j} & & Y \ar[rr]^{i}  & & \fr{P} } 
\]
such that:
\begin{enumerate}
\item $\fr{P}$ is smooth and affine over $\mathcal{V}$, and $j\colon X\rightarrow Y$ is strongly affine;
\item $\mathfrak{P}'\cong \mathfrak{P}\times_\mathcal{V} \mathfrak{Q}$ with $\mathfrak{Q}$ smooth and projective over $\mathcal{V}$, and $u$ is the projection onto the first factor;
\item the left hand square is Cartesian, and $f$ is (finite) \'etale;
\item the conormal sheaf $\mathcal{C}_{X'/u^{-1}(X)}$ of $X'$ in $u^{-1}(X)$ is free.
\end{enumerate}
\end{setup}

\begin{warning} Note that, even assuming conditions (1), (2), and (3), the last condition (4) cannot in general be guaranteed locally on $X$. However, it can be guaranteed locally on $X'$, and therefore generically on $X$. \end{warning}

\subsubsection{}

Suppose that we are in the situation of Setup \ref{setup: conc tr}. Let $\fr{P}=\spf{R}$. Since $u$ is projective it is algebraisable, so there exists a smooth projective morphism of schemes $\mathcal{P}' \rightarrow \spec{R}$ which recovers $u$ upon taking formal completions. Since $X'$ is affine, we may choose a basis of the conormal sheaf $\mathcal{C}_{X'/u^{-1}(X)}$ and lift to obtain functions $x_1,\dots,x_d\in \mathcal{O}_{\mathcal{P}'}$, defined on an open neighbourhood of $X'$, and inducing a rational map $\mathcal{P}'\dashrightarrow \P^d_{R}$, which becomes \'etale in a neighbourhood of $X'$ upon taking formal completions. Resolving the indeterminacy locus of the map $\mathcal{P}'\dashrightarrow \P^d_{R}$
and formally completing we obtain a diagram of frames
\begin{equation} \label{eqn: diamond} 
\xymatrix{ & (X',Y'',\fr{P}'') \ar[dl]_{(\mathrm{id},g',u')} \ar[dr]^{(f,h,v)} \\ (X',Y',\fr{P}') \ar[dr]_{(f,g,u)} & & (X,Y,\widehat{\P}^d_\fr{P}) \ar[dl]^{(\mathrm{id},\mathrm{id},\pi)} \\ & (X,Y,\fr{P}) }
\end{equation} 
where:
\begin{itemize}
\item $Y$ is embedded in $\P^d_P$ via the zero section $x_1=\ldots =x_d=0$;
\item the map $g'\colon Y''\rightarrow Y'$ is projective;
\item the map $u'\colon\fr{P}''\rightarrow \fr{P}'$ is an isomorphism in a neighbourhood of $X'$;
\item the map $v\colon\fr{P}''\rightarrow \widehat{\P}^d_\fr{P}$ is \'etale in a neighbourhood of $X'$.
\end{itemize}
In particular, we have $\tube{X'}_{\fr{P}'}\cong\tube{X'}_{\fr{P}''}$ and $\tube{X}_{\widehat{\P}^d_\fr{P}}\cong \tube{X}_\fr{P}\times_K \mathbb{D}^d_K(0;1)$. To distinguish it from the map $\tube{f}\colon\tube{X'}_{\fr{P}'}\rightarrow \tube{X}_\fr{P}$, we shall denote by $\psi$ the induced map $\tube{X'}_{\fr{P}''}\rightarrow \tube{X}_{\widehat{\P}^d_\fr{P}}$.

\begin{lemma} \label{lemma: not quite finite etale} The complex $\mathbf{R}\psi_{\dR*}\mathcal{O}_{\tube{X'}_{\fr{P}''}}$ of $\mathscr{D}_{\tube{X}_{\widehat{\P}^{d}_\fr{P}}}$-modules is concentrated in degree $0$, where it is coherent as an $\cO_{\tube{X}_{\widehat{\P}^{d}_\fr{P}}}$-module.
\end{lemma}

\begin{proof}
Since $v:\fr{P}''\rightarrow \widehat{\P}^{d}_\fr{P}$ is \'etale in a neighbourhood of $X$, the non-\'etale locus of the induced map $\tube{h}:\tube{Y''}_{\fr{P}'} \rightarrow \tube{Y}_{\widehat{\P}^d_\fr{P}}$ is a closed analytic subspace disjoint from $\tube{X'}_{\fr{P}''}$. Thus $\Omega^\bullet_{\tube{X'}_{\fr{P}''}/\tube{X}_{\widehat{\P}^{d}_\fr{P}}}=\cO_{\tube{X'}_{\fr{P}''}}$, and so
\[ \mathbf{R}\psi_{\dR*}\mathcal{O}_{\tube{X'}_{\fr{P}''}}=\mathbf{R}\psi_{*}\mathcal{O}_{\tube{X'}_{\fr{P}''}}. \]
The map $X'\rightarrow v^{-1}(X)$ is proper, and hence a closed immersion. Moreover, since $v$ is \'etale in a neighbourhood of $X$, and $f$ is \'etale, it follows that $X'\rightarrow v^{-1}(X)$ is \'etale. Hence $X'$ is a union of connected components of $v^{-1}(X)$, and so $\tube{X'}_{\fr{P}''}$ is a union of connected components of $\tube{v^{-1}(X)}_{\fr{P}''}$. Since $v$ is proper, $v:\tube{v^{-1}(Y)}_{\fr{P}''} \rightarrow \tube{Y}_{\widehat{\P}^d_\fr{P}}$ is proper, and hence by Proposition \ref{prop: prop base changes} we deduce that $\mathbf{R}v_*\cO_{\tube{v^{-1}(X)}_{\fr{P}''}}$ has coherent cohomology sheaves. Hence so does its direct summand $\mathbf{R}\psi_*\cO_{\tube{X'}_{\fr{P}''}}$.

Consider the higher direct image $\mathbf{R}^qv_*\cO_{\tube{v^{-1}(Y)}_{\fr{P}''}}$ for some $q>0$, this is a coherent sheaf on $\tube{Y}_{\widehat{\P}^d_\fr{P}}$. We know that, after restricting to $\tube{X}_{\widehat{\P}^d_\fr{P}}$, this contains $\mathbf{R}^q\psi_*\cO_{\tube{X'}_{\fr{P}''}}$ as a direct summand. By Proposition \ref{prop: coherent sheaves colimit neighbourhoods} we see that there exists an open neighbourhood $V$ of $\tube{X}_{\widehat{\P}^d_\fr{P}}$ in $\tube{Y}_{\widehat{\P}^d_\fr{P}}$, and a decomposition
\[ \left.\left(\mathbf{R}^qv_*\cO_{\tube{v^{-1}(Y)}_{\fr{P}''}}\right)\right\vert_V \cong \mathscr{F}_1\oplus \mathscr{F}_2 \]
such that $\mathscr{F}_1|_{\tube{X}_{\widehat{\P}^d_\fr{P}}} = \mathbf{R}^q\psi_*\cO_{\tube{X'}_{\fr{P}''}}$. 

Since $\psi$ is \'etale, we know that every point $x\in \tube{X}_{\widehat{\P}^d_\fr{P}}$ has only finitely many preimages under $\psi$. Hence, by Proposition \ref{prop: prop base changes}, the stalk of $\mathbf{R}^q\psi_*\cO_{\tube{X'}_{\fr{P}''}}$ at any maximal point of $\tube{X}_{\widehat{\P}^d_\fr{P}}$ is zero. The support of $\mathscr{F}_1$ is therefore a closed analytic subspace of $V$, not containing any maximal point of $\tube{X}_{\widehat{\P}^d_\fr{P}}$. It is therefore disjoint from $\tube{X}_{\widehat{\P}^d_\fr{P}}$, which completes the proof.
\end{proof}

Since $f:X'\rightarrow X$ is finite \'etale, every object of $\mathrm{Isoc}(X',Y'')$ is a direct summand of an object in the essential image of $f^*:\mathrm{Isoc}(X,Y)\rightarrow \mathrm{Isoc}(X',Y'')$. Thus Lemma \ref{lemma: not quite finite etale} remains true with $\cO_{\tube{X'}_{\fr{P}''}}$ replaced by (the realisation of) any overconvergent isocrystal $E\in \mathrm{Isoc}(X',Y'')$. We can check using the condition in \cite[Theorem 4.3.9]{LS07} that overconvergence is preserved, thus $\psi_*=\mathbf{R}\psi_{\dR*}$ induces a functor
\[ f_* \colon\mathrm{Isoc}(X',Y')=\mathrm{Isoc}(X',Y'') \rightarrow \mathrm{Isoc}(X,Y)  \]
which is both a right and a left adjoint to $f^*$.

\subsubsection{} 
We now let $i_X=i\circ j:X\rightarrow P$ denote the given locally closed immersion, take $E\in \mathrm{Isoc}(X',Y')$ and consider the constructible isocrystal $i_{X!}(f_*E)_\fr{P}$ on $\fr{P}$. Thus by Theorem \ref{theo: trace Ddag} we have the $\mathscr{D}^\dagger_{\fr{P}\Q}$-linear trace map
\[ \Tr_{f_*E}:u_+\mathrm{sp}_{\fr{P}'!}u^*i_{X!}(f_*E)_\fr{P} \rightarrow \mathrm{sp}_{\fr{P}!}i_{X!}(f_*E)_\fr{P} \] 
of complexes of $\mathscr{D}^\dagger_{\fr{P}\Q}$-modules. On the other hand, if we similarly let $i_{X'}=i'\circ j':X'\rightarrow P'$ denote the given locally closed immersion, and compose the canonical map
\[ i_{X'!}(f^*f_*E)_{\fr{P}'} \rightarrow u^*i_{X!}(f_*E)_\fr{P} \]
with (the realisation on $\fr{P}'$ of) the adjunction map
\[ E\rightarrow f^*f_*E, \]
we obtain a map of constructible isocrystals
\[  i_{X'!}E_{\fr{P}'}  \rightarrow u^*i_{X!}(f_*E)_\fr{P}  \]
on $\fr{P}'_K$. Applying $\mathrm{sp}_{\fr{P}'!}$ gives us a $\mathscr{D}^\dagger_{\fr{P}\Q}$-linear map
\[  \mathrm{sp}_{\fr{P}'!}i_{X'!}E_{\fr{P}'}  \rightarrow \mathrm{sp}_{\fr{P}'!}u^*i_{X!}(f_*E)_\fr{P},\]
and, finally, applying $u_+$ and composing this with the above map $\Tr_{f_*E}$, we obtain a $\mathscr{D}^\dagger_{\fr{P}\Q}$-linear map
\[ \mathrm{Tr}_E: u_+\mathrm{sp}_{\fr{P}'!} i_{X'!}E_{\fr{P}'}\rightarrow \mathrm{sp}_{\fr{P}!}i_{X!}(f_*E)_{\fr{P}}.\]

\begin{theorem} \label{theo: tr conc fin et} Consider Setup \ref{setup: conc tr}, and let $E\in \mathrm{Isoc}(X',Y')$. Then the trace map 
\[ \mathrm{Tr}_E: u_+\mathrm{sp}_{\fr{P}'!} i_{X'!}E_{\fr{P}'}\rightarrow \mathrm{sp}_{\mathfrak{P}!}i_{X!}(f_*E)_{\fr{P}}. \]
is an isomorphism in ${\bf D}^b(\mathscr{D}^\dagger_{\fr{P}\Q})$. 
\end{theorem} 

\begin{proof} It suffices to prove the claim after restricting along $\mathscr{D}_{\fr{P}\Q}\rightarrow \mathscr{D}^\dagger_{\fr{P}\Q}$, or indeed along $\mathcal{O}_{\fr{P}\Q}\rightarrow \mathscr{D}^\dagger_{\fr{P}\Q}$. In this case, we can describe the map alternatively as follow as follows. Let $d$ denote the relative dimension of $u$. Then we have the morphism of germs
\[ \psi \colon \tube{X'}_{\fr{P}'}=\tube{X'}_{\fr{P}''} \rightarrow \tube{X}_{\widehat{\P}^d_\fr{P}} = \tube{X}_{\fr{P}} \times_K \D^d_K(0;1^-) \]
factoring $\tube{f}\colon \tube{X'}_{\fr{P}'}\rightarrow \tube{X}_\fr{P}$ via the first projection $\pi$. Via the projection formula we obtain a trace map
\[ \mathbf{R}\pi_{!}\left( \Omega^\bullet_{\tube{X}_{\fr{P}} \times_K \D^d_K(0;1^-)/\tube{X}_{\fr{P}}} \otimes_{\mathcal{O}_{\tube{X}_{\fr{P}} \times_K \D^d_K(0;1^-)}}	 \pi^*(f_*E)_\fr{P} \right)[2d] \rightarrow (f_*E)_\fr{P}  \]
which is in fact an isomorphism, because the map
\[ \mathbf{R}\pi_{!}\left( \Omega^\bullet_{\tube{X}_{\fr{P}} \times_K \D^d_K(0;1^-)/\tube{X}_{\fr{P}}}  \right)[2d] \rightarrow \mathcal{O}_{\tube{X}_\fr{P}} \]
is an isomorphism by the results of \S\ref{subsec: f! prop}. Since the categories of coherent modules with overconvergent connection on $\tube{X}_\fr{P}$ and $\tube{X}_\fr{P}\times_K \D^d_K(0;1^-)$ are equivalent, via the (de\thinspace Rham) pushforward and pullback functors, we have a natural isomorphism
\[ \psi_*E_{\fr{P}'} \isomto \pi^*(f_*E)_\fr{P}, \]
and composing this with the above trace map gives an isomorphism
\begin{align*} \mathbf{R}\tube{f}_{!} & \left( \Omega^\bullet_{\tube{X'}_{\fr{P}'}/\tube{X}_\fr{P}}  \otimes_{\mathcal{O}_{\tube{X'}_{\fr{P}'}}} E_{\fr{P}'} \right)[2d] \\
 &\cong \mathbf{R}\pi_{!}\left( \Omega^\bullet_{\tube{X}_{\fr{P}} \times_K \D^d_K(0;1^-)/\tube{X}_{\fr{P}}}  \otimes_{\mathcal{O}_{\tube{X}_{\fr{P}} \times_K \D^d_K(0;1^-)}} \psi_*E_{\fr{P}'} \right) \isomto (f_*E)_\fr{P} .
\end{align*}
Taking the extension by zero along $i_X\colon X\rightarrow P$, and applying the projection formula, we thus obtain an isomorphism
\[ \mathbf{R}u_{*}\left(  \Omega^\bullet_{\fr{P}'_K/\fr{P}_K} \otimes_{\mathcal{O}_{\fr{P}_K'}}  i_{X'!}E_{\fr{P}'} \right)[2d] \isomto i_{X!}(f_*E)_\fr{P}, \]
which, after applying $\mathbf{R}\mathrm{sp}_{\fr{P}*}$ and shifting, is precisely the $\mathscr{D}_{\fr{P}\Q}$-linear map
\[ \mathrm{Tr}_E: u_+\mathrm{sp}_{\fr{P}'!} i_{X'!}E_{\fr{P}'}\rightarrow i_{X!}(f_*E)_{\fr{P}}. \]
Hence we are done.
\end{proof}

\section{Finiteness results for \texorpdfstring{$\mathrm{sp}_!$}{sp!}} \label{sec: finite}

In this section we introduce the dual constructible t-structure on the category of overholonomic complexes, and show that the essential image of 
\[ \mathrm{sp}_!: \mathrm{Isoc}_{\mathrm{cons}}(\mathfrak{P}) \rightarrow {\bf D}^b(\mathscr{D}^\dagger_{\fr{P}\Q}) \]
is contained in the heart of this t-structure. We also prove some functoriality properties of $\mathrm{sp}_!$, namely compatibility with pullbacks, and pushforwards along locally closed immersions. 

\subsection{t-structures on arithmetic \texorpdfstring{$\pmb{\mathscr{D}^\dagger}$}{D}-modules} 

For any prime $\ell\neq p$, and any variety $X/k$, there are two well-known t-structures on the triangulated category ${\bf D}^b_c(X_\et,\Q_\ell)$ of bounded constructible complexes of $\ell$-adic sheaves on $X_\et$.  The ordinary t-structure is the most straightforward to define, and has as its heart the category $\mathrm{Con}(X_\et,\Q_\ell)$ of constructible $\ell$-adic sheaves on $X$. One can also define the perverse t-structure on ${\bf D}^b_c(X_\et,\Q_\ell)$, whose heart $\mathrm{Perv}(X_\et,\Q_\ell)$ is called the category of perverse sheaves on $X$. In some ways this category is better behaved than $\mathrm{Con}(X_\et,\Q_\ell)$, for example it is both Artinian and Noetherian, and is stable under the Verdier dual functor $\mathbf{D}_X$. 

The fact that $\mathrm{Con}(X_\et,\Q_\ell)$ is \emph{not} stable under Verdier duality means that there is a third natural t-structure on ${\bf D}^b_c(X_\et,\Q_\ell)$, the dual constructible t-structure, whose heart is canonically anti-equivalent to $\mathrm{Con}(X_\et,\Q_\ell)$, via $\mathbf{D}_X$. This third t-structure appears less often in the literature than the other two, and is, in a sense, of lesser importance, since most of its relevant properties can be deduced from that of the constructible t-structure by dualising. However, it will be important for us, since it is the analogue of the heart of the \emph{dual constructible} t-structure in which our functor $\mathrm{sp}_{!}$ will land.\footnote{This is the reason for the shift in the definition of $\mathrm{sp}_!$.}

Just as in the case of constructible complexes of $\ell$-adic sheaves, there are three natural t-structures on ${\bf D}^b_{\mathrm{hol}}(X,Y)$. 

\subsubsection{Holonomic t-structure} 

The first t-structure on ${\bf D}^b_{\mathrm{hol}}(X,Y)$ is the one most closely related to the ordinary t-structure on ${\bf D}^b(\mathscr{D}^\dagger_{\fr{P}\Q})$. That is, we take $Z\subset P$ to be a closed subscheme such that $X$ is closed in $P\setminus Z$, and set $\fr{U}=\fr{P}\setminus Z$. Then the holonomic t-structure on ${\bf D}^b_{\mathrm{hol}}(X,Y)$ is defined to be the pullback of the ordinary t-structure via the restriction functor
\[ {\bf D}^b_{\mathrm{hol}}(X,Y) \rightarrow {\bf D}^b(\mathscr{D}^\dagger_{\mathfrak{U}\Q}),\]
that is, an object $\mathcal{M}\in {\bf D}^b_{\mathrm{hol}}(X,Y)$ is considered to be in ${\bf D}^{\geq0}$ or ${\bf D}^{\leq 0}$ if and only if it is so in ${\bf D}^b(\mathscr{D}^\dagger_{\mathfrak{U}\Q})$. This is indeed a t-structure, and is independent of the choice of $\mathfrak{U}$ \cite[\S1.2]{AC18a}. We denote by ${\bf D}^{\geq q},{\bf D}^{\leq q}$ the full subcategories of objects concentrated in degrees $\geq q$ and $\leq q$ respectvely, and by $\tau^{\geq q},\tau^{\leq q}$ the truncation functors. The holonomic t-structure is self-dual \cite[Proposition 1.3.1]{AC18a}, in the sense that 
\begin{align*}
 \mathcal{M}\in {\bf D}^{\geq0} &\Leftrightarrow \mathbf{D}_{X}\mathcal{M}\in  {\bf D}^{\leq0}.
\end{align*}
We denote the heart of the holonomic t-structure by $\mathrm{Hol}(X,Y)$, and refer to its objects as holonomic modules on $(X,Y)$. Cohomology functors will be denoted by
\[ \mathcal{H}^q: {\bf D}^b_{\mathrm{hol}}(X,Y) \rightarrow \mathrm{Hol}(X,Y). \]

\subsubsection{Constructible t-structure} The constructible t-structure on ${\bf D}^b_{\mathrm{hol}}(X,Y)$ is analogous to the t-structure on the derived category ${\bf D}^b_\mathrm{rh}(\mathscr{D}_X)$ of regular holonomic $\mathscr{D}$-modules on a smooth complex variety $X$ induced by the ordinary t-structure on ${\bf D}^b_c(X^\mathrm{an},\C)$ via the (covariant) Riemann--Hilbert correspondence
\[ \mathcal{M}\mapsto  \omega_{X^\mathrm{an}/\C} \otimes^{\mathbb{L}}_{\mathscr{D}_{X^\mathrm{an}}}\mathcal{M}^\mathrm{an} \cong \Omega^\bullet_{X^\mathrm{an}} \otimes_{\mathcal{O}_{X^\mathrm{an}}} \mathcal{M}^{\mathrm{an}}[\dim X].\]
Concretely, if $\mathcal{M}\in \mathrm{Hol}(X,Y)$, define the support $\mathrm{Supp}(\mathcal{M})$ to be the smallest closed subscheme $Z\subset X$ such that $\mathcal{M}|_{(X\setminus Z,Y)}=0$. Then define a pair of subcategories $({}^c{\bf D}^{\geq0},{}^c{\bf D}^{\leq0})$ of ${\bf D}^b_{\mathrm{hol}}(X,Y)$ as follows:
\begin{align*}
\mathcal{M} \in {}^\mathrm{c}{\bf D}^{\geq0} &\Leftrightarrow  \dim \mathrm{Supp} \mathcal{H}^n(\mathcal{M}) \leq n\;\;\forall n\geq 0, \;\;\mathcal{H}^n(\mathcal{M})  = 0 \;\;\forall n < 0 \\
\mathcal{M} \in {}^\mathrm{c}{\bf D}^{\leq0} &\Leftrightarrow \mathcal{H}^{n}( i^+\mathcal{M}) =0 \text{ for any closed immersion }i:(X',Y')\hookrightarrow (X,Y)\text{ with }\dim X' < n.
\end{align*}
Here $\mathcal{H}^n$ refer to the cohomology functors with values in $\mathrm{Hol}(X,Y)$ and $\mathrm{Hol}(X',Y')$ respectively, and $i$ being a closed immersion means that both $X'\hookrightarrow X$ and $Y'\hookrightarrow Y$ are closed immersions. That this does indeed define a t-structure is proved identically to \cite[Proposition 1.3.3]{Abe18a}, as are all of the various t-exactness properties of pullback and pushforward functors. We denote the heart of this t-structure by $\mathrm{Con}(X,Y)$, and will call objects of $\mathrm{Con}(X,Y)$ constructible modules on $(X,Y)$.

\subsubsection{Dual constructible t-structure} Whereas the holonomic t-structure is self-dual, the constructible t-structure is \emph{not}. We may therefore define a third t-structure $({}^{\mathrm{dc}}{\bf D}^{\geq0} ,{}^{\mathrm{dc}}{\bf D}^{\leq0} )$ by setting
\begin{align*}
 \mathcal{M}\in {}^{\mathrm{dc}}{\bf D}^{\geq0} &\Leftrightarrow \mathbf{D}_{X}\mathcal{M}\in  {}^{c}{\bf D}^{\leq0}  \\
  \mathcal{M}\in {}^{\mathrm{dc}}{\bf D}^{\leq0} &\Leftrightarrow \mathbf{D}_{X}\mathcal{M}\in  {}^{c}{\bf D}^{\geq0}.
\end{align*}
Again, we denote the heart of this t-structure by $\mathrm{DCon}(X,Y)$, and will call objects of the category  $\mathrm{DCon}(X,Y)$ dual constructible modules on $(X,Y)$.

\subsubsection{}

Thus the duality functor $\mathbf{D}_{X}$ induces anti-equivalences
\begin{align*}
\mathrm{Hol}(X,Y)^\mathrm{op} &\isomto \mathrm{Hol}(X,Y) \\
\mathrm{Con}(X,Y)^\mathrm{op} &\isomto \mathrm{DCon}(X,Y) \\
\mathrm{DCon}(X,Y)^\mathrm{op} &\isomto \mathrm{Con}(X,Y).
\end{align*}
When $Y$ is proper over $k$, we will write $\mathrm{Hol}(X)$, $\mathrm{Con}(X)$ and $\mathrm{DCon}(X)$ instead. We will also use the notation $\mathrm{Hol}(\fr{P})$, $\mathrm{Con}(\fr{P})$ and $\mathrm{DCon}(\fr{P})$ instead of $\mathrm{Hol}(P,P)$, $\mathrm{Con}(P,P)$ and $\mathrm{DCon}(P,P)$. Once more we warn the reader that all objects of these categories are assumed to be of Frobenius type, although this is not explicit from the notation.

\subsubsection{}
Suppose that $(X,Y,\mathfrak{P})$ is an l.p.\ frame, and let $a:X\rightarrow P$ denote the given immersion from $X$ into the special fibre of $\fr{P}$. The functor
\[ a_+:{\bf D}^b_{\mathrm{hol}}(X,Y) \rightarrow {\bf D}^b_{\mathrm{hol}}(\mathfrak{P}) \]
is then induced by the \emph{identity} functor
\[ {\bf D}^b(\mathscr{D}^\dagger_{\fr{P}\Q}) \rightarrow {\bf D}^b(\mathscr{D}^\dagger_{\fr{P}\Q}). \]

\begin{lemma} The functor $a_+$ is t-exact for the dual constructible t-structure.
\end{lemma}

\begin{proof}
By duality, this is equivalent to showing that $a_!$ is exact for the constructible t-structure, which is proved identically to \cite[Lemma 1.3.2]{Abe18a}.
\end{proof}

It therefore induces a functor
\[ a_+: \mathrm{DCon}(X,Y) \rightarrow \mathrm{DCon}(\mathfrak{P})\]
which is fully faithful, and identifies $\mathrm{DCon}(X,Y)$ with the intersection of $\mathrm{DCon}(\mathfrak{P})$ and ${\bf D}^b_{\mathrm{hol}}(X,Y)$ inside ${\bf D}^b(\mathscr{D}^\dagger_{\fr{P}\Q})$. Note that the analogous claim is false for the holonomic or constructible t-structures.

\begin{lemma} \label{lemma: hol fin et} Let
\[ \xymatrix{ X'\ar[r] \ar[d]_f & Y' \ar[d]^g \\ X \ar[r] & Y } \] 
be a morphism of pairs over $k$, with $g$ proper, $f$ \'etale, and the square Cartesian. Then:
\begin{enumerate}
\item an object $\mathcal{M}\in {\bf D}^b_{\mathrm{hol}}(X',Y')$ is holonomic, constructible, or dual constructible if and only if $f_+\mathcal{M}\cong f_!\mathcal{M}\in {\bf D}^b_{\mathrm{hol}}(X,Y)$ is so;
\item an object $\mathcal{N}\in {\bf D}^b_{\mathrm{hol}}(X,Y)$ is holonomic, constructible, or dual constructible if and only if $f^+\mathcal{N}\cong f^!\mathcal{N}\in {\bf D}^b_{\mathrm{hol}}(X',Y')$ is so.
\end{enumerate}
\end{lemma}

\begin{proof}
Since $\mathcal{M}$ is a direct summand of $f^+f_+\mathcal{M}$, and $\mathcal{N}$ is a direct summand of $f_+f^+\mathcal{N}$, it suffices to prove that both $f^!\cong f^+$ and $f_!\cong f_+$ are t-exact for all three t-structures. For the holonomic t-structure, this follows from \cite[Proposition 1.3.2 and Proposition 1.3.13]{AC18a}. For the constructible t-structure, it follows from \cite[Lemmas 1.3.4 and 1.3.8]{Abe18a}, translated to apply to pairs. The case of the dual constructible t-structure then follows from the case of the constructible t-structure by duality.
\end{proof}

\subsection{Overholonomicity and dual constructibility}

We can now prove the required finiteness properties of $\mathrm{sp}_!$, by reducing to the fundamental results of Caro--Tsuzuki \cite{CT12}. Recall from Definitions \ref{defn: support cons} and \ref{defn: support Dmod} the notion of \emph{support} for constructible isocrystals and overholonomic $\mathscr{D}^\dagger$-modules. 

\begin{theorem} \label{theo: DCon} Let $\fr{P}$ be a realisable smooth formal scheme over $\mathcal{V}$, and $\mathscr{F}\in \mathrm{Isoc}_{\mathrm{cons}}(\mathfrak{P})$ a constructible isocrystal of Frobenius type. Then
\[ \mathrm{sp}_! \mathscr{F}\in \mathrm{DCon}(\mathfrak{P}) \subset {\bf D}^b(\mathscr{D}^\dagger_{\fr{P}\Q}) \]
is overholonomic and dual constructible. If $\mathscr{F}$ is supported on some locally closed subscheme $X\hookrightarrow P$, then so is $\mathrm{sp}_!\mathscr{F}$.
\end{theorem}

We begin with the following straightforward observation. 

\begin{lemma} \label{lemma: triangle sp} Let $0\rightarrow \mathscr{F}\rightarrow \mathscr{G}\rightarrow \mathscr{H} \rightarrow 0$ be an exact sequence of constructible isocrystals on $\fr{P}$. Then the exact triangle
\[ \mathbf{R}\mathrm{sp}_*\mathscr{F}\rightarrow \mathbf{R}\mathrm{sp}_*\mathscr{G}\rightarrow \mathbf{R}\mathrm{sp}_*\mathscr{H}\overset{+1}{\rightarrow} \]
in ${\bf D}^b(\mathscr{D}_{\fr{P}\Q})$
lifts canonically to an exact triangle
\[ \mathrm{sp}_!\mathscr{F}\rightarrow \mathrm{sp}_!\mathscr{G}\rightarrow \mathrm{sp}_!\mathscr{H}\overset{+1}{\rightarrow}\]
in ${\bf D}^b(\mathscr{D}^\dagger_{\fr{P}\Q})$.
\end{lemma}

\begin{proof}
This simply follows by taking a simultaneous good stratification for  $\mathscr{F},\mathscr{G},\mathscr{H}$ and applying the explicit construction of $\mathrm{sp}_!$ given in \S\ref{sec: pre Dd}.
\end{proof}

Hence by, \cite[Proposition 3.3]{Car09b}, it suffices to prove the following special case of Theorem \ref{theo: DCon}.

\begin{theorem}  \label{theo: DCon2} Let $\fr{P}$ be a realisable smooth formal scheme over $\mathcal{V}$, $i:Y\hookrightarrow P$ a closed immersion, and $j:X\hookrightarrow Y$ a strongly affine open immersion with $X$ smooth. Let $E\in \mathrm{Isoc}(X,Y)$ with realisation $E_\mathfrak{P}$ on $\tube{X}_\mathfrak{P}$, and set $\mathscr{F}=j_*E_\fr{P}$. Then $\mathrm{sp}_!i_{!}\mathscr{F}$ is overholonomic, dual constructible, and supported on $X$.
\end{theorem}

The proof of Theorem \ref{theo: DCon2} will proceed in several stages.

\medskip
\noindent \it{Step 1:} \rm{Theorem \ref{theo: DCon2} holds if $Y \isomto P$.}

\begin{proof}
In this case, we take a good stratification to be $(Y\setminus X,X)$. We therefore have $\mathrm{sp}_!i_{!}\mathscr{F}= \mathrm{sp}_*\mathscr{F}[d_\fr{P}]$, 
where $d_\fr{P}$ is the dimension of $\fr{P}$. That this is overholonomic now follows from \cite[Theorem 2.3.15]{CT12}, and that it is supported on $X$ from the fact that $\mathrm{sp}_*\mathscr{F}$ is a coherent $\mathcal{O}_{\fr{P}\Q}(^\dagger Y\setminus X)$-module. To see that it is dual constructible it suffices by \cite[Proposition 4.3.1]{Car09a} to show that $\mathrm{sp}_*\mathscr{F}^\vee[-d_\fr{P}]$ is constructible, which follows from a direct calculation using \cite[Proposition 4.1.8]{Car09a}.
\end{proof}

\noindent \it{Step 2:} \rm{In the situation of Theorem \ref{theo: DCon2}, assume $Y$ is smooth,  admits a smooth lifting $\fr{Y}$, and $\fr{P} = \widehat{\A}^d_\fr{Y}$, with $i:Y \hookrightarrow \fr{P}$ embedded via the zero section. Then Theorem \ref{theo: DCon2} holds.}

\begin{proof}
Let $\pi:\fr{P}=\widehat{\A}^d_\fr{Y}\rightarrow \fr{Y}$ the projection, and let $x_1,\ldots,x_d$ be co-ordinates on $\widehat{\A}^d_\fr{Y}$. We may consider $\pi^*E\in \mathrm{Isoc}(\A^d_X,\A^d_Y)$, with realisation $(\pi^*E)_\fr{P}$ on the tube of $j:\A^d_{X}\hookrightarrow P=\A^d_Y$. Write $\mathscr{G}=j_*(\pi^*E)_\fr{P}$. Thus
\[ 0\rightarrow i_{!}\mathscr{F} \rightarrow \mathscr{G} \rightarrow \bigoplus_i j_{D(x_i)}^\dagger \mathscr{G} \rightarrow \ldots \rightarrow j_{D(x_1)\cap\ldots \cap D(x_d)}^\dagger \mathscr{G} \rightarrow 0  \]
is an exact sequence of constructible isocrystals on $\mathfrak{P}$. By Step 1 above if we apply $\mathrm{sp}_!$ to any of the terms except the first, we obtain an overholonomic and dual constructible complex of $\mathscr{D}^\dagger_{\mathfrak{P}\Q}$-modules. Hence by Lemma \ref{lemma: triangle sp} and \cite[Proposition 3.3]{Car09b} we deduce that $\mathrm{sp}_!i_{!}\mathscr{F}$ is overholonomic. 

To see that $\mathrm{sp}_!i_{!}\mathscr{F}$ has the right support, we first note that its restriction to $\fr{P}\setminus Y$ vanishes. Since it is overholonomic, and thus in particular coherent, it follows that it is supported on $Y$ in the sense of arithmetic $\mathscr{D}$-modules. Finally, if we apply $\mathrm{sp}_!$ to any of the other terms in the above exact sequence, we get something supported on $\A^d_X$. Hence $\mathrm{sp}_!i_!\mathscr{F}$ is supported on $\A^d_X\cap Y=X$.

Finally, to see that $\mathrm{sp}_!i_{!}\mathscr{F}$ is dual constructible, we consider the complex
\[  0\rightarrow \mathrm{sp}_!\mathscr{G} \rightarrow  \bigoplus_i \mathrm{sp}_!j_{D(x_i)}^\dagger \mathscr{G} \rightarrow \ldots \rightarrow  \mathrm{sp}_!j_{D(x_1)\cap\ldots \cap D(x_d)}^\dagger \mathscr{G} \rightarrow 0 \]
of dual constructible modules on $\fr{P}$, starting in degree $0$. If we can show that this only has cohomology in degree $0$ (as a complex in the abelian category $\mathrm{DCon}(\fr{P})$), then this degree zero cohomology group will be precisely $\mathrm{sp}_!i_!\mathscr{F}$, which will therefore be dual constructible.

We let $\mathcal{M}=\mathrm{sp}_!\mathscr{G}\in \mathrm{DCon}(\A^d_X,\A^d_Y)$. Then by repeatedly combining the support claim from Step 1 with Lemma \ref{lemma: triangle sp}, we can identify our complex with the complex
\[ 0 \rightarrow \mathcal{M} \rightarrow \bigoplus_i j_{i+}j_{i}^+\mathcal{M} \rightarrow \ldots \rightarrow j_{1\ldots n+}j_{1\ldots n}^+\mathcal{M}\rightarrow 0 \]
where $j_{i_1\ldots i_l}:D(x_{i_1})\cap \ldots \cap D(x_{i_l})\hookrightarrow \A^d_X$ is the natural open immersion (note that $j_{i_1\ldots i_l+}$ and $j_{i_1\ldots i_l}^+$ are exact for dual constructible modules). We want to show that this complex is exact in higher degrees, which can be checked after applying the (exact) functor $i^!_x$ for the inclusion $i_x$ of any closed point of $\A^d_X$. But now this just boils down to a straightforward combinatorial calculation.
\end{proof}

\medskip
\noindent \it{Step 3:} \rm{Theorem \ref{theo: DCon2} holds if $Y$ is smooth.}

\begin{proof}
The question being local on $\fr{P}$, we may assume that $\fr{P}$ is affine, and admits \'etale co-ordinates $x_1,\ldots,x_d$. Hence, by combining \cite[Proposition 5.2]{Ach17} with formal smoothness of $\fr{P}$, we may suppose that there exists a finite \'etale map $u:\fr{P}\rightarrow \widehat{\A}^d_\mathcal{V}$. Since $\fr{P}$ is affine, so is $Y$, and hence there exists a smooth formal scheme $\fr{Y}$ lifting $Y$. Again, by formal smoothness, we may lift the closed immersion $i:Y\hookrightarrow P$ to a closed immersion $i:\fr{Y}\hookrightarrow \fr{P}$ of smooth formal schemes over $\mathcal{V}$. We now look at the following diagram.
\[ \xymatrix{ & & & & \fr{Y}\times_\mathcal{V} \fr{P} \ar[d]^-{\mathrm{id}\times u} \ar[r]^-{p_2} & \fr{P} \\ X \ar[r]^-{j} & Y \ar[r] & \fr{Y} \ar[urr]^-{\mathrm{id}\times i} \ar[rr]_-{\mathrm{id} \times ui} & & \fr{Y}\times_{\mathcal{V}} \widehat{\A}^d_\mathcal{V}   }\]
By Step 2 we know that $\mathrm{sp}_{\fr{Y}\times \widehat{\A}^d_\mathcal{V}!}(\mathrm{id}\times ui)_{!}j_{*}E_{\fr{Y}\times \widehat{\A}^d_\mathcal{V}}$ is overholonomic, dual constructible, and supported on $X$. Since $\mathrm{id}\times u:\fr{Y}\times_\mathcal{V} \fr{P}\rightarrow \fr{Y}\times_\mathcal{V} \widehat{\A}^d_\mathcal{V}$ is finite \'etale, we see from Proposition \ref{prop: comm finite etale} that
\[ \mathrm{sp}_{\fr{Y}\times \widehat{\A}^d_\mathcal{V}!}(\mathrm{id}\times ui)_{!}j_{*}E_{\fr{Y}\times \widehat{\A}^d_\mathcal{V}} \cong (\mathrm{id}\times u)_+\mathrm{sp}_{\fr{Y}\times \fr{P}!}(\mathrm{id}\times i)_{!}j_{*}E_{\fr{Y}\times \fr{P}},\]
and hence by Lemma \ref{lemma: hol fin et}, together with the fact that finite \'etale maps reflect overholonomicity, we can see that $\mathrm{sp}_{\fr{Y}\times \fr{P}!}(\mathrm{id}\times i)_{!}j_{*}E_{\fr{Y}\times \fr{P}}$ is overholonomic, dual constructible, and supported on $X$.

We can now further localise on $\mathfrak{P}$ to assume that there exists a finite \'etale map $v:\fr{Y}\rightarrow \widehat{\A}^e_\mathcal{V}$, and we therefore deduce again that
\[ \mathrm{sp}_{\widehat{\A}^e_\mathcal{V}\times \fr{P}!}(v\times i)_{!}j_{*}E_{\widehat{\A}^e_\mathcal{V}\times \fr{P}} \cong (v\times\mathrm{id})_+\mathrm{sp}_{\mathfrak{Y} \times \fr{P}!}(\mathrm{id}\times i)_{!}j_{*}E_{\mathfrak{Y}\times \fr{P}}\]
is overholonomic, dual constructible, and supported on $X$. By embedding $\widehat{\A}^e_\mathcal{V}$ inside $\widehat{\P}^e_\mathcal{V}$, and using the fact that $\mathrm{sp}_{\widehat{\P}^e_\mathcal{V}\times \fr{P}!}(v\times i)_!j_*E_{\widehat{\P}^e_\mathcal{V}\times \fr{P}}$ is clearly supported on $Y$ (in the na\"ive sheaf-theoretic sense) we deduce that $\mathrm{sp}_{\widehat{\P}^e_\mathcal{V}\times \fr{P}!}(v\times i)_!j_*E_{\widehat{\P}^e_\mathcal{V}\times \fr{P}}$ is overholonomic, and supported on $X$ in the sense of arithmetic $\mathscr{D}$-modules. Finally applying Theorem \ref{theo: tr conc fin et} to the frame
\[ \xymatrix{ & & \widehat{\P}^e_\mathcal{V} \times_\mathcal{V} \fr{P} \ar[d] \\ X \ar[r] & Y\ar[r]\ar[ur] & \fr{P} } \]
we can conclude.
\end{proof}

\medskip
\noindent \it{Step 4:} \rm{Theorem \ref{theo: DCon2} holds if there exists a morphism of pairs
\[ \xymatrix{ X' \ar[r]^{j'}\ar[d]_f & Y' \ar[d]^g \\ X\ar[r]^{j} & Y } \]
with $f$ \'etale, $g$ projective, $Y'$ smooth, and the square Cartesian.}

\begin{proof}
Note that $f$ is finite \'etale. Since $g$ is projective, we may extend the above square to a morphism of frames
\[\xymatrix{  X'  \ar[r] \ar[d]^f & Y' \ar[r]^{i'} \ar[d]^g & \fr{P}' \ar[d]^u \\   X \ar[r] & Y \ar[r]^{i}   & \fr{P}, } 
\]
with $\mathfrak{P}'\cong  \fr{P}\times_\mathcal{V} \widehat{\P}^d_\mathcal{V}$ and $u$ the first projection. By Noetherian induction on $X$, we can pass to a dense open subset of $X$, and therefore we can assume that we are in the situation of Setup \ref{setup: conc tr}. We therefore know by Step 3 that $\mathrm{sp}_{\fr{P}'!}i'_!\tube{g}^*\!\!\mathscr{F}$ is overholonomic on $\fr{P}'$, dual constructible, and supported on $X$, and by Theorem \ref{theo: tr conc fin et} that
\[ u_+\mathrm{sp}_{\fr{P}'!}i'_!\tube{g}^*\!\!\mathscr{F} \isomto \mathrm{sp}_{\fr{P}!}i_{!}\tube{g}_*\!\tube{g}^*\!\!\mathscr{F}. \]
Hence $\mathrm{sp}_{\fr{P}!}i_!\tube{g}_*\!\tube{g}^*\!\!\mathscr{F}$ is overholonomic and supported on $X$, again by \cite[Th\'eor\`eme 2.10]{Car09b}; it is dual constructible by Lemma \ref{lemma: hol fin et}. Finally, since $E$ is a direct summand of $f_*f^*E$, we see that $\mathrm{sp}_{\fr{P}!}i_{!}\mathscr{F}$ is a direct summand of $\mathrm{sp}_{\fr{P}!}i_{!}\tube{g}_*\!\tube{g}^*\!\!\mathscr{F}$, and the claim follows.
\end{proof}

\medskip
\noindent \it{Step 5:} \rm{Theorem \ref{theo: DCon2} holds in general.}

\begin{proof} Via Noetherian induction on $X$, we may use Lemma \ref{lemma: triangle sp} and de Jong's theorem on alterations \cite[Theorem 4.1]{dJ96} to reduce to Step 4 above.
\end{proof}

\subsection{Functoriality}

We can also now show some important functoriality properties of $\mathrm{sp}_{!}$, starting with a simple consequence of the localisation exact triangle Lemma \ref{lemma: triangle sp}.  

\begin{lemma} \label{lemma: support} Let $\mathfrak{P}$ be a realisable smooth formal scheme over $\mathcal{V}$, $i_Z:Z\hookrightarrow P$ a locally closed immersion, and $\mathscr{F}\in \mathrm{Isoc}_{\mathrm{cons}}(\mathfrak{P})$. Then there is a canonical isomorphism
\[ \mathrm{sp}_{!}i_{Z!}i_Z^{-1} \mathscr{F} \isomto \mathbf{R}\underline{\Gamma}^\dagger_Z \mathrm{sp}_{!}\mathscr{F}  \]
in $\mathrm{DCon}(\mathfrak{P})$. 
\end{lemma} 

\begin{proof}
We may assume that $i_Z:Z\hookrightarrow X$ is either a closed immersion or an open immersion. If it is a closed immersion, then this follows from the exact sequence
\[ 0\rightarrow i_{Z!}i_Z^{-1}\mathscr{F} \rightarrow \mathscr{F}\rightarrow j^\dagger_{P\setminus Z}\mathscr{F}\rightarrow0 \]
combined with Lemma \ref{lemma: triangle sp} and the part of Theorem \ref{theo: DCon} concerning the support of $ \mathrm{sp}_{!}\mathscr{F}$. The case of an open immersion is entirely similar.
\end{proof}

\begin{corollary} \label{cor: funct RSp} Let $u:\mathfrak{P}'\rightarrow \mathfrak{P}$ be a morphism of realisable smooth formal $\mathcal{V}$-schemes. Then there is a canonical isomorphism
\[ u^!\mathrm{sp}_{\mathfrak{P}!}\mathscr{F} \isomto \mathrm{sp}_{\mathfrak{P}'!}u^*\mathscr{F}\]
in $\mathrm{DCon}(\mathfrak{P}')$.
\end{corollary}

\begin{proof} Since $\fr{P}'$ admits a locally closed embedding into a smooth and proper formal $\mathcal{V}$-scheme, we may use Lemma \ref{lemma: support} to reduce to the case when $\mathfrak{P}'\cong \mathfrak{P}\times_\mathcal{V}\mathfrak{Q}$, with $\mathfrak{Q}$ smooth and proper over $\mathcal{V}$, and $u$ the first projection. In this case, Theorem \ref{theo: trace Ddag} together with the adjunction between $u^!$ and $u_+$ gives a canonical map
\[  u^!\mathrm{sp}_{\mathfrak{P}!}\mathscr{F} \rightarrow \mathrm{sp}_{\mathfrak{P}'!}u^*\mathscr{F}. \]
To check that it is an isomorphism, by applying \cite[Lemma 1.3.11]{AC18a} it suffices to do so after taking the extraordinary pullback along any section $\spf{\mathcal{V}'}\rightarrow \fr{P}'$ with $\mathcal{V}'$ the ring of integers in a finite extension of $K$. By transitivity of pullbacks, and after base changing to $\mathcal{V}'$, this then reduces to the already considered case when $u$ is a closed immersion.
\end{proof}

\subsection{Dual constructible modules on pairs and varieties}

Given an l.p.\ frame $(X,Y,\fr{P})$ over $\mathcal{V}$, we therefore have a functor
\[ \mathrm{sp}_{\fr{P}!}\colon\mathrm{Isoc}_\mathrm{cons}(X,Y)\rightarrow \mathrm{DCon}(X,Y)\]
which could, \emph{a priori}, depend on the choice of formal embedding of $(X,Y)$. To see that it does not, we prove the:

\begin{lemma} \label{cor: independence of P} Let
\[ \xymatrix{   & Y' \ar[r]^-{i'} \ar[d]^g & \fr{P}' \ar[d]^u \\   X \ar[ur]^-{j'} \ar[r]^-j & Y \ar[r]^-i & \fr{P} } 
\]
be a morphism of l.p.\ frames, with $u$ smooth and $g$ proper. Then, for all $E\in \mathrm{Isoc}_{\mathrm{cons}}(X,Y')$ there exists a canonical isomorphism
\[ u_+\mathrm{sp}_{\fr{P}'!}i'_{!}j'_*E_{\mathfrak{P}'} \isomto \mathrm{sp}_{\fr{P}!}i_{!}j_*E_\mathfrak{P} \]
in ${\bf D}^b_{\mathrm{hol}}(X,Y)$. 
\end{lemma}

\begin{proof} Follows from combining Lemma \ref{lemma: support}, Corollary \ref{cor: funct RSp} and \cite[Lemme 2.5]{Car12}.
\end{proof}

For constructible isocrystals of Frobenius type, we therefore get well-defined functors
\begin{align*}
 \mathrm{sp}_! : \mathrm{Isoc}_\mathrm{cons}(X,Y) &\rightarrow \mathrm{DCon}(X,Y) \\
 \mathrm{sp}_! : \mathrm{Isoc}_\mathrm{cons}^\dagger(X) &\rightarrow \mathrm{DCon}(X)
\end{align*}
for realisable couples $(X,Y)$ and realisable varieties $X$ respectively. Since the \emph{abelian} categories $\mathrm{Isoc}_\mathrm{cons}(X,Y)$ and $\mathrm{DCon}(X,Y)$ satisfy descent under Zariski open covers, we can extend their definition to cover not-necessarily realisable couples $(X,Y)$, and using Corollary \ref{cor: funct RSp} we can then also extend the functor $\mathrm{sp}_!$ to the non-realisable case. It follows from Lemma \ref{lemma: support} and Corollary \ref{cor: funct RSp} that whenever $(f,g):(X',Y')\rightarrow (X,Y)$ is a morphism of couples, the diagram
\[ \xymatrix{ \mathrm{Isoc}_\mathrm{cons}(X,Y) \ar[r]^{\mathrm{sp}_!}\ar[d]_{f^*} & \mathrm{DCon}(X,Y) \ar[d]^{f^!} \\  \mathrm{Isoc}_\mathrm{cons}(X',Y') \ar[r]^{\mathrm{sp}_!} &\mathrm{DCon}(X',Y') } \]
commutates up to natural isomorphism, and that whenever $Y'\hookrightarrow Y$ is a closed immersion, the same is true of the diagram
\[ \xymatrix{ \mathrm{Isoc}_\mathrm{cons}(X',Y') \ar[r]^{\mathrm{sp}_!}\ar[d]_{f_!} & \mathrm{DCon}(X',Y') \ar[d]^{f_+} \\  \mathrm{Isoc}_\mathrm{cons}(X,Y)\ar[r]^{\mathrm{sp}_!} &\mathrm{DCon}(X,Y). } \]
There are of course similar statements for varieties in place of pairs. We then have the following precise formulation of Le Stum's conjecture on the $p$-adic Deligne--Kashiwara equivalence.

\begin{conjecture} For any pair $(X,Y)$ over $k$ the functor $\mathrm{sp}_!\colon \mathrm{Isoc}_\mathrm{cons}(X,Y)\rightarrow \mathrm{DCon}(X,Y)$ is an equivalence of categories.
\end{conjecture}

\begin{remark}
\begin{enumerate}
\item It suffices to show that the functor
\[\mathrm{sp}_!\colon \mathrm{Isoc}_\mathrm{cons}(\fr{P})\rightarrow \mathrm{DCon}(\fr{P}) \]
is an equivalence, for any smooth formal $\mathcal{V}$-scheme $\fr{P}$.
\item Our condition on constructible isocrystals being `of Frobenius type' is stronger than that of the same name defined in \cite{LS14} in the case of curves. We did not try especially hard to formulate the least restrictive condition to put on constructible isocrystals in order to reasonably conjecture that $\mathrm{sp}_!$ is an equivalence. Some form of non--Liouville exponents condition would be a natural candidate.
\end{enumerate}
\end{remark}

\section{The case of overconvergent isocrystals} \label{sec: comp}

We can now prove our main comparison theorems for overconvegent isocrystals: our functor $\mathrm{sp}_!$ is the dual of Caro's functor $\mathrm{sp}_+$, and induces an isomorphism between the compactly supported cohomolgy of an overconvergent isocrystal and its associated arithmetic $\mathscr{D}^\dagger$-module.

\subsection{Recovering $\bm{\mathrm{sp}_+}$} As a special case of Caro's construction in \cite{Car11}, there exists, for $X$ smooth and realisable over $k$, a fully faithful functor
\[ \widetilde{\mathrm{sp}}_+: \mathrm{Isoc}^\dagger(X) \rightarrow \mathrm{Hol}(X). \]
Note that in \emph{loc. cit.} this functor is simply denoted $\mathrm{sp}_+$, but we will want to reserve this instead for the shifted functor
\[ \mathrm{sp}_+:= \widetilde{\mathrm{sp}}_+[-\dim X] \]
which lands in the category $\mathrm{Con}(X)$ of constructible modules on $X$, see for example \cite[Example 1.3.1]{Abe18a}. Since the category $\mathrm{Con}(X)$ is of a Zariski local nature on $X$, we can extend the defintion to not-necessarily realisable varieties, and it is then explained in \cite[\S3]{Abe19} how to use descent to extend $\mathrm{sp}_+$ to a fully faithful functor
\[ \mathrm{sp}_+ : \mathrm{Isoc}^\dagger(X) \rightarrow \mathrm{Con}(X) \]
for an arbitrary variety $X$, not necessarily smooth or realisable (note that in \cite{Abe19} this functor is denoted $\rho$). In fact, both of these constructions work for pairs, and provide a functor
\[ \mathrm{sp}_+ : \mathrm{Isoc}(X,Y) \rightarrow \mathrm{Con}(X,Y), \]
defined for an arbitrary pair $(X,Y)$ over $k$.

\begin{theorem} \label{theo: sp_+ comp}
Let $(X,Y)$ be a pair over $k$, and $E\in \mathrm{Isoc}(X,Y)$ a partially overconvergent isocrystal of Frobenius type. Then there exists a canonical isomorphism
\[ \mathrm{sp}_+E^\vee \cong \mathbf{D}_{X}\mathrm{sp}_!E \]
in $\mathrm{Con}(X,Y)$.
\end{theorem}

\begin{proof}
Note that the RHS commutes with pullback, by Corollary \ref{cor: funct RSp}, and the functor $\mathrm{sp}_+$ appearing on the LHS is \emph{defined} using h-descent. The claim is therefore local for the h-topology on $Y$. We may therefore assume that $Y$ is smooth, affine, and irreducible.

The claim is also local on $X$, which we may assume to be similarly smooth and affine, in particular $Y\setminus X$ is a divisor. We know that $Y$ lifts to a smooth formal scheme $\fr{Y}$ over $\mathcal{V}$, in which case $\mathrm{sp}_!=\widetilde{\mathrm{sp}}_+[\dim X]$, and the result follows from \cite[Proposition 4.3.1]{Car09a}.
\end{proof}

\begin{remark} More generally, the proof shows that if $F:\mathrm{Isoc}(X,Y)\rightarrow \mathrm{Con}(X,Y)$ is any collection of functors such that:
\begin{enumerate}
\item $F$ commutes with pullback;
\item $F\cong\mathrm{sp}_+$ whenever $Y$ is smooth and affine, and $Y\setminus X$ is a divisor,
\end{enumerate}
then $F\cong \mathrm{sp}_+$ in general.
\end{remark}

\subsection{Compactly supported cohomologies} Finally, we can show that compactly supported $\mathscr{D}^\dagger$-module cohomology coincides with compactly supported rigid cohomology, at least for realisable varieties. 

\begin{theorem} \label{theo: cohom comp 1}
Let $X$ be a realisable $k$-variety, $f:X\rightarrow\spec{k}$ the structure morphism, and $E\in \mathrm{Isoc}^\dagger(X)$ an overconvergent isocrystal of Frobenius type. Then there is a canonical isomorphism
\[ {\rm H}^*_{c,\rig}(X,E) \cong {\rm H}^*(f_!\mathrm{sp}_+E) \]
of finite dimensional, graded $K$-vector spaces.
\end{theorem}

\begin{proof}
Choose a frame $(X\overset{j}{\hookrightarrow} Y \overset{i}{\hookrightarrow} \fr{P})$ with $\fr{P}$ smooth and proper over $\mathcal{V}$. Let $E^\vee_\fr{P}$ be the realisation of $E^|vee$ on $\tube{X}_\fr{P}$, and set $\mathscr{F}^\vee=j_*E^\vee_\fr{P}$. Then we have isomorphisms
\begin{align*}
{\rm H}^*_{c,\rig}(X,E)  &\cong  {\rm H}_*^{\mathrm{BM},\rig}(X,E^\vee)^\vee \\
&\cong {\rm H}^{2\dim \fr{P}-*}(\fr{P}_K, \Omega^\bullet_{\fr{P}_K}\otimes i_{!}\mathscr{F}^\vee )^\vee  \\
&\cong {\rm H}^{\dim \fr{P}-*}(\fr{P}, \Omega^\bullet_{\fr{P}/\mathcal{V}} \otimes \mathrm{sp}_{\fr{P}!}i_{!}\mathscr{F}^\vee )^\vee \\
&\cong  {\rm H}^{-*}(f_+\mathrm{sp}_!E^\vee)^\vee \\
&\cong {\rm H}^*(f_!\mathrm{sp}_+E ) 
\end{align*}
where the first isomorphism is Proposition \ref{prop: BM}, the second is Definition \ref{defn: BM}, the third follows from the isomorphism
\[ \mathbf{R}\mathrm{sp}_*\left(   \Omega^\bullet_{\fr{P}_K} \otimes i_{!}\mathscr{F}^\vee\right) \cong  \Omega^\bullet_{\fr{P}/\mathcal{V}}\otimes \mathbf{R}\mathrm{sp}_{*}\left(  i_{!}\mathscr{F}^\vee\right) \]
of complexes on $\fr{P}$, the fourth follows from the overconvergent Spencer resolution \cite[(4.3.6.3)]{Ber02}, and finally the fifth follows from Theorem \ref{theo: sp_+ comp}.
\end{proof}

\begin{corollary} \label{cor: cohom comp 2} Let $X$ be a realisable $k$-variety, which is either smooth \emph{or} proper, $f:X\rightarrow\spec{k}$ the structure morphism, and $E\in \mathrm{Isoc}^\dagger(X)$ an overconvergent isocrystal of Frobenius type. Then there is a canonical isomorphism
\[ {\rm H}^*_{\rig}(X,E) \cong {\rm H}^*(f_+\mathrm{sp}_+E) \]
of finite dimensional, graded $K$-vector spaces.
\end{corollary}

\begin{proof} In the proper case there is nothing to prove, in the smooth case this follows from Theorem \ref{theo: cohom comp 1} and Poincar\'e duality.
\end{proof} 

\bibliographystyle{../Templates/bibsty}
\bibliography{../Templates/lib.bib}

\end{document}